\documentclass[10pt,reqno]{amsart}
\usepackage{amsmath,amsxtra,amssymb,amsthm,amsfonts,bm,cite}

\usepackage[cal=cm]{mathalfa}
\usepackage{mathrsfs}
\usepackage[usenames,dvipsnames]{xcolor}

\usepackage{bbold}
\usepackage{bbm}
\usepackage{nonfloat}
\usepackage{braket}
\usepackage{dsfont}
\usepackage{mathdots}
\usepackage{mathtools}
\usepackage[shortlabels]{enumitem}
\usepackage{csquotes}
\usepackage{stmaryrd}
\usepackage{graphicx}
\usepackage{stackengine}
\usepackage{scalerel}
\usepackage{array}
\usepackage{makecell}
\newcolumntype{x}[1]{>{\centering\arraybackslash}p{#1}}
\usepackage{tikz}
\usetikzlibrary{shapes.geometric, shapes.misc, positioning, arrows, decorations.pathreplacing, angles, quotes}
\usepackage{booktabs}
\usepackage{xfrac}
\usepackage{siunitx}
\usepackage{centernot}
\usepackage{comment}
\usepackage{chngcntr}
\usepackage{extarrows}

\usepackage{tikz-cd} 

\makeatletter
\def\thmhead@plain#1#2#3{%
  \thmname{#1}\thmnumber{\@ifnotempty{#1}{ }\@upn{#2}}%
  \thmnote{ {\the\thm@notefont#3}}}
\let\thmhead\thmhead@plain
\makeatother

\usepackage[hidelinks]{hyperref}
\hypersetup{colorlinks = true,
linkcolor = MidnightBlue, anchorcolor =red,
citecolor = ForestGreen,
filecolor = red,urlcolor = red,
            pdfauthor=author}
\usepackage{cleveref}

\newtheorem{theorem}{Theorem}[section]
\newtheorem{lemma}[theorem]{Lemma}
\newtheorem{proposition}[theorem]{Proposition}
\newtheorem{corollary}[theorem]{Corollary}

\numberwithin{equation}{section}
\newtheorem{example}{Example}[section]

\theoremstyle{remark}
\newtheorem{remark}[theorem]{Remark}

\theoremstyle{definition}
\newtheorem{definition}[theorem]{Definition}

\newcommand{\norm}[1]{\left\lVert #1 \right\rVert}

\def \sss{\scriptscriptstyle}

\newcommand{\mb}[1]{{\color{black} #1}}

\newcommand{\mc}[1]{\mathcal{#1}}
\newcommand{\mf}[1]{\mathfrak{#1}}

\newcommand{\w}[1]{\widetilde{#1}}
\newcommand{\h}[1]{\hat{#1}}

\def\eqq{\Longleftrightarrow}
\def\q {\quad}

\def \l{\langle}
\def \r{\rangle}

\def\bb{\begin{equation}
  \left\{\ 
   \begin{aligned} }
\def\ee{   \end{aligned}
  \right.
  \end{equation}}

\def\mm{ \left[
 \begin{matrix}}
\def\nn{\end{matrix} \right] } 

\def \lad {\lambda}
\def \Lad {\Lambda}
\def \d{\delta}
\def \ep {\varepsilon}
\def \vp {\varphi}

\def \na {\nabla}
\def \ddiv {{\rm div}}

\def \ww {\omega}

\def \R{\mathbb{R}}
\def \C{\mathbb{C}}

\def \p {\partial}

\def \dd {\cdot}
\def \cd {\cdots}
\def \t {\times}

\def \ck {\chi_\kappa^2}
\def \ka {\kappa}
\def \si{\sigma}

\newcommand{\pure}[1]{\ket{#1}\bra{#1}}

\def \gs {\Gamma_\si}

\def \bh {\mathcal{B}(\mc{H})}
\def \dh {\mathcal{D}(\mc{H})}
\def \dhh {\mathcal{D}_+(\mc{H})}

\DeclareMathOperator{\tr}{Tr}
\DeclareMathOperator{\ran}{Ran}
\DeclareMathOperator{\Var}{Var}

\DeclareMathOperator{\hess}{Hess}
\DeclareMathOperator{\grad}{grad}

\DeclareMathOperator{\Ric}{Ric}
\DeclareMathOperator{\Ent}{Ent}

\def \id {{\rm id}}

\def  \mi {{\bf 1}}

\title[Interpolation quantum functional inequalities]{Interpolation between modified logarithmic Sobolev and Poincar\'{e} inequalities for quantum Markovian dynamics
}

\begin{document}

\author[B. Li]{Bowen Li} %
\address[B. Li]{Department of Mathematics, Duke University, Durham, NC 27708.}
\email{bowen.li200@duke.edu}

\author[J. Lu]{Jianfeng Lu}
\address[J. Lu]{Departments of Mathematics, Physics, and Chemistry, Duke University, Durham, NC 27708.}
\email{jianfeng@math.duke.edu}


\begin{abstract}
We define the quantum $p$-divergences and introduce Beckner's inequalities for primitive quantum Markov semigroups \mb{on a finite-dimensional matrix algebra}
satisfying the detailed balance condition. Such inequalities quantify the convergence rate of the quantum dynamics in the noncommutative $L_p$-norm. We obtain a number of implications between Beckner's inequalities and other quantum functional inequalities, \mb{as well as the hypercontractivity}.
In particular, we show that the quantum Beckner's inequalities interpolate between the Sobolev-type inequalities and the Poincar\'{e} inequality in a sharp way. We provide a uniform lower bound for the Beckner constant $\alpha_p$ in terms of the spectral gap and establish the stability of $\alpha_p$ with respect to the invariant state. As applications, we compute the Beckner constant for the depolarizing semigroup and \mb{discuss the mixing time}. For symmetric quantum Markov semigroups, we derive the moment estimate, which further implies a concentration inequality. 

We introduce a new class of quantum transport distances $W_{2,p}$ interpolating the quantum 2-Wasserstein distance 
by Carlen and Maas [J. Funct. Anal. 273(5), 1810–1869 (2017)] and a noncommutative $\dot{H}^{-1}$ Sobolev distance. We show that the quantum Markov semigroup with \mb{$\si$-GNS detailed balance} is the gradient flow of a quantum $p$-divergence with respect to the metric $W_{2,p}$. We prove that the set of quantum states equipped with $W_{2,p}$ is a complete geodesic space. We then consider
the associated entropic Ricci curvature lower bound via the geodesic convexity of $p$-divergence, and obtain an HWI-type interpolation inequality. This enables us to prove that the positive Ricci curvature implies the quantum Beckner's inequality, from which a transport cost and Poincar\'{e} inequalities can follow.

\end{abstract}

\maketitle

\section{Introduction}

Realistic quantum systems inherently interact with their surroundings and can be generally modeled by the open quantum dynamics. In the weak coupling limit between the system and the environment, the dynamics would be Markovian and described by the so-called quantum Markov semigroup (QMS) or the Lindblad equation, which is a natural analog of the Fokker-Planck equation in the quantum setting \cite{breuer2002theory}. Similarly to the theory of Markov semigroups, the analysis of the mixing time 
is of central importance for a  QMS, and is closely related to the functional inequalities. 
In this work, we are interested in a class of convex Sobolev inequalities, referred to as quantum Beckner's inequalities. We will investigate their main properties and relations with other known quantum functional inequalities, such as Poincar\'{e} and modified log-Sobolev inequalities, via both algebraic and geometric approaches.  

\subsection{Classical convex Sobolev inequality} To motivate this work, we first review the results on the convex Sobolev inequalities in the classical setting. 
Let $(P_t)_{t \ge 0}$ be the symmetric diffusion semigroup associated with a Markov process $(X_t)_{t \ge 0}$ on a Riemannian manifold $M$ with metric $g(\dd,\dd)$ that has a unique invariant measure $\pi$. We denote by $L$ the generator of $P_t$ and define the Dirichlet form $\mc{E}(f,g)  := - \pi [f L g]$ for functions $f$ and $g$, \mb{where $\pi[\dd]$ denotes the expectation with respect to the measure $\pi$}. Bakry and \'{E}mery in their seminal work \cite{bakry1985diffusions} showed that if there exists $\kappa > 0$ such that for $f \ge 0$,
\begin{align} \label{eq:curvature_dimension}
    \Gamma_2(f,f) \ge \kappa \Gamma (f,f)\,,
\end{align}
then the convex Sobolev inequality holds:
\begin{align} \label{eq:convex_ineq}
 2 \kappa \Ent_\pi^\phi(f) \le \mc{E}(\phi'(f),f)\,,
\end{align}
which is equivalent to the exponential decay of the $\phi$-entropy $\Ent_\pi^\phi(f):= \pi[\phi(f)] - \phi(\pi[f])$ and characterizes the convergence rate of the Markov process towards its invariant measure. Here $\phi:[0,\infty) \to \R$ is assumed to be
a smooth convex function such that $\phi(1) = \phi'(1) = 0$ and $1/\phi''$ is concave.
\mb{$\Gamma$ and $\Gamma_2$ are carr\'{e} du champ operators defined as follows \cite[Section\,1.16.1]{bakry2014analysis}: for suitable functions $f$ and $g$, }
\begin{align*}
\mb{ \Gamma (f,g) := \frac{1}{2}(L(fg) - f L(g) - L(f)g)\,,\ \Gamma_2 (f,g) := \frac{1}{2}(L \Gamma(f,g) - \Gamma(f, Lg) - \Gamma(Lf, g))\,.}
\end{align*}
In the cases: $\phi_1(s) = s(\log s - 1) + 1$ and $\phi_2(s) = s^2 - 2 s + 1$, up to some constant,  \eqref{eq:convex_ineq} gives the well-known modified log-Sobolev inequality (MLSI) and the Poincar\'{e} inequality, respectively, 
\begin{align} \label{ineq_a}
\alpha \big(\pi[f \log f] - \pi[f] \log \pi[f]\big) \le \mc{E}(\log f, f)\,,
\end{align}
and 
\begin{align} \label{ineq_b}
    \lad \big(\pi[f^2] - \pi[f]^2 \big) \le \mc{E}(f , f)\,.
\end{align}
If we consider the interpolating family $\phi_p(s) = (s^p - s)/(p - 1) - s + 1$, $1 < p \le 2$, between $\phi_1$ and $\phi_2$, we obtain the Beckner's inequality:
\begin{equation} \label{ineq_c}
    \alpha_p \big( \pi[f^p] - \pi[f]^p \big)  \le p \mc{E}(f^{p - 1}, f)\,.
\end{equation}
Moreover, note from the diffusion property: $L \psi(f) = \psi'(f)L f + \psi''(f) \Gamma f$ for suitably smooth functions $\psi$ and $f$, that $\mc{E}(\log f , f) = 4 \mc{E}(f^{1/2}, f^{1/2})$ and $\mc{E}(f^q, f^{2 - q}) = (2 q - q^2) \mc{E}(f,f)$. By substituting $f = g^2$ and $f = g^q$ with $q = 2/p$ into \eqref{ineq_a} and \eqref{ineq_c}, respectively, up to constants, we have the usual log-Sobolev inequality (LSI):
\begin{equation} \label{ineq_d}
    \beta \big(\pi[g^2 \log g^2] - \pi[g^2] \log \pi[g^2]\big) \le  \mc{E}(g, g)\,,
\end{equation}
and the original Beckner's inequality first introduced in \cite{beckner1989generalized} for the Gaussian measure on $\R^d$:
\begin{equation} \label{ineq_e}
     \beta_q \big( \pi[g^2] - \pi[g^q]^{2/q} \big)  \le (2-q) \mc{E}(g, g)\,.
\end{equation}
The condition \eqref{eq:curvature_dimension} admits a deep geometric interpretation, and it is called the
Bakry-\'{E}mery (curvature-dimension) condition or $\Gamma_2$-criterion. To make this point clearer, let $L = \Delta_g - \na W \dd \na$ be the generator associated with the Ornstein–Uhlenbeck process on the manifold $M$ that admits an invariant measure $d \pi = e^{- W} d \text{vol}_M$, where $W$ is the potential and $\text{vol}_M$ is the volume form on $M$. With the help of Bochner's formula, we can compute $\Gamma(f,f) = |\na f|^2$ and
$\Gamma_2(f,f) = |\na^2 f|^2 + \Ric(L)(\na f , \na f)$, where $\Ric(L)$ is the Ricci tensor for the generator $L$, defined by $\Ric(L): = \Ric_g + \na^2 W$ with $\Ric_g$ being the standard Ricci curvature of $M$. It is easy to prove that the condition \eqref{eq:curvature_dimension} holds if and only if the Ricci curvature of $L$ is bounded below: $\Ric(L)(\dd,\dd) \ge \kappa g(\dd,\dd)$. Otto-Villani \cite{otto2000generalization} and von Renesse-Sturm \cite{von2005transport} further observed that \eqref{eq:curvature_dimension} is also equivalent to that the relative entropy with respect to $d \pi$ is displacement $\kappa$-convex on the Wasserstein space of probability measures on $M$. Inspired by this characterization, Sturm \cite{sturm2006geometry} and Lott-Villani \cite{lott2009ricci} extended the notion of Ricci curvature 
to metric measure spaces by exploiting the convexity properties of entropy functionals. See \cite{bakry2014analysis,villani2009optimal} for more details.

The above framework establishes a beautiful connection between various subjects such as partial differential equations (PDE), probability, and geometry, and has led to important research progress in these fields. The key step in the Bakry-\'{E}mery arguments \cite{bakry1985diffusions} is to estimate the second derivative of the relative entropy along the Markov semigroup, where the calculation depends on  Bochner’s formula or, more abstractly, the diffusion property. Arnold et al. \cite{arnold2001convex,arnold2008large,arnold2014sharp} revisited the Bakry-\'{E}mery method in the PDE framework and characterized the long-time asymptotics for various classes of Fokker-Planck type equations based on the convex Sobolev inequalities; see also \cite{otto2001geometry,del2002best,carrillo2000asymptotic,markowich2000trend} for the applications of functional inequalities in nonlinear Fokker-Planck type equations. Among the general convex Sobolev inequalities, the Beckner's inequality is of particular interest, since it provides an interpolating family between MLSI and Poincar\'{e} inequality \mb{and can estimate the tail of the given measure.} The recent work \cite{gentil2021family} proved a class of weighted Beckner's inequalities and the refined ones based on the Bakry-\'{E}mery method and the curvature-dimension conditions.  We also mention that \cite{arnold2007interpolation} proved the inequality \eqref{ineq_e} by  the hypercontractivity and spectral estimates. In particular, Dolbeault et al. \cite{dolbeault2009new,dolbeault2012poincare} explored the gradient flow structure of the Fokker-Planck equation for general entropy functionals and proved the contraction of the associated transport distance along the Fokker-Planck flow, which gave a unified  gradient flow framework for investigating the convex Sobolev inequalities \eqref{eq:convex_ineq}. 

It is also desirable to extend the theory of convex Sobolev inequalities to the setting of finite Markov chains. In this case, due to the lack of chain rule, the inequalities \eqref{ineq_a} and \eqref{ineq_d}, also, \eqref{ineq_c} and \eqref{ineq_e} are not equivalent (one is stronger than the other) \cite{bobkov2006modified}. For instance, Dai Pra et al.  \cite{dai2002entropy} provided an example where the MLSI \eqref{ineq_a} holds while the LSI \eqref{ineq_d} fails.  In what follows, to avoid confusion between \eqref{ineq_c} and \eqref{ineq_e}, following \cite{adamczak2022modified} and \cite{chafai2004entropies} we call the inequality of the form \eqref{ineq_e} the dual Beckner's inequality. 
Similarly to the diffusion case, the Bakry-\'{E}mery method and gradient flow techniques are two main approaches for the validity of \eqref{eq:convex_ineq}. J\"{u}ngel and Yue \cite{jungel2017discrete} followed Bakry-\'{E}mery's ideas and gave the conditions of $\phi$ under which \eqref{eq:convex_ineq} holds. The proof 
    relies on a discrete Bochner-type identity that was first introduced in \cite{boudou2006spectral,caputo2009convex}. Recently, Weber and Zacher \cite{weber2021entropy} proposed discrete analogs of the condition \eqref{eq:curvature_dimension} such that the MLSI \eqref{ineq_a} and Beckner’s inequality \eqref{ineq_c} hold. Their argument, different from \cite{jungel2017discrete}, is based on the modified $\Gamma$ and $\Gamma_2$ operators that satisfy some kind of discrete diffusion property. 
    We point out that a probabilistic approach, based on the Bakry-\'{E}mery method and the coupling arguments, for the discrete convex Sobolev inequalities can be found in \cite{conforti2022probabilistic} by Conforti. 
    
    The starting point of the gradient flow approach for discrete functional inequalities is \cite{maas2011gradient} where Maas defined a discrete transport distance such that the continuous time finite Markov chains can be identified as the gradient flow of the relative entropy. Following the ideas of \cite{lott2009ricci,sturm2006geometry}, Erbar and Maas \cite{erbar2012ricci} introduced the discrete Ricci curvature based on this discrete Wasserstein metric, and derived a number of functional inequalities including the discrete MLSI and the transport cost inequalities; see also \cite{erbar2018poincare}. Later, Fathi and Maas \cite{fathi2016entropic} generalized the discrete Bochner formula \cite{caputo2009convex} and developed a systematic approach for estimating the discrete Ricci curvature lower bounds. It is worth mentioning that both the discrete Bakry-\'{E}mery condition in \cite{weber2021entropy} and the discrete Ricci curvature in \cite{erbar2012ricci} enjoy the tensorization properties, and the aforementioned general results can be applied to several interesting models such as birth-death processes, random transposition models and Bernoulli-Laplace models (see related papers for details). We refer the readers to the review \cite{maas2017entropic} and the references therein for other notions of the Ricci curvature in the discrete setting and their implications on functional inequalities. 

\subsection{Quantum functional inequalities} 
In analogy with the classical case, quantum functional inequalities play a fundamental role in understanding the asymptotic behavior of a QMS. The study of LSI in the noncommutative setting may date back to \cite{gross1975hypercontractivity},
and its connections with hypercontractivity were fully discussed in the seminal work by Olkiewicz and  Zegarlinski \cite{olkiewicz1999hypercontractivity}. The quantum MLSI was initially introduced by Kastoryano and Temme \cite{kastoryano2013quantum} to derive improved bounds on the mixing time of primitive quantum Markov processes, surpassing those obtained via the Poincaré inequality in \cite{temme2010chi}.
\mb{The investigation of the quantum MLSI constant has been carried out in detail for specific models using various techniques: the depolarizing semigroup by explicit computation \cite{muller2016relative}, the doubly stochastic qubit Lindbladian by a comparison method \cite{muller2016entropy}, and quantum spin
lattice systems by quasi-factorization for the entropy \cite{bardet2021modified,cuevas2019quantum,capel2020modified}, to name a few.}
Regarding the general validity of MLSI in the quantum setting, the notion of Ricci curvature lower bounds (geodesic convexity), pioneered by Carlen and Maas \cite{carlen2014analog,carlen2017gradient}, has shown its utility in proving the quantum MLSI and related functional inequalities with numerous 
applications in concrete physical models. We briefly outline the main progress in this direction below. Carlen and Maas \cite{carlen2017gradient} introduced a quantum analog of 2-Wasserstein distance such that the primitive QMS satisfying $\si$-GNS detailed balance condition (cf.\,Definition 
\ref{def:sidbc}) can be written as the gradient flow of the relative entropy and showed that the relative entropy is geodesically convex for the Fermi and Bose Ornstein-Uhlenbeck semigroups, which extended their previous work \cite{carlen2014analog}. Based on Carlen and Maas's results, Datta and 
Rouz{\'e} \cite{rouze2019concentration,datta2020relating} considered the Ricci curvature of a QMS, and obtained some quantum Sobolev and concentration inequalities, generalizing the results in the classical regime \cite{otto2001geometry,erbar2012ricci,erbar2018poincare}. Wirth 
and Zhang \cite{wirth2021curvature} further introduced noncommutative curvature-dimension conditions and derived some dimension-dependent quantum functional inequalities. \mb{In addition, the Bakry-\'{E}mery 
method, which has been successfully adapted to the discrete case, has also been explored for the quantum semigroups recently with fruitful applications \cite{li2020complete,li2020graph,gao2020fisher,brannan2021complete,brannan2022complete,wirth2021complete}, where the crucial monotonicity of Fisher information was derived from different starting points: the gradient condition \cite{gao2020fisher}, the geometric
Ricci curvature condition \cite{li2020graph,brannan2022complete}, and the gradient estimate \cite{wirth2021complete}. The relations between these conditions, as well as the entropic Ricci curvature bound \cite{carlen2020non,datta2020relating}, can be found in \cite{li2020graph,wirth2021complete,brannan2022complete,wirth2021curvature}.}

One of the favorable features of the classical LSI is the tensorization property, which enables obtaining the functional inequalities for the tensor product systems from those for the subsystems. However, this property is known to fail for the quantum MLSI (cf.\cite[Proposition 4.21]{brannan2022complete}). To circumvent such difficulty, Gao et al. \cite{gao2020fisher} introduced the complete modified log-Sobolev inequality (CMLSI), which is a stronger notion than the MLSI,
and showed that it satisfies the desired tensorization properties. In \cite{gao2021complete},
Gao and Rouz\'{e} proved, by a two-sided estimate for the relative entropy, that the CMLSI holds for any finite-dimensional non-primitive QMS with $\si$-GNS detailed balance. \mb{The very recent work \cite{gao2022complete} provided a generic lower bound for the CMLSI constant by the inverse of completely bounded mixing time and an improved data processing inequality, which improves the results in \cite{gao2021complete}.} See also  \cite{brannan2022complete,li2020graph,brannan2021complete, gao2021geometric} for 
the geometrical approaches for studying the CMLSI.

\subsection{Main results} Although there has been much progress on Beckner's inequalities in both diffusion and discrete cases as reviewed above, the results for quantum Beckner's inequalities are quite limited. We only note the recent work by Li \cite{li2020complete} where the author investigated the matrix-valued Beckner's inequalities, in terms of the Bregman relative entropy \cite{molnar2016maps}, for symmetric semigroups on a finite von Neumann algebra. This work is devoted to further investigation on this topic. 
We consider the primitive QMS
satisfying certain detailed balance conditions, and define the family of quantum $p$-Beckner's inequalities and their dual version, by extending the definitions in \cite{adamczak2022modified,latala2000between,chafai2004entropies} for classical Markov semigroups; see Definition \ref{def:quantum_func} for the functional inequalities that we will mainly focus on.
It turns out that the $p$-Beckner's inequality \eqref{ineq_becp} describes the rate at which the quantum $p$-divergence $\mc{F}_{p,\si}$ \eqref{def:quanpdivi} tends to zero along the QMS. Note that $\mc{F}_{p,\si}$ can be viewed as the normalized noncommutative $L_p$-norm. The diagram in Figure \ref{fig:diagram} below summarizes part of the main results of Sections \ref{sec:inter_primi} and \ref{sec:positive_stab}. 

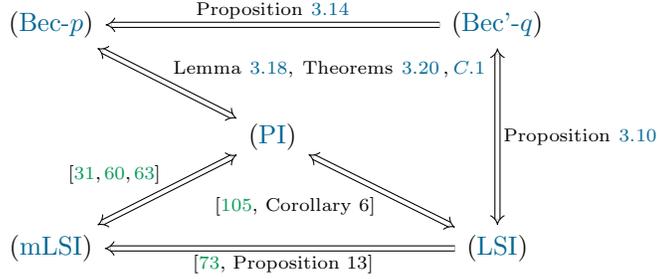
\begin{figure}[!htbp] 
\begin{tikzcd}[row sep= 2.5 em , column sep= 5 em]
\eqref{ineq_becp} 
\arrow[Leftrightarrow]{rd}{\mb{\text{Lemma}\ \ref{lem:upppoin},\ \text{Theorems}\ \ref{them:lowerbeck}\,,\, \ref{thm:beck_poincare}}} & &
\eqref{ineq_dbecq} \arrow[Leftrightarrow]{dd}{\text{Proposition} \ \ref{dualbecktolsi}} \arrow[Rightarrow]{ll}[swap]{\text{Proposition} \ \ref{prop:beck_to_dual_beck}} \\
&  \eqref{ineq_pi2} \arrow[Leftrightarrow]{rd}[swap]{\text{\cite[Corollary 6]{temme2014hypercontractivity}}}       &    \\
\eqref{ineq_mlsi} \arrow[Leftrightarrow]{ur}{\text{\cite{gao2021complete,gao2022complete,cao2019gradient}}}  && \eqref{ineq_lsi}  \arrow[Rightarrow]{ll}{\text{\cite[Proposition 13]{kastoryano2013quantum}}}
\end{tikzcd}
\caption{Chain of quantum convex Sobolev inequalities. \mb{We remark that the implications from \eqref{ineq_pi2} to other inequalities \eqref{ineq_becp}, \eqref{ineq_mlsi}, and \eqref{ineq_lsi} would generally involve a constant depending} \mb{on the properties of invariant states}.}
\label{fig:diagram}
\end{figure}

The relations between \eqref{ineq_pi2}, \eqref{ineq_mlsi}, and \eqref{ineq_lsi} have been well investigated. \cite[Corollary 6]{temme2014hypercontractivity} provided a two-sided bound for the LSI constant by Poincar\'{e} constant; \cite{kastoryano2013quantum} showed that the MLSI constant is bounded below by the LSI constant and above by Poincar\'{e} constant. This also allows us to conclude that the MLSI constant and Poincar\'{e} constant can be compared with each other; 
see also \cite[Proposition I.7]{cao2019gradient} and \cite[Theorem 3.3]{gao2021complete}. In Section \ref{sec:inter_primi}, we prove that the dual  $q$-Beckner's inequalities \eqref{ineq_dbecq}, as an interpolating family between \eqref{ineq_lsi} and \eqref{ineq_pi2}, is comparable to \eqref{ineq_lsi} in the sense of Proposition \ref{dualbecktolsi}; \mb{see Remark \ref{lem:cont_dualbeck} for the tightness of the comparison.}
We also show in Proposition \ref{prop:beck_to_dual_beck} that \eqref{ineq_dbecq} is stronger than \eqref{ineq_becp}. Propositions \ref{prop:mono_dual_beck} and \ref{prop:mono_beck} discuss the monotonicity of the dual Beckner constant $\beta_q(\mc{L})$ and the Beckner constant $\alpha_p(\mc{L})$, respectively.

In Section \ref{sec:positive_stab}, we investigate the quantum $p$-Beckner's inequalities in detail. \mb{We first discuss in Proposition \ref{prop:sand} the connection between $p$-Beckner's inequality \eqref{ineq_becp} and the sandwiched R\'{e}nyi 
entropic inequality \eqref{eq:sand_decay} defined in \cite{muller2018sandwiched}. In brief, the R\'{e}nyi entropic inequality \eqref{eq:sand_decay} implies \eqref{ineq_becp}, 
while \eqref{ineq_becp} only implies a restricted \eqref{eq:sand_decay}. Then, we consider the relations between \eqref{ineq_becp} and quantum $p$-log-Sobolev inequalities 
\eqref{ineq_lsip} (equivalently, the hypercontractivity) in Proposition \ref{prop:connectpsob}, where we find \eqref{ineq_lsip} is stronger (resp.,weaker) than \eqref{ineq_becp} for $p > 1$ (resp., $0 < p < 1$). A two-sided 
estimate for the Beckner constant $\alpha_p(\mc{L})$ in terms of the Poincar\'{e} constant $\lad(\mc{L})$ is given in Lemma \ref{lem:upppoin} and Theorem \ref{them:lowerbeck}. We prove, 
by contradiction, in Theorem \ref{thm:beck_log_sobo} that $\alpha_p(\mc{L}) \to \alpha_1(\mc{L})$ holds as $p \to 1^+$, where $\alpha_1(\mc{L})$ is the MLSI constant.} This extends the result \cite[Theorem 1.1]{adamczak2022modified} in the 
classical setting. We also extend the main result in \cite{junge2019stability} and provide a stability estimate for the Beckner constant $\alpha_p(\mc{L})$ with respect to the invariant state $\si$; see Theorem \ref{thm:beck_stab}. In Section \ref{sec:app_exp}, \mb{we first compute the quantum Beckner constant for the depolarizing semigroup with asymptotically tight lower and upper bounds;  see Propositions \ref{prop:beck_depol} and \ref{prop:upperlowerdep}. We then derive a mixing time bound for the QMS from the $p$-Beckner's inequality \eqref{ineq_becp} in Proposition \ref{prop:mixing}.} Moreover, in Proposition \ref{prop:moment}, we extend \cite[Proposition 3.3]{adamczak2022modified} for the classical case and obtain moment estimates under \eqref{ineq_becp}. \mb{As a complementary result, we also provide a generic lower bound for the Beckner constant for non-primitive QMS in Appendix \ref{app:nonprimitive} based on the key Lemma \ref{lem:two-sided}, following the work \cite{gao2021complete}.}


Another motivation for the current work is \cite{dolbeault2009new,dolbeault2012poincare}, where the authors provided a gradient flow approach for the classical Beckner's inequalities. To be precise, 
Dolbeault et al. \cite{dolbeault2009new} defined the following class of transport distances $\mc{W}_{2,\alpha,\gamma}$ with $\alpha \in [0,1]$ \`{a} la Benamou-Brenier: for probability measures $\mu_0$ and $\mu_1$ on $\R^d$, 
\begin{align} \label{def:distance_savare}
\mc{W}_{2,\alpha,\gamma}(\mu_0,\mu_1): = \inf\Big\{
\int_0^1\int_{\R^d} \rho_t^{-\alpha}|w_t|^2\ d \gamma dt\,;\ \p_t \mu_t + \na \dd {\bf \nu}_t = 0\,,\ \mu_t = \rho_t \gamma + \mu_t^\perp\,,\ \nu_t = w_t \gamma \ll \gamma 
\Big\}\,,
\end{align}
where $\gamma$ is a reference Radon measure \mb{and $\mu_t^\perp$ is the singular part of $\mu_t$ in its Radon-Nikodym decomposition with respect to $\gamma$.}
In the  case $\alpha = 1$, $\mc{W}_{2,\alpha,\gamma}(\mu_0,\mu_1)$ gives the 2-Wasserstein distance \cite{benamou2000computational}, while when $\alpha = 0$, it is equivalent to the weighted homogeneous $\dot{H}_\gamma^{-1}$ Sobolev distance \cite{peyre2018comparison,villani2003topics}:
\begin{align*}
    \norm{\mu_0 - \mu_1}_{\dot{H}_\gamma^{-1}} = \sup \Big\{\int_{\R^d} \xi\ d(\mu_0 - \mu_1)\,;\ \xi \in C_c^1(\R^d)\,,\ \int_{\R^d} |\na \xi|^2\ d \gamma \le 1  \Big\}\,.
\end{align*}
Thus $\mc{W}_{2,\alpha,\gamma}$ can be viewed as a natural interpolating family between them. Moreover, let the reference measure be $\gamma := e^{-V} \mathscr{L}^d $ with the potential $V$ being smooth and convex, where $\mathscr{L}^d$ is the Lebesgue measure on $\R^d$.
With such choice of $\gamma$, they showed that the gradient flow of the (Tsallis) functional, for $\alpha \in [0,1)$,
\begin{align} \label{def:class_func}
    \mathscr{F}_\alpha(\mu): = \frac{1}{(2-\alpha)(1-\alpha)} \int_{\R^d} \rho^{2-\alpha}\  d\gamma\,,\q \mu = \rho \gamma\,,
\end{align}
is the Fokker-Planck equation: 
\begin{align*}
\p_t \mu - \Delta \mu - \na \dd (\mu \na V) = 0\,.
\end{align*}
In the subsequent work \cite{dolbeault2012poincare}, they further proved that $\mc{F}_\alpha$ is geodesically $\lad$-convex under the assumption: $\na^2 V \ge \lad I$ for $\lad > 0$, which implies the classical Beckner's inequality.

It is easy to note that in the commutative setting, up to some constant, our quantum $p$-divergence $\mc{F}_{p,\si}$ \eqref{def:quanpdivi} is nothing else but the functional $ \mathscr{F}_\alpha$ in \eqref{def:class_func} with $p = 2 - \alpha$.  In Section \ref{sec:qot_beck}, we extend the results in \cite{dolbeault2009new,dolbeault2012poincare}
to the quantum regime and provide a geometric characterization for the quantum $p$-Beckner's inequality. To do so, we first construct a Riemannian metric $g_{p,\rho}$ on the quantum states in Section \ref{sec:gradient}, so that the 
$\si$-{\rm GNS} symmetric QMS is the gradient flow of $p$-divergence $\mc{F}_{p,\si}$ with respect to $g_{p,\rho}$. Then in Section \ref{sec:quantum_distance}, we investigate the properties of the associated Riemannian distance, denoted by $W_{2,p}$ (cf.\,\eqref{def:wp1}), which can be regarded as a quantum analog of $W_{2,\alpha,\gamma}$ in \eqref{def:distance_savare}. The main result in this section is Theorem \ref{thm:main_wasser}, where we show that $(\dh,W_{2,p})$ is a complete geodesic metric space. We also prove in Proposition \ref{prop:repkernel} that similarly to the classical case,
the new class of distances $W_{2,p}$ is an interpolating family between the quantum 2-Wasserstein distance defined by Carlen and Maas \cite{carlen2017gradient} and a noncommutative $\dot{H}^{-1}$ Sobolev distance \eqref{eq:distance_1}. With these results, it is straightforward to define the entropic Ricci curvature associated with the functional $\mc{F}_{p,\si}$ 
in the spirit of \cite{sturm2006geometry,lott2009ricci,datta2020relating}. We then derive an HWI-type interpolation inequality from the Ricci curvature lower bound and show that the positive Ricci curvature can imply Beckner's inequality \eqref{ineq_becp}. Further, we prove the 
following chain of quantum functional inequalities:
\begin{equation} \label{eq:chain}
\eqref{ineq_becp}   \xLongrightarrow{\text{Proposition}\  \ref{propa}}   \eqref{ineq:tc} \xLongrightarrow{\text{Proposition}\  \ref{propb}}  \eqref{ineq_pi}\,,
\end{equation} 
where \eqref{ineq:tc} is a transport cost inequality associated with $W_{2,p}$. These results are presented in Section \ref{sec:geodesic_convexity}.

\subsection{Layout and notation} \mb{The rest of this work is organized as follows. We will restrict our discussion to a finite-dimensional matrix algebra. 
In Section \ref{sec:pre}, we give preliminary definitions and results used throughout this work. Some additional preliminaries are included in Appendix \ref{app:pre}. In Section \ref{sec:beck}, we define the family of quantum Beckner's inequalities and investigate its properties and relations with other known functional inequalities. Section \ref{sec:qot_beck} is devoted to a gradient flow framework for Beckner's inequality. In Section \ref{sec:discussion}, we conclude this work with a discussion of some open questions. Moreover, Appendix \ref{app:dbc} includes a comparison of detailed balance conditions, while in Appendix \ref{app:nonprimitive} we give a brief introduction for the non-primitive Beckner's inequality. 

We fix some notations that will be used in this work. }
\begin{enumerate}[\textbullet]
    \item Let $\mc{B}(\mc{H})$ denote the space of bounded operators on a finite-dimensional Hilbert space $\mc{H}$ of dimension $d < \infty$. 
  \mb{We use $\mc{B}_+(\mc{H})$ for the set of full-rank (invertible) operators in $\mc{B}(\mc{H})$.}
    $\mc{B}_{sa}(\mc{H})$ is the subspace of self-adjoint operators on $\mc{H}$, while $\mc{B}^{+}_{sa}(\mc{H})$ is the cone of positive semidefinite operators.  For simplicity, in what follows, by $A \ge 0$ (resp., $A > 0$) we mean a positive semidefinite (resp., definite) operator.  
      \item The identity operator on $\mc{H}$ is written as $\mi_{\mc{H}}$ (or $\mi$, if there is no confusion). Similarly, the identity superoperator on $\bh$ is denoted by $\id_{\mc{H}}$, or simply $\id$.  
       \item We denote by $\mc{D}(\mc{H}) := \{\rho \in \mc{B}_{sa}^+(\mc{H})\,;\  \tr \rho =1 \}$ the set of density operators (quantum states), and by $\mc{D}_+(\mc{H})$ the full-rank density operators.
    \item We denote by $\l\dd,\dd\r$ the Hilbert-Schmidt inner product on $\bh$, i.e., $\l X, Y\r = \tr (X^*Y)$, where $X^*$ is the adjoint operator of $X$. Moreover, we write $\Phi^\dag$ for the adjoint of a superoperator $\Phi: \mc{B}(\mc{H}) \to \mc{B}(\mc{H})$ with respect to the inner product $\l \dd,\dd \r $. 
    The modulus of $X \in \bh$ is defined by $|X|: = \sqrt{X^* X} $. 
    \item We define the Schatten $p$-norm by $\norm{X}_p = \tr(|X|^p)^{1/p}$ \mb{for $X \in \bh$ if $p > 0$;  $X \in \mc{B}_+(\mc{H})$ if $p < 0$,} 
    where $\norm{\dd}_\infty$ is the operator norm. For a superoperator on $\bh$, we simply use $\norm{\dd}$ for its operator norm.
    \item Let $\mc{M}$ be a subset of $\mc{B}(\mc{H})$. We denote by $\mc{M}^J$ the set of vector fields over $\mc{M}$, i.e., ${\bf A} = (A_1, \cdots, A_J)  \in \mc{M}^J$ for $A_j \in \mc{M}$, $1 \le j \le J$. The Hilbert-Schmidt inner product     naturally extends to $\mc{M}^J$ as $\l {\bf A} , {\bf B} \r = \sum_{j = 1}^J \l A_j, B_j\r$. 
    \item \mb{For $p \in \R \backslash \{0,1\}$, we define its H\"{o}lder conjugate $\h{p}: = p/(p-1)$ satisfying $1/p + 1/\h{p} = 1$.}
\end{enumerate}

\mb{We end the introduction with some remarks. 
For ease of exposition, in many statements below, we only consider invertible $X \in \mc{B}_+(\mc{H})$ (so that $|X|^\alpha$ for any $\alpha \in \R$ is well-defined). Most of them still hold for non-invertible matrices by an approximation argument, which can be easily checked. For example, the first statement in Lemma \ref{lem:prop:norm_power} clearly holds for $0 < p \le q$ and $X \in \bh$. Moreover, to make the presentation cleaner, many results below are only stated for $p \neq 1$, while the case $p = 1$ can be easily obtained by taking a limit; see Remark \ref{rem:limit} for example.}

\section{Preliminaries} \label{sec:pre}

\subsection{Quantum Markov semigroup}

Let us first recall preliminaries about the \mb{finite-dimensional} Markovian open quantum dynamics. 
We say that $(\mc{P}_t)_{t \ge 0}: \mc{B}(\mc{H}) \to \mc{B}(\mc{H})$ is a quantum Markov semigroup (QMS) if $\mc{P}_t$ is a $C_0$-semigroup of completely positive, unital maps, whose generator $\mc{L}$ is called the Lindbladian, defined by $\mc{L}(X): = \lim_{t \to 0} t^{-1}(\mc{P}_t (X) - X)$. A quantum channel $\Phi: \bh \to \bh$ is a completely positive trace preserving (CPTP) map. Then, the dual QMS $\mc{P}_t^\dag$ is a semigroup of quantum channels, and the associated equation $\frac{d}{dt} \rho = \mc{L}^\dag \rho$ is referred to as the Lindblad equation. \mb{It is easy to check that $X > 0$ implies $\mc{P}_t (X) > 0$ for any $t \ge 0$. 
A QMS $\mc{P}_t$ is said to be primitive} if it admits a unique full-rank invariant state $\si$ such that $\mc{P}_t^\dag(\si) = \si$ for $t \ge 0$. In this case, there holds \cite{frigerio1982long}
\begin{align} \label{eq:conver_qms}
    \lim_{t \to \infty} \mc{P}_t(X) = \tr(\si X)\mi\,,\q \forall X \in \bh\,.
\end{align}

For $A \in \bh$ we define the left and right multiplication operator on $\bh$ by $L_A = A X$ and $R_A = X A$, respectively. It is easy to see that $L_{f(A)} = f(L_A)$ and $R_{f(A)} = f(R_A)$ holds for $A > 0$ and 
functions $f: (0,\infty) \to \R$. We also define the relative modular operator $\Delta_{\rho,\si} = L_\rho R_\si^{-1}: \mc{B}(\mc{H}) \to \mc{B}(\mc{H})$ for 
$\rho,\si \in \mc{D}_+(\mc{H})$, and simply write it as $\Delta_\si$ if $\si = \rho$. We next introduce the quantum detailed balance condition (DBC). For this, we define a family of inner products on $\bh$: for a given $\si \in \mc{D}_+(\mc{H})$ and $s \in \R$, 
\begin{align} \label{def:s-inner}
    \l X, Y\r_{\si,s} := \tr(\si^s X^* \si^{1-s} Y) = \l X, \Delta_\si^{1-s} (Y) \si \r\,,\q \forall X, Y \in \bh\,,
\end{align}
where $\l \dd, \dd\r_{\si,1}$ and $\l \dd, \dd \r_{\si,1/2}$ are GNS and KMS inner products, respectively. In particular, when $\si$ is the maximally mixed state $\mi/d$, all the inner products $\l \dd,\dd \r_{\si,s}$ reduce to the normalized Hilbert-Schmidt inner product: 
\begin{align} \label{eq:productid}
    \l X , Y\r_{\frac{\mi}{d}} := \frac{1}{d} \l X , Y\r\,.
\end{align}

\begin{definition} \label{def:sidbc}
\mb{We say that a QMS $\mc{P}_t$ satisfies the
$\si$-{\rm GNS DBC} (resp., $\si$-{\rm KMS DBC}) for some $\si \in \mc{D}_+(\mc{H})$ if its generator $\mc{L}$ is self-adjoint with respect to the inner product $\l\dd, \dd\r_{\si,1}$  (resp., $\l\dd, \dd\r_{\si,1/2}$).}
\end{definition}
One can readily see that if $\mc{P}_t$ satisfies the \mb{$\si$-{\rm GNS DBC} or $\si$-{\rm KMS DBC} for some $\si \in \mc{D}_+(\mc{H})$}, then $\si$ is an 
invariant state of $\mc{P}_t$, and that $\mc{P}_t$ is symmetric (i.e., $\mc{P}_t = \mc{P}_t^\dag$) if and only if it 
satisfies the \mb{{\rm GNS DBC} or {\rm KMS DBC}} for the maximally mixed state $\si = \mi/d$. For any $f: (0,\infty) \to (0,\infty)$ and $\si \in \mc{D}_+(\mc{H})$, we can define the operator:
\begin{align} \label{def:operator_kernel}
    J_\si^f := R_\si f(\Delta_\si): \bh \to \bh\,,
\end{align}
and the associated inner product:
\begin{align} \label{eq:general_inner}
    \l X, Y\r_{\si,f} := \l X, J_\si^f (Y)\r\,.
\end{align}
It is clear that $\l X, Y\r_{\si,f}$ with $f = x^{1-s}$ gives the inner product \eqref{def:s-inner}; and the adjoint of a linear operator $\mc{K}$ on $\bh$ with respect to $\l\dd,\dd\r_{\si,f}$ is given by $ ( J_\si^f)^{-1} \mc{K}^\dag   J_\si^f $.
The following result from \cite[Theorem 2.9]{carlen2017gradient} relates the self-adjointness of $\mc{L}$ with respect to different inner products. 

\begin{lemma} \label{lem:self_adjoint}
If a QMS $\mc{P}_t$ satisfies the \mb{$\si$-{\rm GNS DBC}} for some  $\si \in \mc{D}_+(\mc{H})$, then its generator $\mc{L}$  
commutes with the modular operator $\Delta_\si$, and it is self-adjoint with respect to $\l\dd, \dd\r_{\si,f}$ for any $f:(0,\infty) \to (0,\infty)$, i.e., 
\begin{align*}
    \mc{L}\Delta_\si = \Delta_\si \mc{L}\,,\quad   J_\si^f  \mc{L} = \mc{L}^\dag    J_\si^f\,.
\end{align*}
\mb{In consequence, $\mc{P}_t$ also satisfies $\si$-{\rm KMS DBC} and} there holds
\begin{align*}
    \gs \mc{L} = \mc{L}^\dag \gs\,.
\end{align*}
\end{lemma}

The next lemma, due to Alicki \cite{alicki1976detailed}, characterizes the generator $\mc{L}$ of a QMS that satisfies the \mb{$\si$-{\rm GNS DBC}}.


\begin{lemma}\label{lem:struc_gene}
\mb{For a Lindbladian $\mc{L}$ satisfying $\si$-{\rm GNS DBC} for some $\si \in \mc{D}_+(\mc{H})$}, it holds that 
 \begin{align} \label{eq:structure}
        \mc{L}(X) = \sum_{j = 1}^{J} \left(e^{-\omega_j/2}V_j^*[X,V_j] + e^{\omega_j/2}[V_j,X]V^*_j\right)\,,
    \end{align}
with $\ww_j \in \R$ and $J \le d^2 -1$. Here, $V_j \in \bh$, $1\le j \le J$, are trace zero and orthogonal eigenvectors of $\Delta_\si$:
\begin{align} \label{eq:eigmodular}
    \Delta_\si(V_j) =  e^{-\omega_j}V_j\,,\quad \l V_j, V_k\r = c_j\d_{j,k}\,,\quad \tr(V_j) = 0\,,
\end{align}
where $c_j > 0$ are  normalization constants. Moreover, for each $1 \le j \le J$, there exists $1 \le  j' \le J $ such that 
\begin{align} \label{eq:adjoint_index}
    V_j^* = V_{j'}\,,\quad \ww_j = - \ww_{j'}\,.
\end{align}
\end{lemma}

The real numbers $\ww_j \in \R$ are called Bohr frequencies of the Lindbladian $\mc{L}$, which are uniquely determined by the invariant state $\si$, while the operators $V_j$ are only unique up to unitary transformations. \mb{Indeed, let $\{v_k\}$ be orthonormal eigenvectors of $\si$ such that $\si v_k = \si_k v_k$. 
Then, $\ket{v_k}\bra{v_l}$ are eigenvectors of $\Delta_\si$ with $
\Delta_\si \ket{v_k}\bra{v_l} = (\si_k/\si_l) \ket{v_k}\bra{v_l}
$, which implies that for each $1 \le  j \le J$, there exists $1 \le k, l \le d$ such that}
\begin{align} \label{eq:repbohr}
    \mb{\ww_j = \log \si_l - \log \si_k \,.}
\end{align}
In what follows, we will fix a set of $V_j$ for the representation \eqref{eq:structure} with the properties in Lemma \ref{lem:struc_gene}.

\begin{remark} \label{eq:tracial}
When the invariant state $\si$ is the maximally mixed state $\mi/d$, 
$\Delta_\si$ becomes the identity operator and hence $\ww_j = 0$ by \eqref{eq:eigmodular}, and we can take the operators $V_j$ to be self-adjoint. In this symmetric case, the QMS $\mc{P}_t$ may be regarded as a noncommutative heat semigroup, and its generator has the form:
\begin{align} \label{eq:gen_symm}
    \mc{L}(X) = - \sum_j [V_j, [V_j, X]]\,.
\end{align}
\end{remark}

\begin{example} \label{exp:dep_channel}
An important example of QMS is the generalized depolarizing semigroup: 
\begin{align} \label{def:depolt}
    \mc{P}_t(X) = e^{- \gamma t} X + (1 - e^{- \gamma t}) \tr(\si X) \mi\,, \q X \in \bh\,,
\end{align}
for $\si \in \dhh$ and $\gamma > 0$,  which is generated by 
\begin{align} \label{def:depol}
    \mc{L}_{\rm depol}(X) = \gamma  (\tr(\si X) \mi - X)\,.
\end{align}
It is easy to see that $\mc{P}_t$ is primitive and satisfies \mb{$\si$-{\rm GNS DBC}}. 
\end{example}

Lemma \ref{lem:struc_gene} actually gives a very useful first-order differential structure for a  QMS with \mb{$\si$-{\rm GNS DBC}}. We introduce the weighting operator for a full-rank quantum state 
$\si \in \mc{D}_+(\mc{H})$:  
\begin{align*}
    \gs X = \si^{\frac{1}{2}}X\si^{\frac{1}{2}}:\ \bh \to \bh\,,
\end{align*} 
and the noncommutative analog of partial derivatives (associated with $\si \in \mc{D}_+(\mc{H})$): 
\begin{align*}
    \p_j X = [V_j, X]:\ \bh \to \bh\,. 
\end{align*}
Then the noncommutative gradient $\na: \bh \to \bh^J$ and divergence $ \ddiv : \bh^J \to \bh$ can be defined by 
\begin{equation} \label{def:grad}
    \na X = (\p_1 X, \cdots, \p_J X) \q \text{for} \ X \in \bh\,,
\end{equation}
and 
\begin{equation*}
    \ddiv {\bf X} = - \sum_{j = 1}^J \p_j^\dag X_j  \q \text{for} \ {\bf X} \in \bh^J\,,
\end{equation*}
respectively. By definition and \eqref{eq:eigmodular},
it follows that the adjoint of $\p_j$ with respect to $\l \dd, \dd  \r_{\si,1/2}$ is 
\begin{align} \label{eq:adjpjkms}
    \p_{j,\si}^\dag B = \gs^{-1}\p_j^\dag \gs B = e^{-\ww_j/2} V_j^* B - e^{\ww_j/2} B V_j^*\,.  
\end{align}
With the help of $\p^\dag_{j,\si}$ \eqref{eq:adjpjkms}, we can rewrite \eqref{eq:structure} as 
\begin{align} \label{eq:rep_gen}
    \mc{L} (X) = - \sum_{j = 1}^J \p_{j,\si}^\dag \p_j X\,,
\end{align}
and a noncommutative integration by parts formula holds: 
\begin{align} \label{eq:inte_by_parts}
    - \l Y, \mc{L}(X)\r_{\si,1/2} = \sum_{j = 1}^J \l \p_j Y, \p_j X    \r_{\si,1/2} \q \text{for} \ X, Y \in \bh\,.
\end{align}
The following lemma will be useful below, which generalizes \cite[Proposition 4.12]{carlen2020non}. 
\begin{lemma}\label{lem:DBC}
Suppose that $\mc{P}_t$ satisfies \mb{$\si$-{\rm GNS DBC}}. For any continuous function $f:(0,\infty) \to (0,\infty)$, there holds 
\begin{align*}
    - \l Y, \mc{L} X\r_{\si,f} = \sum_{j = 1}^J \left\l \p_j Y,  R_{e^{-\ww_j/2} \si} f\left(L_{e^{\ww_j/2}\si} R^{-1}_{e^{-\ww_j/2} \si}\right) \p_j X \right\r,\q X, Y \in \bh\,.
\end{align*}
\end{lemma}
\begin{proof}
By Stone-Weierstrass theorem, it suffices to consider $f = x^s$. For this, we have 
\begin{align*}
    - \l Y, R_\si \Delta_\si^s \mc{L} X\r = \sum_{j = 1}^J \l Y, \Delta_\si^{s-\frac{1}{2}}\p_j^\dag \gs \p_j X  \r & = \sum_{j = 1}^J  e^{(s-\frac{1}{2})\ww_j} \left\l \p_j Y, R_\si \Delta_\si^s \p_j X \right\r \\
    & = \sum_{j = 1}^J  \left\l \p_j Y, R_{e^{-\ww_j/2}\si} \left(L_{e^{\ww_j/2}\si} R^{-1}_{e^{-\ww_j/2}\si}\right)^s \p_j X \right\r,
\end{align*}
by noting $
    \Delta_\si^s \p_j X = e^{-s\ww_j} \p_j \Delta_\si^s X
$ from \eqref{eq:eigmodular}.
\end{proof}

\subsection{Quantum entropy and Dirichlet form}
This section will introduce quantum relative entropies and Dirichlet forms and discuss their basic properties.

\medskip

\noindent \emph{Noncommutative $L_p$ space.}
We start with the noncommutative weighted $L_p$ space. For \mb{$p \in \R \backslash \{0\}$} and $\si \in \dhh$, we define the $\si$-weighted $p$-\mb{functional} \cite{olkiewicz1999hypercontractivity,kastoryano2013quantum}:
\begin{align*}
    \norm{X}_{p,\si} := \tr\Big(|\Gamma_\si^{1/p}(X)|^p \Big)^{1/p}\,,
\end{align*}
\mb{for $X  \in \mc{B}(\mc{H})$ if $p > 0$; $X \in \mc{B}_+(\mc{H})$ if $p < 0$, which is a norm when $p \ge 1$.} In particular, if $\si = \frac{\mi}{d}$, then $\norm{\dd}_{p,\frac{\mi}{d}} = \frac{1}{d^{1/p}}\norm{\dd}_p$ is the normalized Schatten $p$-norm. We also need the power operator $I_{q,p}$ for $p,q \in \R \backslash \{0\}$:
\begin{align*}
    I_{q,p}(X):= \Gamma_{\si}^{-1/q}\left(|\Gamma_{\si}^{1/p}(X)|^{p/q} \right)\,,
\end{align*} 
\mb{for $X \in \mc{B}(\mc{H})$ if $p/q > 0$; $X \in \mc{B}_+(\mc{H})$ if $p/q < 0$.} Some important properties of $\norm{\dd}_{p,\si}$ and $I_{p,q}$ are summarized in the following lemma; see \cite[Lemmas 1,\,2]{kastoryano2013quantum} and 
\cite[Corollary 5]{beigi2020quantum}. 
\begin{lemma} \label{lem:prop:norm_power}
For $p, q, r \in \R \backslash \{0\}$, it holds that
\begin{enumerate}[1.]
\item $\norm{X}_{p,\si} \le \norm{X}_{q,\si}$ \mb{for $p \le q$ and $X \in \mc{B}_+(\mc{H})$}.
    \item $\norm{I_{q,p}(X)}_{q,\si}^q = \norm{X}_{p,\si}^p$ and $I_{q,r} \circ I_{r,p}(X) = I_{q,p}(X)$ \mb{for $X \in \mc{B}_+(\mc{H})$}.
    \item $I_{p,p}(X) = X$ for $X \ge 0$. 
\end{enumerate}
\end{lemma}

\medskip

\noindent \emph{Dirichlet form.}
We define the $p$-Dirichlet form \mb{($p \in \R \backslash \{0\}$)} for a QMS $\mc{P}_t$ with generator $\mc{L}$  (our definition differs from the one in \cite{kastoryano2013quantum} by a factor of $p/2$):  for any full-rank invariant state $\si = \mc{P}_t^\dag(\si)$,  
\begin{align} \label{def:p_diri}
    \mc{E}_{p,\mc{L}}(X) := - \frac{\h{p}p}{4} \l I_{\h{p},p}(X),\mc{L}(X) \r_{\si,1/2}\,,
\end{align}
\mb{for $X  \in \mc{B}(\mc{H})$ if $p > 1$; $X \in \mc{B}_+(\mc{H})$ if $p<1$}.  In particular, for $p = 2$, we have 
\begin{align*}
    \mc{E}_{2,\mc{L}}(X) = - \l X, \mc{L}(X) \r_{\si,1/2}\,, \q X \ge 0\,.
\end{align*}
The case $p = 1$ is defined by the limit $p \to 1$ \cite[Proposition  8]{kastoryano2013quantum}:
\begin{align} \label{def:1diri}
  \mc{E}_{1, \mc{L}}(X) := \lim_{p \to 1} \mc{E}_{p,\mc{L}}(X) = - \frac{1}{4} \l \log \gs(X) - \log \si, \mc{L}(X) \r_{\si,1/2}\,,\q X > 0\,.
\end{align}

\begin{lemma} 
  Let $\mc{L}$ be a Lindbladian satisfying $\si$-{\rm GNS DBC} for some $\si \in \mc{D}_+(\mc{H})$. Then, we have, for $p \in \R \backslash \{0\}$,  
\begin{align} \label{eq:rep_epl}
        \mc{E}_{p,\mc{L}}(X)
        & = \frac{p^2}{4} \sum_{j = 1}^J \left\l \gs^{1/p} \big(\p_j X\big), f_p^{[1]}\left(e^{\ww_j/2p}\gs^{1/p}(X), e^{-\ww_j/2p}\gs^{1/p}(X)\right)  \gs^{1/p} \big(\p_j X \big)  \right\r\,, \q X > 0\,,
\end{align}
\mb{where $f_p^{[1]}(\dd,\dd)$ is the double sum operator \eqref{def:doutble_op_sum} associated with the divided difference \eqref{def:divi_diff} of the function: }
\begin{equation} \label{def:funfp}
\mb{f_p(x): = \begin{cases}
    \frac{1}{p-1} x^{p-1} & \text{if} \ p \neq 1\,, \\
    \log x & \text{if} \ p = 1\,.
\end{cases}}
\end{equation} 
\end{lemma}

\begin{proof}
\mb{It is sufficient to show \eqref{eq:rep_epl} for $p \in \R \backslash \{0,1\}$, since the case $p = 1$ can be obtained by taking the limit $p \to 1$.} By formula \eqref{eq:inte_by_parts}, we have 
\begin{align} \label{eq:symm_p_diri}
      \mc{E}_{p,\mc{L}}(X) & = \frac{\h{p}p}{4} \sum_{j = 1}^J \l \p_j I_{\h{p},p}(X), \p_j X \r_{\si,1/2}\,.
\end{align}
Thanks to the relation \eqref{eq:eigmodular}, we can further compute, for $p,q \neq 0$, 
\begin{align} \label{eq:partialpower_1}
    \p_j I_{q,p}(X) & = V_j \Gamma_\si^{-1/q}\left(|\Gamma_\si^{1/p}(X)|^{p/q} \right) - \Gamma_\si^{-1/q}\left(|\Gamma_\si^{1/p}(X)|^{p/q} \right) V_j \notag \\
    & = \Gamma_\si^{-1/q} \left(  V_j |\Gamma_\si^{1/p}(e^{-\ww_j/2p}X)|^{p/q}  -  |\Gamma_\si^{1/p}(e^{\ww_j/2p}X)|^{p/q}  V_j \right)\,,
\end{align} 
and, for $s \neq 0$,  
\begin{align} \label{eq:partialpower_2}
\p_j X =  \Gamma_\si^{-1/s}\left(V_j \Gamma_\si^{1/s}(e^{-\ww_j/2s}X) - \Gamma_\si^{1/s}(e^{\ww_j/2s}X)V_j\right)\,.
\end{align}
Then, substituting the formulas \eqref{eq:partialpower_1} with $q=\h{p}$ and \eqref{eq:partialpower_2} with $s = p$ back into \eqref{eq:symm_p_diri}, and using Lemma \ref{lem:chain_rule}, we can find the desired representation \eqref{eq:rep_epl} of $\mc{E}_{p,\mc{L}}$. 
\end{proof}

\begin{corollary} \label{cor:posi_dirichlet} 
Let $\mc{L}$ be a Lindbladian satisfying $\si$-{\rm GNS DBC} for some $\si \in \mc{D}_+(\mc{H})$. For any $X > 0$ and $p \neq 0$, there holds $\mc{E}_{p,\mc{L}}(X) \ge 0$. Moreover, the equality $\mc{E}_{p,\mc{L}}(X) = 0$ holds if and only if $\na X = 0$,  where $\na$ is given in \eqref{def:grad}.
\end{corollary}
\begin{proof}
   \mb{It suffices to observe that $f_p(x)$ in \eqref{def:funfp} is increasing for $x > 0$, and hence the divided difference $f_p^{[1]}$ is strictly positive on $(0,\infty) \t (0,\infty)$, which, by Lemma \ref{lem:double_inner}, implies that the operator 
    $f_p^{[1]}(e^{\ww_j/2p}\gs^{1/p}(X), e^{-\ww_j/2p}\gs^{1/p}(X))$ is positive definite on $\mc{B}(\mc{H})$. Then, 
     the representation \eqref{eq:rep_epl} readily gives $\mc{E}_{p,\mc{L}}(X) \ge 0$, and $\mc{E}_{p,\mc{L}}(X) = 0$ if and only if $\p_j X = 0$ for any $1 \le j \le J$.}
\end{proof}

We derive integral representation formulas of $f_p^{[1]}$ for later use.  We recall \cite[p.116]{bhatia2013matrix}
\begin{align*}
    x^{p - 1} = \frac{\sin(p \pi)}{\pi}\int_0^\infty \frac{ s^{p-1}}{s + x}\  ds\,, \q \text{for}\ x > 0\,,\ 0 < p < 1\,,
\end{align*}
which yields, for $1 < p < 2$,  
\begin{align} \label{eq:inte_divi_p}
f_p^{[1]}(x,y) & = \frac{1}{x - y} \int_x^y t^{p - 2}\ dt  \notag \\
  &  = \frac{1}{x - y} \frac{\sin((p-1) \pi)}{\pi} \int_x^y \int_0^\infty \frac{s^{p-2}}{s + t}\  d s  dt  \notag \\
  &  =  \frac{\sin((p - 1) \pi)}{\pi} \int_0^\infty s^{p-2} g_s^{[1]}(x,y)\ d s d t \,,
\end{align}
where $g_s(x) = \log(x + s)$ and $g_s^{[1]}$ is the associated divided difference. 
By the integral form of $g_0^{[1]}$:
\begin{align*}
g_0^{[1]}(x,y) = \int_0^\infty \frac{1}{(t + x)(t + y)} dt\,,
\end{align*}
we further have, from the formula \eqref{eq:inte_divi_p},  
\begin{align} \label{eq:integral_thetap}
    f_p^{[1]}(x,y)
  &  =  \frac{\sin((p-1) \pi)}{\pi} \int_0^\infty  s^{p-2} \int_0^\infty \frac{1}{(t + s + x)(t + s + y)}\ dt ds\,.
\end{align}

We next recall the comparison result for Dirichlet forms $\mc{E}_{p,\mc{L}}$, known as the quantum Stroock-Varopoulos inequality, which was proved in \cite[Theorem 14]{beigi2020quantum}. 
 \begin{lemma} \label{lem:qvs}
Let $\mc{L}$ be the generator of a QMS $\mc{P}_t$ satisfying \mb{$\si$-{\rm GNS DBC}} for some invariant state $\si \in \mc{D}_+(\mc{H})$, and $\mc{E}_{p,\mc{L}}$ be defined in \eqref{def:p_diri}. Then, for $X > 0$, we have 
\begin{equation} \label{eq:quantumsv}
    \mc{E}_{p,\mc{L}}(I_{p,2}(X)) \ge \mc{E}_{q,\mc{L}}(I_{q,2}(X))\,, \q 0 < p \le q \le 2\,.
\end{equation}
\end{lemma}
In the special case $p \ge 1$ and $q = 2$, Lemma \ref{lem:qvs} gives the strong $L_p$ regularity of the Dirichlet form \cite{kastoryano2013quantum,bardet2017estimating}, which we slightly generalize as follows. The proof follows from the basic inequality: for $p \in (1,2]$, $a,b > 0$,
\begin{align*}
    (a-b)(a^{p-1} - b^{p-1}) \le \left( a^{p/2} - b^{p/2}\right)^2 \le \frac{p^2}{4(p-1)}(a-b)(a^{p-1} - b^{p-1})\,,
\end{align*}
and similar arguments in \cite[Theorem 4.1]{bardet2017estimating}. Hence we omit it here. 

\begin{corollary} \label{lem:com_diri}
Under the same assumptions as in Lemma \ref{lem:qvs}, it holds that, for $p \in (1,2]$ and $X \ge 0$, 
\begin{align} \label{eq:lpreg}
    \mc{E}_{2, \mc{L}}(I_{2,p}(X)) \le  \mc{E}_{p, \mc{L}}(X) \le \frac{p^2}{4(p-1)} \mc{E}_{2, \mc{L}}(I_{2,p}(X))\,.
\end{align}
In particular, the lower inequality in \eqref{eq:lpreg} holds for all $p > 0$. 
\end{corollary}

\medskip

\noindent \emph{Relative entropy.} We now introduce the entropy function $\Ent_{p,\si}(X)$, \mb{for $p \in \R \backslash \{0\}$} and $\si \in \mc{D}_+(\mc{H})$, as follows \cite{beigi2020quantum} (our definition differs from the one in \cite{olkiewicz1999hypercontractivity,kastoryano2013quantum} by a factor of $p$): 
\begin{align*}
    \Ent_{p,\si}(X) := \tr \left( \left(\gs^{1/p}(X)\right)^p \left(\log \left(\gs^{1/p}(X) \right)^p - \log \si \right) \right) - \norm{X}_{p,\si}^p \log \norm{X}_{p,\si}^p\,,\q X > 0\,,
\end{align*}
and Umegaki's relative entropy:
\begin{equation} \label{def:relative_entropy}
    D(\rho \| \si) = \tr(\rho \log \rho - \rho \log \si)\,,\q \rho \in \dh\,.
\end{equation}
We recall from \cite[Proposition 3]{beigi2020quantum} that \mb{for $p \neq 0$ and invertible $Y \in \mc{B}_{sa}(\mc{H})$}, 
\begin{align} \label{eq:differ_norm}
\frac{d}{d p} \norm{Y}_{p,\si} = \frac{1}{p^2} \norm{Y}_{p,\si}^{1-p} 
\Ent_{p,\si}(I_{p,p}(Y))\,,
\end{align} 
which, by chain rule, implies
\cite[Theorem 2.7]{olkiewicz1999hypercontractivity} 
\begin{equation} \label{eq:dev_pnorm}
    \frac{d}{d p} \norm{Y}_{p,\si}^p = \tr \Big((\Gamma_\si^{1/p}(Y))^p (\log \Gamma_\si^{1/p}(Y) - \frac{1}{p} \log \si ) \Big)\,,\q Y > 0\,.
\end{equation}
The above formulas relate the differential of $L_p$-norm $\norm{\dd}_{p,\si}$ and the entropy function $\Ent_{p,\si}$. The following lemma provides some basic properties of $\Ent_{p,\si}$; see \cite{kastoryano2013quantum,beigi2020quantum}. 

\begin{lemma} \label{lem:ent_rela}
For all $X > 0$ and $p \neq 0$, we have $\Ent_{p,\si}(X) \ge 0$. Moreover, for a state $\rho \in \mc{D}(\mc{H})$, it holds that 
\begin{align*}
    \Ent_{2,\si}\big(\gs^{-1/2}\left(\sqrt{\rho}\right)\big) = D(\rho || \si) \q \text{and}\q  \Ent_{1,\si}\big(\gs^{-1}(\rho) \big) = D(\rho || \si)\,.
\end{align*}
\end{lemma}

We also recall the sandwiched R\'{e}nyi relative entropy 
introduced in \cite{muller2013quantum,wilde2014strong}: for $p \in (0,1) \cup (1,\infty)$,
\begin{align} \label{def:sandwi}
   D_p(\rho \| \si) := 
    \h{p}\log\big(\norm{\Gamma_\si^{-1}(\rho)}_{p,\si}\big), \q  \rho \in \dh\,,
\end{align}
and 
\begin{align} \label{def:maxentropy}
    D(\rho\|\si) = \lim_{p \to 1} D_p(\rho\|\si)\,, \q D_{\infty}(\rho \| \si) := \log \inf\{c > 0\,;\ \rho \le c \si
\} = \lim_{p \to \infty} D_p(\rho \| \si)\,,
\end{align}
where $ D(\rho\|\si)$ is given in \eqref{def:relative_entropy}, and $D_\infty(\rho\|\si)$ is the max-relative entropy \cite{datta2009min,muller2013quantum}. \cite[Lemma 2.1]{muller2018sandwiched} shows that for $\si \in \dhh$ and $p \in [1,\infty]$, the sandwiched R\'{e}nyi entropy $D_p$ satisfies
\begin{align} \label{eq:suprhodp}
\sup_{\rho \in \dh} D_p(\rho \| \si) = \log \si_{\min}^{-1}\,.
\end{align}
\mb{Here and throughout this work, we denote by $\si_{\min}$ the minimal eigenvalue of a state $\si \in \dhh$.}

We finally recall the Araki–Lieb–Thirring (ALT) inequality \cite{lieb1997inequalities,araki1990inequality}, which is also useful in the sequel.
\begin{lemma}
For any $A \ge 0$, $B \ge 0$, and $q \ge 0$, it holds that 
\begin{align} \label{eq_alt_2}
    \tr( (B^r A^r B^r)^q) \le \tr( (BAB)^{rq})\,,\q  0\le r \le 1\,. 
\end{align}
\end{lemma}

\section{Quantum interpolation functional inequalities}  \label{sec:beck}
In this section, we shall introduce and 
investigate two new families of quantum functional inequalities:  quantum $p$-Beckner's inequality and quantum dual $q$-Beckner's inequality, which interpolate quantum Sobolev-type inequalities and Poincar\'{e} inequality. 

\subsection{Quantum divergence} For our purpose, we need a modified sandwiched R\'{e}nyi entropy, called \emph{quantum $p$-divergence}, defined by, for $\si \in \dhh$ and \mb{$p \in \R \backslash \{0,1\}$}, 
\begin{align} \label{def:quanpdivi}
    \mc{F}_{p,\si}(\rho) := \frac{1}{p(p-1)} \big(\norm{\Gamma_\si^{-1}(\rho)}_{p,\si}^p - 1 \big)\,,
\end{align}
\mb{where $\rho \in \dh$ if $p \ge 0$; $\rho \in \dhh$ if $p < 0$}. By definition \eqref{def:sandwi}, it follows that
\begin{equation} \label{eq:rela_sandpdivi}
     \mc{F}_{p,\si}(\rho) = \frac{1}{p(p-1)}  \big( \exp((p-1)D_p(\rho\|\si)) - 1 \big)\,,
\end{equation}
and then, by \eqref{def:maxentropy}, we have
\begin{align*}
  \mc{F}_{1,\si}(\rho): = \lim_{p \to 1} \mc{F}_{p,\si}(\rho)  = D(\rho\| \si)\,.
\end{align*}
The operator $\Gamma_\si^{-1}(\rho)$ can be viewed as the relative density of $\rho \in \dh$ with respect to the reference state $\si \in \dhh$. For convenience, we refer to operators $X \ge 0$ with $\norm{X}_{1,\si} = 1$ as relative densities with respect to $\si$ in what follows. We recall that the variance of $X \in \mc{B}_{sa}(\mc{H})$ is defined by 
\begin{align*}
    \Var_\si(X): = \norm{X - \tr(\si X)}_{2,\si}^2 = \norm{X}_{2,\si}^2 - \norm{X}_{1,\si}^2\,.
\end{align*}
Clearly, when $p = 2$, $\mc{F}_{p,\si}(\rho)$ reduces to the variance of the relative density of $\rho$, up to a constant factor,
\begin{align*}
    \mc{F}_{2,\si}(\rho) = \frac{1}{2} \Var_{\si}(X)\,, \q X = \gs^{-1} (\rho)\,.
\end{align*}
Thanks to \eqref{eq:rela_sandpdivi}, many properties of $D_p(\rho\|\si)$ can be directly translated to $\mc{F}_{p,\si}(\rho)$.

\begin{lemma} \label{lem:prop_divi}
For any \mb{$\rho\,,\si \in \dhh$}, we have 
\begin{enumerate}[1.]
    \item $\mc{F}_{p,\si}(\rho) \ge 0$, and $ \mc{F}_{p,\si}(\rho) = 0$ if and only if $\rho = \si$, \mb{for $p \in \R \backslash \{0\}$}.  
    \item $\mc{F}_{p,\si}(\rho)$ is jointly convex with respect to $(\rho,\si)$, \mb{for $p \in \R \backslash (-1,1/2)$}. 
    \item The data processing inequality holds for $ \mc{F}_{p,\si}(\rho)$ with \mb{$p \in \R \backslash (-1,1/2)$}, 
    \begin{align} \label{eq:dpi_divi}
        \mc{F}_{p,\Phi(\si)}(\Phi(\rho)) \le \mc{F}_{p,\si}(\rho)\,,
    \end{align}
    where $\Phi$ is a quantum channel such that $\mc{F}_{p,\Phi(\si)}(\Phi(\rho))$ is well-defined.
\end{enumerate}
\end{lemma}

\begin{proof}\mb{
It was proved in \cite[Theorem 5]{beigi2013sandwiched} that $\norm{\Gamma_\si^{-1}(\rho)}_{p,\si} \ge 1$ for $p > 1$, and $\norm{\Gamma_\si^{-1}(\rho)}_{p,\si} \le 1$ for $p < 1$, with the equality condition $\rho = \si$. Hence, the first statement follows. For the second and third statements, the case $p \ge 1/2$ has been discussed in  \cite[Theorems 1,\,2]{frank2013monotonicity}. We now consider the case $p \le -1$, where the joint convexity of $\mc{F}_{p,\si}(\rho)$ is a special case of the general result \cite[Theorem 1.1]{zhang2020wigner}. Then, the data processing inequality \eqref{eq:dpi_divi} can be derived by the standard argument; see \cite{frank2013monotonicity}. }
\end{proof}

We next extend the key estimates in \cite[Lemmas 2.1,\,2.2]{gao2021complete} for the relative entropy $D(\rho \| \si)$ to our $p$-divergence \eqref{def:quanpdivi}, which are new and will be useful in the sequel.

\begin{lemma}\label{lem:mono_norm}
Let operators $X_i, Y_i > 0$, $ i = 1,2$, satisfy $X_i \le c Y_i$ for some $c > 0$. It holds that, \mb{for $p \in (1,2)$},
\begin{align} \label{eq:mono_norm_two}
\l A, f_p^{[1]}(Y_1, Y_2) A\r \le c^{2-p} \l A , f_p^{[1]}( X_1,  X_2)  A \r\,, \q \forall A \in \bh\,,
\end{align}
where $f_p^{[1]}$ is the divided difference of the function $f_p$ defined in \eqref{def:funfp}. 
\end{lemma}

\begin{proof}
The proof follows from the integral form \eqref{eq:integral_thetap} of $f_p^{[1]}$: 
\begin{align} \label{eq:auxest_mono} 
 \l A,  f_p^{[1]}(X_1, X_2) A\r
  &  =  \frac{\sin((p-1) \pi)}{\pi} \int_0^\infty  s^{p-2} \int_0^\infty \tr \left( A^* \frac{1}{t + s + X_1} A \frac{1}{t + s + X_2} \right) dt ds \notag \\
 & \ge  \frac{\sin((p-1) \pi)}{\pi} \int_0^\infty  s^{p-2} \int_0^\infty \tr \left(A^* \frac{1}{t + s +  c Y_1} A \frac{1}{t + s + c Y_2} \right) dt ds \notag \\
  & =  \frac{\sin((p-1) \pi)}{\pi} \int_0^\infty  s^{p-2} \frac{1}{c^2} \int_0^\infty \tr \left(A^* \frac{1}{\frac{t + s}{c} +  Y_1} A \frac{1}{\frac{t + s}{c} +  Y_2} \right) dt ds \notag \\
   & =  \frac{\sin((p-1) \pi)}{\pi} \int_0^\infty  (cs)^{p-2} \int_0^\infty \tr \left(A^* \frac{1}{t + s + Y_1} A \frac{1}{t + s + Y_2} \right) dt ds \\
   & = c^{p-2}  \l A,  f_p^{[1]}(Y_1, Y_2) A \r\,,\notag
\end{align}
where in the first inequality we have used the fact that $t^{-1}$ is operator monotone decreasing; in the third line we have used the change of variable $r \to r/s$ and $t \to t/s$. 
\end{proof}

Before we proceed, several interesting and helpful observations are in order. First, we define the function
\begin{align*}
    \vp_p(x) = \frac{1}{p-1}\frac{x - x^{1/p}}{x^{1/p} - 1}\,,
\end{align*}
\mb{for $x > 0$ and $p \in \R \backslash \{0,1\}$}.  
Then the kernel $J_\si^{\vp_p} = \vp_p(\Delta_\si) R_\si $ for the inner product $\l \dd, \dd\r_{\si,\vp_p}$ (cf.\,\eqref{def:operator_kernel} and \eqref{eq:general_inner}) can be reformulated as follows, in terms of the double sum operator associated with $f_p^{[1]}$, 
\begin{align} \label{eq:kernel_iden}
   \vp_p(\Delta_\si) R_\si & = \frac{1}{p-1} \frac{\Delta_{\si} - \Delta_\si^{1/p}}{\Delta_{\si}^{1/p} - 1} R_\si \notag \\
   & = \frac{1}{p-1} \frac{L_{\si}^{(p-1)/p} - R^{(p-1)/p}_\si}{L_{\si}^{1/p} - R^{1/p}_\si}L_\si^{1/p}R_\si^{1/p} \notag \\ & = \gs^{1/p} f_p^{[1]}\big(\si^{1/p},\si^{1/p}\big)\gs^{1/p}\,.
\end{align}
Second, the inner product $\l \dd , \dd  \r_{\si,\vp_p}$ \mb{with $p \in \R \backslash \{0,1\}$} also relates to the $\chi^2$-divergence \eqref{def:quantumchi} corresponding to the power difference $\kappa_\alpha$ \eqref{def:pdiff} \mb{with $\alpha \in \R \backslash \{0,1\}$}. Indeed, for $\rho \in \dh$ with the relative density $X = \gs^{-1}(\rho)$, we can directly compute 
\begin{equation*}
  \norm{X - \mi}_{\si,\vp_p}^2 = \big\l \rho - \si, \gs^{-1}\vp_p(\Delta_\si)R_\si \gs^{-1}(\rho - \si)\big\r = \big\l \rho-\si, \Omega_\si^{x^{-1}\vp_p}(\rho - \si)\big\r\,,
\end{equation*}
by noting
\begin{align}  \label{eq:kernel_iden_2}
    \gs^{-1}\vp_p(\Delta_\si)R_\si \gs^{-1} = \Delta_\si^{-1}\vp_p(\Delta_\si)R_\si^{-1} = \Omega_\si^{x^{-1}\vp_p}\,,
\end{align}
where the operator $\Omega_\si^{x^{-1}\vp_p}$ is defined as in \eqref{def:omega}.  Then we find, by letting $\alpha = 1/p$ and \eqref{def:pdiff}, 
\begin{align} \label{auxeq_power_diff}
    x^{-1}\vp_{1/\alpha}(x) = \frac{1}{p-1}\frac{1-x^{(1-p)/p}}{x^{1/p}-1} = \frac{\alpha}{\alpha-1} \frac{x^{\alpha-1} - 1}{x^{\alpha} - 1} = \kappa_\alpha\,.
\end{align}
With the notation defined above, we conclude
\begin{align} \label{eq:chi_heisen}
    \norm{X - \mi}_{\si,\vp_p}^2 = \chi^2_{\kappa_\alpha}(\rho,\si)\,.
\end{align}
Third, by Lemma \ref{lem:DBC} \mb{with $f = \vp_p$ and the same calculation as in \eqref{eq:kernel_iden}}, it holds that, \mb{for $p \in \R \backslash \{0,1\}$}, 
\begin{align} \label{eq:rep_dbcl}
  \mb{ - \left\l X, \mc{L} X \right\r_{\si,\vp_p} 
    = \sum_{j = 1}^J  \left\l \gs^{1/p}  \big(\p_j X\big), f_p^{[1]}\left(e^{\ww_j/2p}\si^{1/p}, e^{-\ww_j/2p}\si^{1/p}\right) \gs^{1/p}\big(\p_j X\big) \right\r,\q X \in \bh\,.}
\end{align}

We are ready to give the following lemma, which shows that 
the $p$-divergence $\mc{F}_{p,\si}(\rho)$ can be bounded from above and below by the $\chi^2$-divergence associated with the power difference $\kappa_{1/p}$.

\begin{lemma} \label{lem:two-sided}
For $\rho, \si \in \dhh$ satisfying $\rho \le c \si$ for some $c > 0$, it holds that, \mb{for $p \in (1,2)$,}
\begin{equation} \label{eq:two_sided_dense}
    k_p(c) \chi_{\kappa_{1/p}}^2(\rho,\si) \le \mc{F}_{p,\si}(\rho) \le p^{-1} \chi_{\kappa_{1/p}}^2(\rho,\si)\,,
\end{equation}
where the constant $k_p(c)$ is given by 
\begin{align*}
    k_p(c) =  \frac{c^p - 1 - p(c-1)}{p(c-1)^2 (p-1)}\,.
\end{align*}
In particular, the upper bound estimate in \eqref{eq:two_sided_dense} holds for any $\rho \in \dhh$. 
\end{lemma}

\begin{proof}
Recalling the relation \eqref{eq:chi_heisen}, we will prove the inequality \eqref{eq:two_sided_dense} in the Heisenberg picture:
\begin{align} \label{eq:two_sided_est}
 k_p(c) \norm{X - \mi}^2_{\si,\vp_p}   \le  \frac{1}{p(p-1)}
 \left(\norm{X}_{p,\si}^p - 1\right) \le p^{-1} \norm{X - \mi}^2_{\si,\vp_p},
\end{align}
for $X > 0$ satisfying $\tr(\si X) = 1$ and $X \le c \mi$. 
We define $X_t = (1 - t) \mi + t X$, $t \in [0,1]$, and consider the function: 
\begin{equation*}
    \vp(t) := \frac{1}{p(p-1)} \left(\norm{X_t}_{p,\si}^p - 1 \right). 
\end{equation*}
It is easy to compute the derivatives:
\begin{align*}
    \vp'(t) =  \frac{1}{p-1} \tr \left( \left( \gs^{1/p}(X_t)\right)^{p-1} \gs^{1/p}\left(X-\mi\right) \right)\,,
\end{align*}
and 
\begin{align} \label{auxeq:2nd_der}
    \vp''(t) = \left\l \gs^{1/p}(X-\mi), f_p^{[1]}\left(\gs^{1/p}(X_t),\gs^{1/p}(X_t) \right) \gs^{1/p}(X-\mi) \right\r. 
\end{align}
By assumption $X \le c \mi$, we have
\begin{align} \label{aauxeq}
(1 - t) \mi   \le X_t \le (1 + (c - 1) t) \mi\,,
\end{align}
and hence $(1-t) \si^{1/p}\le \gs^{1/p}(X_t) \le (1 + (c - 1) t) \si^{1/p}$. 
Then applying Lemma \ref{lem:mono_norm} to \eqref{auxeq:2nd_der} gives 
\begin{align*}
(1 + (c-1)t)^{p-2} \left \l X -\mi, J_\si^{\vp_p} (X -\mi) \right\r  \le \vp''(t) \le (1 - t)^{p-2} \left \l X -\mi, J_\si^{\vp_p} (X -\mi) \right\r,
\end{align*}
where we have used \eqref{aauxeq} and the observation \eqref{eq:kernel_iden}. 
Note from $\tr (\si X) = 1$ that $\vp'(0)= 0$. It follows that 
\begin{align*}
    \vp(1) - \vp(0) = \int_0^1 \int_0^x \vp''(t) dt dx & \le  \Big( \int_0^1 \int_0^x  (1 - t)^{p-2}\ dt dx \Big) \big\l X -\mi, J_\si^{\vp_p} (X -\mi) \big\r \\
    & \le  p^{-1} \norm{X - \mi}^2_{\si,\vp_p}\,.
\end{align*}
Similarly, for the lower bound, we have 
\begin{align*}
    \vp(1) - \vp(0) & \ge   \Big(\int_0^1 \int_0^x (1 + (c - 1) t)^{p-2} \ dt dx \Big) \big \l X -\mi, J_\si^{\vp_p} (X -\mi) \big\r \\
    & \ge \frac{c^p - 1 - p(c-1)}{p(c-1)^2 (p-1)} \norm{X - \mi}^2_{\si,\vp_p}. 
\end{align*}
The proof is completed by noting $\vp(0) = 0$. 
\end{proof}

\begin{remark} \label{rem:limit}
   \mb{Lemmas \ref{lem:mono_norm} and \ref{lem:two-sided} above rely on the integral formula \eqref{eq:integral_thetap} of $f^{[1]}_p$, which limit the range of $p$ to $(1,2)$. The cases $p = 1$ and $p =2$ can be easily obtained by taking the limit. In particular, when $p = 1$, we recover the the results in \cite[Lemmas 2.1,\,2.2]{gao2021complete}, while, when $p = 2$, the estimates in \eqref{eq:mono_norm_two} and \eqref{eq:two_sided_dense} are trivial.}
\end{remark}

\subsection{Interpolation between Sobolev and Poincar\'{e} inequalities} \label{sec:inter_primi}

To motivate the quantum Beckner's inequalities, we consider a primitive QMS $\mc{P}_t$ \mb{satisfying $\si$-KMS DBC for some $\si \in \dhh$}. 
By the limit \eqref{eq:conver_qms} and the data processing inequality \eqref{eq:dpi_divi},  $\mc{F}_{p,\si}$ decreases along the dynamic $\rho_t = \mc{P}_t^{\sss \dag}(\rho)$ and there holds $\mc{F}_{p,\si}(\rho_t) \to 0$ as $t \to \infty$. 
We will see that the quantum Beckner's inequality characterizes the convergence rate of $\mc{F}_{p,\si}(\rho_t)$. It is convenient to consider the evolution of the relative density $X = \gs^{-1}(\rho)$, i.e., the QMS in the 
Heisenberg picture. By \mb{KMS DBC} and Lemma \ref{lem:self_adjoint}, we have
\begin{align} \label{eq:evo_density}
    X_t := \gs^{-1}(\rho_t) =  \gs^{-1} e^{t \mc{L}^\dag} \gs (X) = \mc{P}_t(X)\,.
\end{align}
Then we compute the time derivative of $\mc{F}_{p,\si}(\rho_t)$ as follows, \mb{for $p \in \R \backslash (-1,1/2)$,}
\begin{align} \label{eq:ep_divi}
    \frac{d}{d t} \mc{F}_{p,\si}(\rho_t) & = \frac{1}{p-1}
 ((\gs^{-1/\h{p}}(\rho_t))^{p-1}), \mc{L}^\dag(\rho_t)\r \notag \\
    & = \frac{1}{p-1} \l \gs^{-1/\h{p}}(\gs^{1/p}(X_t))^{p-1}),  \gs\mc{L}(X_t)\r \notag \\ 
    & = - \frac{4}{p^2}\mc{E}_{p,\mc{L}}(X_t) \le 0\,,
\end{align}
\mb{where $\rho_t = \mc{P}_t^{\dag}(\rho) > 0$ for any $ t \ge 0$ with $\rho \in \dhh$, and last inequality is by \eqref{eq:dpi_divi}.
If $\mc{P}_t$ satisfies the $\si$-{\rm GNS DBC}, by 
the positivity of $\mc{E}_{p,\mc{L}}$ in Corollary \ref{cor:posi_dirichlet}, the calculation \eqref{eq:ep_divi} implies the contractivity of $\mc{F}_{p,\si}$ along the dynamic for all $p \in \R \backslash \{0,1\}$.} In view of \eqref{eq:ep_divi}, we can call
$- 4\,\mc{E}_{p,\mc{L}}(\gs^{-1}(\rho))/p^2$ the entropy production of the QMS $\mc{P}_t = e^{t \mc{L}}$ for the $p$-divergence $\mc{F}_{p,\si}$. A simple use of Gr\"{o}nwall’s inequality gives the equivalence between the exponential decay of $\mc{F}_{p,\si}$: for some $\alpha_p > 0$,
\begin{align} \label{eq:exp_divi}
    \mc{F}_{p,\si}(\rho_t) \le \mb{e^{- 4 \alpha_p t}}  \mc{F}_{p,\si}(\rho)\,, \q \forall\rho \in \dhh\,,
\end{align}
and the functional inequality:
\begin{equation} \label{eq_beck}
    \alpha_p \mc{F}_{p,\si}(\rho) \le \mb{p^{-2}} \mc{E}_{p,\mc{L}}(\gs^{-1}(\rho))\,, \q \forall\rho \in \dhh\,,
\end{equation}
which we call the quantum Beckner's inequality. By analogy with the classical case \eqref{ineq_e}, we can also easily define 
the quantum dual Beckner's inequality; see \eqref{ineq_dbecq} below.

Let us now formally define the quantum functional inequalities in the Heisenberg picture, which will be the main focus in the following sections. 
\begin{definition} \label{def:quantum_func}
For a primitive QMS $\mc{P}_t$ with generator $\mc{L}$ \mb{satisfying $\si$-{\rm KMS DBC} for some $\si \in \dhh$}, let $\mc{E}_{p,\mc{L}}$ be the associated $p$-Dirichlet form \eqref{def:p_diri}. Then we say that $\mc{P}_t$ satisfies:
\begin{enumerate}[1.]
    \item the Poincar\'{e} inequality if there exists a constant $\lad > 0$ such that for all $X \in \mc{B}(\mc{H})$,
    \begin{equation} \label{ineq_pi2}
        \lad \Var_\si(X) \le \mc{E}_{2,\mc{L}}(X) \,. \tag{PI}
    \end{equation}
    \item  the modified log-Sobolev inequality (MLSI) if there exists $\alpha_1 > 0$ such that for all $X \ge 0$, 
\begin{equation} \label{ineq_mlsi}
    \alpha_1 \Ent_{1,\si}(X) \le \mc{E}_{1,\mc{L}}(X)\,. \tag{mLSI} 
\end{equation}
    \item the $p$-Beckner's inequality with \mb{$p \in \R \backslash \{0,1\}$} if there exists $\alpha_p > 0$ such that for all $X > 0$,
    \begin{align} \label{ineq_becp}
       \mb{ \alpha_p\,\h{p}\, (\norm{X}^p_{p,\si} - \norm{X}_{1,\si}^p) \le  \mc{E}_{p,\mc{L}}(X)\,. }\tag{Bec-$p$}
    \end{align}
    \item the log-Sobolev inequality (LSI) if there exists  $\beta > 0$ such that for all $Y \ge 0$,  
    \begin{align} \label{ineq_lsi}
        \beta \Ent_{2,\si}(Y) \le \mc{E}_{2,\mc{L}}(Y)\,.  \tag{LSI}
    \end{align} 
    \item the dual $q$-Beckner's inequality with \mb{$q \in (0,2)$} if there exists  $\beta_q > 0$ such that for all $Y \ge 0$,
    \begin{align}  \label{ineq_dbecq}
        \beta_q \Var_{q,\si}(Y)\le (2-q) \mc{E}_{2,\mc{L}}(Y)\,,  \tag{Bec'-$q$}
    \end{align}
    where $\Var_{q,\si}(Y)$ is the $q$-variance: $\Var_{q,\si}(Y): = \norm{Y}_{2,\si}^2 - \norm{Y}_{q,\si}^2$.
\end{enumerate}
\end{definition}

\begin{remark} 
Note that inequalities in Definition \ref{def:quantum_func} can be easily reformulated in the Schr\"{o}dinger picture by inserting $X = \gs^{-1}(\rho)$. For instance, by definition \eqref{def:quanpdivi} of $\mc{F}_{p,\si}(\rho)$, \eqref{ineq_becp} is clearly equivalent to \eqref{eq_beck}, while, by Lemma \ref{lem:ent_rela}, \eqref{ineq_mlsi} with $X = \gs^{-1}(\rho)$ gives the familiar one in terms of quantum states \cite{rouze2019concentration}:
\begin{equation*}
    \alpha_1 D(\rho \| \si) \le - \frac{1}{4} \tr \left(\mc{L}^\dag(\rho) (\log \rho - \log \si) \right)\,, \q \forall \rho \in \dh\,.
\end{equation*}
\end{remark}

\begin{remark}\mb{
Recalling the classical cases \eqref{ineq_c} and \eqref{ineq_e}, the most interesting regimes for \eqref{ineq_becp} and \eqref{ineq_dbecq} are $p \in (1,2]$ and $q \in [1,2)$, respectively. Since many arguments  in the following discussions can work beyond these ranges, here we choose to define \eqref{ineq_becp} and \eqref{ineq_dbecq} for general $p$ and $q$.}
\end{remark}

We call the optimal constant in \eqref{ineq_becp}
the quantum Beckner constant, denoted by $\alpha_p(\mc{L})$. Similar 
notions apply to other functional inequalities defined above. In particular, the Poincar\'{e} constant $\lad(\mc{L})$ is nothing  but the spectral gap of $\mc{L}$. Indeed, since $\mc{L}$ is self-adjoint with respect to $\l\dd,\dd\r_{\si,1/2}$ and $\ker({\mc{L}}) = {\rm span}\{ \mi_{\mc{H}} \}$ holds for a primitive QMS, 
by the min-max theorem, the Poincar\'{e} constant
\begin{equation}  \label{def:spect}
    \lad(\mc{L}) = \inf_{X \in \bh, X \neq \mi} \frac{- \l X, \mc{L} X \r_{\si, 1/2}}{\norm{X - \tr(\si X)\mi}^2_{2, \si}} 
\end{equation}
characterizes the smallest non-zero eigenvalue of $- \mc{L}$ (i.e., the spectral gap). It is worth pointing out that as $\mc{L}$ is Hermitian-preserving, the infimum in \eqref{def:spect} can be taken 
over $X$ in $\mc{B}_{sa}(\mc{H})$. \mb{Moreover, by 
Lemma \ref{lem:self_adjoint}, for a primitive QMS satisfying $\si$-GNS DBC, the Poincar\'{e} inequality 
can be equivalently defined by, for any  $X \in \mc{B}(\mc{H})$, 
 \begin{equation}  \label{ineq_pi}
        \lad \norm{X - \tr(\si X) \mi}^2_{\si,f} \le - \l X, \mc{L} X \r_{\si,f}\,, \tag{$\text{PI}_f$}
    \end{equation}
 where $f:(0,\infty) \to (0,\infty)$ and the norm $\norm{\dd}_{\si,f}$ is defined by the inner product \eqref{eq:general_inner}. It is clear from Lemma \ref{lem:self_adjoint} that under $\si$-GNS DBC, the optimal constant $\lad(\mc{L})$ is independent of the choice of $f$.}

\mb{In the rest of this section}, we will derive some properties of the optimal constants for functional inequalities in Definition 
\ref{def:quantum_func} and the relations among them. 
We summarize some known relations between \eqref{ineq_mlsi}, \eqref{ineq_lsi}, and 
\eqref{ineq_pi2} from \cite{kastoryano2013quantum,temme2014hypercontractivity,olkiewicz1999hypercontractivity} 
in the following lemma for completeness and future use. 

\begin{lemma} \label{lem:aux_func}
For a primitive QMS $\mc{P}_t$ satisfying \mb{$\si$-{\rm KMS DBC}}, it holds that\mb{
\begin{align*}
2 \alpha_1(\mc{L}) \le \lad (\mc{L})\q \text{and}\q  \frac{\lad (\mc{L})}{2 - \log (\si_{\min})} \le 2 \beta(\mc{L}) \le \lad(\mc{L})\,.
\end{align*}
Moreover, if $\mc{P}_t$ satisfies $\si$-{\rm GNS DBC}, we also have $\beta(\mc{L}) \le  \alpha_1(\mc{L})$. }
\end{lemma}

We first consider the properties of quantum dual $q$-Beckner's inequalities \eqref{ineq_dbecq}. It is clear from definition that when $q = 1$, the inequality \eqref{ineq_dbecq} reduces to the Poincar\'{e} inequality \eqref{ineq_pi2}. On the other hand, it is easy to see that in the limit $q \to 2$, \eqref{ineq_dbecq} gives \eqref{ineq_lsi}. In this sense, \eqref{ineq_dbecq} can be considered as an interpolating family between the quantum LSI and the Poincar\'{e} inequality. Indeed, we have Proposition \ref{dualbecktolsi} below. The proof is based on the following monotonicity lemma.

\begin{lemma}\label{lem:monotone}
The function $  \frac{\Var_{q,\si}(Y)}{1/q - 1/2}$ is monotone increasing \mb{for $q \in (0, \infty)\backslash \{2\}$}.
\end{lemma}

\begin{proof}
We recall the interpolation of the noncommutative $L_p$ space \cite{beigi2013sandwiched,gu2019interpolation}: \mb{for $0 < p_0 < p_1 \le \infty$} and $\theta \in [0,1]$, letting $1/p_\theta = (1-\theta)/p_0 + \theta/p_1$ with $\theta \in [0,1]$, there holds
$$
\norm{Y}_{p_\theta,\si} \le \norm{Y}_{p_0,\si}^\theta \norm{Y}_{p_1,\si}^{1-\theta}\,, \q  \forall\, Y \in \bh\,.
$$ 
We immediately see that the function $\log \norm{Y}_{1/t,\si}$ is a convex function for \mb{$t \in [0, \infty)$}, which implies that 
\begin{equation*}
\vp(t) := \exp(2 \log \norm{Y}_{1/t,\si}) =\norm{Y}^2_{1/t,\si}    
\end{equation*}
is also convex. Therefore, we have that the function
\begin{align*}
    \frac{\Var_{q,\si}(Y)}{1/q - 1/2} = \frac{\vp(1/2) - \vp(1/q)}{1/q - 1/2} 
\end{align*}
is increasing in $q$. 
\end{proof}

\begin{proposition} \label{dualbecktolsi} 
Let $\mc{P}_t$ be a primitive QMS with $\si$-{\rm KMS DBC}.  If \eqref{ineq_dbecq} holds with $\limsup_{q \to 2^-}\beta_q > 0$, then \eqref{ineq_lsi} holds with $\beta \ge \limsup_{q \to 2^-} \beta_q/2$.  Conversely, if \eqref{ineq_lsi} holds with $\beta > 0$, then \eqref{ineq_dbecq} holds for any \mb{$q \in (0,2)$} with constant $\beta_q \ge q \beta$.
\end{proposition}

\begin{proof}
Suppose that \eqref{ineq_dbecq} holds with $\limsup_{q \to 2^-}\beta_q > 0$. By the formula \eqref{eq:differ_norm}, we have
\begin{align} \label{auxeq_1:limit}
    \frac{1}{2}\Ent_{2,\si}(Y) = \lim_{q \to 2^-}\frac{\norm{Y}_{2,\si}^2 - \norm{Y}_{q,\si}^2}{2-q}  \q \text{for any}\ Y \ge 0\,.
\end{align}
Then taking the upper limit as $q \to 2^-$ in \eqref{ineq_dbecq}, we find 
\begin{align*}
   \frac{1}{2} \bigl(\limsup_{q \to 2^-} \beta_q\bigr) \Ent_{2,\si}(Y) \le \mc{E}_{2,\mc{L}}(Y)\,.
\end{align*}
Thus, by definition, \eqref{ineq_lsi} holds with $\beta \ge \limsup_{q \to 2^-} \beta_q/2$.  
For the reverse direction, by Lemma \ref{lem:monotone},
it follows, from \eqref{auxeq_1:limit} and the assumption  \eqref{ineq_lsi}, that 
\begin{align*}
    \frac{\Var_{q,\si}(Y)}{1/q - 1/2} \le \limsup_{q \to 2^-}  \frac{\Var_{q,\si}(Y)}{1/q - 1/2}  = 2 \Ent_{2,\si}(Y) \le 2\beta^{-1} \mc{E}_{2,\mc{L}}(Y)\,,
\end{align*}
that is, \eqref{ineq_dbecq} holds with $\beta_q \ge q \beta$.
\end{proof}

\begin{remark} \label{lem:cont_dualbeck}
\mb{Note from  $\norm{Y}_{1,\si} \le \norm{Y}_{q,\si}$ for $q \in [1,2)$ that
\begin{align*}
\Var_{q,\si}(Y)\le \Var_{\si}(Y) \le \lad^{-1}\mc{E}_{2,\mc{L}}(Y)\,,
\end{align*}
which gives $\beta_q \ge (2-q) \lad$. It follows that for $q \in [1,2)$, the lower bound $\beta_q \ge q \beta$ above can be improved:
\begin{align} \label{bound_dbeck}
    \beta_q(\mc{L}) \ge \max\{(2-q) \lad (\mc{L}), q \beta(\mc{L})\}\,,
\end{align}
which is tight when $q \to 1^+$ and $q \to 2^-$. Indeed, for $q = 1$, we have $\beta_1(\mc{L}) = \lad(\mc{L}) = \max\{\lad (\mc{L}), \beta(\mc{L})\}$ by Lemma \ref{lem:aux_func}. When  $p \to 2^-$, by Proposition \ref{dualbecktolsi}, there holds, 
\begin{align} \label{est:limit_dbeck}
    \liminf_{q \to 2^-} \beta_q(\mc{L}) \ge  \lim_{q \to 2^-} \max\{(2-q) \lad (\mc{L}), q \beta(\mc{L})\} =  2 \beta(\mc{L}) \ge \limsup_{q \to 2^-} \beta_q(\mc{L})\,,
\end{align}
that is, $\lim_{q \to 2^-} \beta_q(\mc{L}) = \lim_{q \to 2^-} \max\{(2-q) \lad (\mc{L}), q \beta(\mc{L})\} =  2 \beta(\mc{L})$.}
\end{remark}

The quantum dual Beckner constant $\beta_q(\mc{L})$ has the following monotonicity property, which implies that if \eqref{ineq_dbecq} holds for some \mb{$q \in (0,2)$}, then it holds for all \mb{$q \in (0,2)$}. 

\begin{proposition} \label{prop:mono_dual_beck} 
Let $\mc{P}_t = e^{t\mc{L}}$ be a primitive QMS with $\si$-{\rm KMS DBC}.
For the optimal dual Beckner constant $\beta_q(\mc{L})$ in \eqref{ineq_dbecq}, it holds that $\beta_q(\mc{L})/(2-q)$ is increasing and  $\beta_q(\mc{L})/q$ is decreasing for \mb{$q \in (0,2)$}. 
\end{proposition}

\begin{proof}
For the first claim, by the ordering of $\norm{\dd}_{p,\si}$ in Lemma \ref{lem:prop:norm_power}, we have, for \mb{$ 0 < q \le q' < 2$} and $Y \ge 0$,
\begin{equation*}
\beta_q(\norm{Y}^2_{2,\si} - \norm{Y}^2_{q',\si}) \le \beta_q(\norm{Y}^2_{2,\si} - \norm{Y}^2_{q,\si}) \le \frac{2-q}{2-q'} (2 - q') \mc{E}_{2,\mc{L}}(Y)\,,
\end{equation*}
that is, $\beta_{q'}/(2-q') \ge \beta_q/(2-q)$.  The second claim is a direct consequence of Lemma \ref{lem:monotone}. Indeed, due to the monotonicity, we have 
\begin{equation*}
    2q \beta_{q'} \frac{\Var_{q,\si}(Y)}{2 - q} \le 2q' \beta_{q'} \frac{\Var_{q',\si}(Y)}{2-q'} \le 2q'  \mc{E}_{2,\mc{L}}(Y)\,,
\end{equation*}
which clearly shows $\beta_q/q \ge \beta_{q'}/q'$. 
\end{proof}

We next prove an analog result for the quantum Beckner constant $\alpha_p(\mc{L})$.

\begin{proposition} \label{prop:mono_beck} 
Let $\mc{P}_t = e^{t\mc{L}}$ be a primitive QMS with \mb{$\si$-{\rm GNS DBC}}.
If the Beckner's inequality \eqref{ineq_becp} holds for some $p' \in (1,2]$ with $\alpha_{p'} > 0$, then for any $1 < p \le p'$, the inequality \eqref{ineq_becp} holds with constant $\alpha_p$ satisfying
\begin{align} \label{eq:monoconst}
    \mb{\h{p}}\,\alpha_p \ge  \mb{\widehat{p'}}\, \alpha_{p'}. 
\end{align}
Equivalently, $\mb{\h{p}}\,\alpha_p(\mc{L})$, as a function of $p \in (1,2]$, is nonincreasing. 
\end{proposition}

\begin{proof}
It suffices to prove the inequality \eqref{eq:monoconst}.
For this, by the quantum Stroock-Varopoulos inequality in \eqref{eq:quantumsv} and \eqref{ineq_becp} for $p'$, we have
\begin{align} \label{auxeq:mono}
    \mc{E}_{p,\mc{L}}(X) \ge  \mc{E}_{p',\mc{L}}(I_{p',p}(X)) \ge \alpha_{p'} \mb{\widehat{p'}} \left(\norm{I_{p',p}(X)}_{p',\si}^{p'} - \norm{I_{p',p}(X)}_{1,\si}^{p'}\right)\,.
\end{align}
Note from Lemma \ref{lem:prop:norm_power} that $\norm{I_{p',p}(X)}_{p',\si}^{p'} = \norm{X}_{p,\si}^p$. By ALT inequality \eqref{eq_alt_2}, we find
\begin{align*}
    \norm{I_{p',p}(X)}_{1,\si} & = \tr\left( \si^{\frac{1}{2\widehat{p'}}} \left(\si^{\frac{1}{2p}} X \si^{\frac{1}{2p}}\right)^{\frac{p}{p'}}\si^{\frac{1}{2\widehat{p'}}}\right) \le \norm{X}^{\frac{p}{p'}}_{\frac{p}{p'},\si}\,,
\end{align*}
since 
\begin{equation*}
    \frac{p}{p'} < 1\,,\q \frac{1}{2\widehat{p'}} \frac{p'}{p} + \frac{1}{2p} = \frac{p'}{2p}\,.
\end{equation*}
Therefore, by \eqref{auxeq:mono} and Lemma \ref{lem:prop:norm_power}, it follows that 
\begin{equation*}
    \mc{E}_{p,\mc{L}}(X)  \ge \alpha_{p'} \mb{\widehat{p'}} \left(\norm{X}_{p,\si}^p - \norm{X}^p_{\frac{p}{p'},\si}\right)  \ge \alpha_{p'} \mb{\widehat{p'}} \left(\norm{X}_{p,\si}^p - \norm{X}^p_{1,\si}\right)\,.
\end{equation*}
The proof is complete by definition \eqref{ineq_becp}.  
\end{proof}

We finally relate the $p$-Beckner's inequality and the dual $q$-Beckner's inequality with $q = 2/p$.

\begin{proposition} \label{prop:beck_to_dual_beck} 
Let $\mc{P}_t = e^{t\mc{L}}$ be a primitive QMS with \mb{$\si$-{\rm GNS DBC}}. 
Let $p \in (1,2]$ and $q = 2/p \in [1,2)$. If \eqref{ineq_dbecq} holds with $\beta_q$, then \eqref{ineq_becp} holds with \mb{$\alpha_p \ge \beta_q/2$}.
\end{proposition}

\begin{proof}
We substitute $Y  = I_{2,p}(X)$ for $X \ge 0$ in \eqref{ineq_dbecq} and find 
\begin{equation} \label{auxeq_dualbeck_to_beck}
    \norm{I_{2,p}(X)}^2_{2,\si} - \norm{I_{2,p}(X)}_{2/p,\si}^{2} = \norm{X}^p_{p,\si} - \norm{I_{2,p}(X)}_{2/p,\si}^{2} \le \beta_q^{-1}(2-q) \mc{E}_{2,\mc{L}}(I_{2,p}(X))\,.
\end{equation}
By ALT inequality \eqref{eq_alt_2} and Lemma \ref{lem:prop:norm_power}, we have 
\begin{align} \label{auxeq_dualbeck_to_beck_2}
    \norm{I_{2,p}(X)}_{2/p,\si}^{2/p} & = \tr \big(\si^{\frac{p}{4} - \frac{1}{4}} \big(\si^{\frac{1}{2p}} X \si^{\frac{1}{2p}}  \big)^{\frac{p}{2}}\si^{\frac{p}{4} - \frac{1}{4}}\big)^{\frac{2}{p}} \notag \\
    & \le \tr \bigl((\Gamma_\si(X))^{\frac{p}{2}}\bigr)^{\frac{2}{p}} \le \norm{X}_{1,\si}.
\end{align}
Then, by $L_p$ regularity in Corollary \ref{lem:com_diri}, it follows from \eqref{auxeq_dualbeck_to_beck} and \eqref{auxeq_dualbeck_to_beck_2} that 
\begin{align*}
  \norm{X}^p_{p,\si} - \norm{X}_{1,\si}^p \le \beta_q^{-1}(2-q) \mc{E}_{p,\mc{L}}(X)\,,
 \end{align*}
 which gives \mb{$\alpha_p \ge \beta_q/2$}. 
\end{proof}

\subsection{Quantum Beckner constant} \label{sec:positive_stab}
In this section, we focus on quantum Beckner's inequalities \eqref{ineq_becp} and investigate the properties of  
$\alpha_p(\mc{L})$. \mb{We will first exploit the connections between \eqref{ineq_becp} and the functional inequality related to sandwiched R\'{e}nyi entropy, as well as the hypercontractivity. We then provide a two-sided bound for $\alpha_p(\mc{L})$ in terms of the Poincar\'{e} constant $\lad(\mc{L})$ for a certain range of $p$.} The relations between \eqref{ineq_becp} and \eqref{ineq_mlsi} will also be discussed. 
Moreover, we extend the quantum Holley-Stroock’s argument from \cite{junge2019stability} and give a 
stability estimate for the Beckner constant $\alpha_p(\mc{L})$ with respect to the invariant state $\si$. \mb{In the remaining of this work, we mainly consider the range $p \in (1,2]$ for \eqref{ineq_becp} for ease of exposition, unless otherwise specified.}

\mb{\subsubsection*{Connection with sandwiched R\'{e}nyi entropic inequality} 
In \cite[Definition 3.1]{muller2018sandwiched},
a functional inequality very similar to Beckner's inequality \eqref{eq_beck} above was introduced for quantifying the convergence of sandwiched R\'{e}nyi entropy \eqref{def:sandwi} with $p > 1$ along the QMS:   for some $\h{\alpha}_p > 0$,
\begin{align} \label{ineq:sand}
    2 \h{\alpha}_p D_p(\rho\|\si) \le \frac{4\mc{E}_{p,\mc{L}}(\gs^{-1}(\rho))}{p\norm{\gs^{-1}(\rho)}^p_{p,\si}}\,, \q \forall\rho \in \dh\,.
\end{align}
Note from \cite[Corollary 3.1]{muller2018sandwiched} that the right-hand side of \eqref{ineq:sand} is the entropy production of $D_p(\rho\| \si)$. Since the sandwiched R\'{e}nyi entropy $D_p(\rho \| \si)$ is the logarithm of $p$-divergence in some sense, we may expect that its exponential convergence is a stronger notion than the one of $\mc{F}_{p,\si}$. Indeed, we show in the following proposition that the inequality \eqref{ineq:sand} for $D_p(\rho\|\si)$ can imply the one  \eqref{eq_beck} for $\mc{F}_{p,\si}$, while \eqref{eq_beck} can only guarantee \eqref{ineq:sand} for $\rho$ in a neighborhood of $\si$ (in other words, the exponential convergence of $D_p(\rho_t\| \si)$ if a warm start is given).}

\begin{proposition} \label{prop:sand} \mb{
Let $\mc{P}_t = e^{t \mc{L}}$ be a primitive QMS with \mb{$\si$-{\rm KMS DBC}} and $p > 1$. If the inequality \eqref{ineq:sand} for sandwiched R\'{e}nyi entropy holds, then Beckner's inequality \eqref{eq_beck} holds with $\alpha_p \ge \h{\alpha}_p/2$.  Conversely, if \eqref{eq_beck} holds, then \eqref{ineq:sand} holds in a neighborhood of the invariant state $\si$: for any $a > 0$ with $c_a = (1 - e^{-a})/a$, }
\begin{equation}   \label{eq:sand_decay}
\mb{   \frac{4\mc{E}_{p,\mc{L}}(\gs^{-1}(\rho))}{p \norm{\Gamma_\si^{-1}(\rho)}_{p,\si}^p} \ge
    \left\{
    \begin{aligned}
& 4 c_a \alpha_p  D_p(\rho\|\si) && \text{if} \ D_p(\rho\|\si) \le \frac{a}{p-1}\,, \\
        & \frac{4 \alpha_p}{p-1}  (1-e^{-a})  &&  \text{if} \ D_p(\rho\|\si) \ge  \frac{a}{p-1}\,.
    \end{aligned}
    \right.}
\end{equation}
\end{proposition}
\begin{proof} \mb{
We start with the easy direction \eqref{ineq:sand} $\Longrightarrow$ \eqref{eq_beck}. We reformulate \eqref{ineq:sand} as:
\begin{align} \label{sand_beck}
     \frac{\h{\alpha}_p}{2} \h{p} \log\big(\norm{\Gamma_\si^{-1}(\rho)}^p_{p,\si}\big) \norm{\gs^{-1}(\rho)}^p_{p,\si} \le \mc{E}_{p,\mc{L}}(\gs^{-1}(\rho))\,.
\end{align}
Then, by the elementary inequality $x\log x \ge x - 1$ with $x = \norm{\gs^{-1}(\rho)}^p_{p,\si}$, Beckner's inequality \eqref{eq_beck} with $\alpha_p \ge \h{\alpha}_p/2$ follows from  \eqref{sand_beck} immediately. Now, assuming \eqref{eq_beck} holds, we rewrite it as
\begin{align} \label{eq_becsand}
     \frac{4 \alpha_p}{p-1} \Big(1 - \norm{\Gamma_\si^{-1}(\rho)}_{p,\si}^{-p}\Big)
     \le \frac{4\mc{E}_{p,\mc{L}}(\gs^{-1}(\rho))}{p \norm{\Gamma_\si^{-1}(\rho)}_{p,\si}^p}\,.
\end{align}
Note that for any $a > 0$, there holds, with $c_a = (1 - e^{-a})/a$, 
\begin{align} \label{elem_eq}
      1 - e^{-x} \ge c_a x\,,\q \forall  x \in [0,a]\,.
\end{align}
Letting $e^x = \norm{\Gamma_\si^{-1}(\rho)}_{p,\si}^p$ in \eqref{elem_eq}, we readily have the restricted inequality \eqref{eq:sand_decay} from \eqref{eq_becsand}.}
\end{proof}

\mb{\subsubsection*{Connection with hypercontractivity}  
In view of existing works \cite{beckner1989generalized,arnold2007interpolation} and the direct relations between \eqref{ineq_becp} and the noncommutative $p$-norm, one may expect that there also exist natural connections between Beckner's inequalities and the hypercontractivity of QMS, which we will elaborate below. We first recall the $p$-log-Sobolev inequality ($p \in \R \backslash \{0\}$): for some $\w{\beta}_p > 0$,
\begin{align} \label{ineq_lsip}
   \w{\beta}_p  \Ent_{p,\si}(X) \le \mc{E}_{p,\mc{L}}(X)\,,\ \forall X > 0\,, \tag{LSIp}
\end{align} 
which reduces to \eqref{ineq_mlsi} and \eqref{ineq_lsi} when $p = 1$ and $p = 2$, respectively. \eqref{ineq_lsip} is known to be equivalent to the hypercontractivity ($p \ge 1$) and the reverse hypercontractivity ($p < 1$) \cite{olkiewicz1999hypercontractivity,beigi2020quantum}. 
To be specific, we have the following result from \cite[Theorem\,11 and Corollary\,17]{beigi2020quantum}.

\begin{lemma} \label{lem:sobo_hyper}
Let $\mc{P}_t$ be a primitive QMS satisfying $\si$-{\rm KMS DBC} for some $\si \in \dhh$. It holds that 
\begin{itemize}
    \item If for $1 \le p \le q(t) = 1 + (p-1)e^{4 \beta_c t}$ with fixed $p$ and $\beta_c > 0$, there holds 
    \begin{align} \label{eq:hyper}
        \norm{\mc{P}_t(X)}_{q(t), \si} \le \norm{X}_{p,\si}\,, \q  \forall t \ge 0\,,\ X > 0\,,
    \end{align}
    then the p-log-Sobolev inequality \eqref{ineq_lsip} holds with $\w{\beta}_p \ge \beta_c$. Conversely, suppose that $\mc{P}_t$ satisfies $\si$-{\rm GNS DBC} and \eqref{ineq_lsi} holds with optimal constant $\beta(\mc{L})$. Then \eqref{eq:hyper} holds with $\beta_c = \beta(\mc{L})$. 
     \item If for $1 > p \ge q(t) = 1 + (p-1)e^{4 \beta_c t}$ with $p,q \neq 0$ and $\beta_c > 0$, there holds 
    \begin{align} \label{eq:reverhyper}
        \norm{\mc{P}_t(X)}_{q(t), \si} \ge \norm{X}_{p,\si}\,, \q  \forall t \ge 0\,,\ X > 0\,,
    \end{align}
    then the p-log-Sobolev inequality \eqref{ineq_lsip} holds with $\w{\beta}_p \ge \beta_c$. Conversely, suppose that $\mc{P}_t$ satisfies $\si$-{\rm GNS DBC} and \eqref{ineq_mlsi} holds with optimal constant $\alpha_1(\mc{L})$. Then \eqref{eq:reverhyper} holds with $\beta_c = \alpha_1(\mc{L})$. 
\end{itemize}
\end{lemma}

 It is straightforward to derive the next proposition, which, along with Lemma \ref{lem:sobo_hyper}, relates \eqref{ineq_becp}, \eqref{ineq_lsip}, and the hypercontractivity of $\mc{P}_t$.} 

\begin{proposition} \label{prop:connectpsob}
\mb{Let $\mc{L}$ be the generator of a primitive QMS with $\si$-{\rm KMS DBC} and $\alpha_p(\mc{L})$ and $\w{\beta}_p(\mc{L})$ be the optimal constants for \eqref{ineq_becp} and \eqref{ineq_lsip}, respectively, for $p \in (0,1)\cup(1,\infty)$. Then, we have $p\alpha_p(\mc{L}) \ge  \w{\beta}_p(\mc{L})$ for $p > 1$ and $p\alpha_p(\mc{L}) \le \w{\beta}_p(\mc{L})$ for $0 < p < 1$. }
\end{proposition}

\begin{proof}\mb{
We recall that $\log \norm{X}_{p, \si}^{\h{p}}$ for $X = \gs^{-1} (\rho)$ with $\rho \in \dhh$ is increasing in $ p > 0$, by the monotonicity of sandwiched R\'{e}nyi entropy $D_p(\rho \| \si)$ in $p$. Then, a direct computation gives 
\begin{align*}
    \frac{d}{d p} \log \norm{X}_{p, \si}^{\h{p}} = - \frac{1}{(p-1)^2} \log \norm{X}_{p,\si} + \frac{1}{p(p-1)} \norm{X}_{p,\si}^{-p} \Ent_{p,\si}(X) \ge 0\,,
\end{align*}
which implies, by again $x\log x \ge x - 1$ with $x = \norm{\gs^{-1}(\rho)}^p_{p,\si}$, 
\begin{align*}
       \frac{p}{p-1}  \Ent_{p,\si}(X) \ge  \frac{p}{(p-1)^2} \norm{X}_{p,\si}^{p}  \log \norm{X}^p_{p,\si} \ge \frac{\h{p}}{p-1} (\norm{X}^p_{p,\si} - 1)\,.
\end{align*}
The proof is complete by the definition of \eqref{ineq_becp} and \eqref{ineq_lsip}. }
\end{proof}


 \subsubsection*{Two-sided estimates of $\alpha_p(\mc{L})$.} We next consider the estimation of the Beckner constant $\alpha_p(\mc{L})$
in terms of $\lad(\mc{L})$. We first show that \eqref{ineq_becp} implies \eqref{ineq_pi} with $f = \vp_p$  by the standard linearization argument. 
\begin{lemma} \label{lem:upppoin}
\mb{Let $\mc{P}_t = e^{t\mc{L}}$ be a primitive QMS satisfying $\si$-{\rm GNS DBC}. Then we have, for $p \in \R \backslash \{0,1\}$,}
\begin{align*}
\mb{ 2 \alpha_p(\mc{L}) \le \lad(\mc{L})\,. }
\end{align*}
\end{lemma}

\begin{proof}
We consider $Z = \mi + \ep X$ for a $X \in \mc{B}_{sa}(\mc{H})$ with $\tr(\si X) = 0$, where $\ep > 0$ is small enough such that $Z > 0$. By a direct expansion with respect to $\ep$, we find 
\begin{equation} \label{eq:asy_divi}
    \norm{Z}_{p,\si}^p = 1 + \ep p \tr(\si X) + \frac{\ep^2}{2}p(p-1)\norm{X}_{\si,\vp_p}^2 + O(\ep^3)\,,
\end{equation}
and 
\begin{align*}
    \norm{Z}_{1,\si}^p = \big(1 + \ep \tr(\si X)\big)^p = 1 + \ep p \tr(\si X) + \frac{\ep^2}{2}p(p-1)\tr(\si X)^2 + O(\ep^3)\,.
\end{align*}
We also compute 
\begin{align} \label{eq:expansion_power}
 \gs I_{\h{p},p}(Z) = \si + \ep (p-1) \vp_p(\Delta_\si)R_\si(X) + O(\ep^2)\,,
\end{align}
and $\mc{L}(Z) = \ep \mc{L}(X)$, which yields 
\begin{align} \label{eq:asy_diri}
    \mc{E}_{p,\mc{L}}(Z) = -  \ep^2(p-1)\frac{\h{p}p}{4} \l X, \mc{L} X\r_{\si,\vp_p} + O(\ep^3)\,. 
\end{align}
Hence, applying \eqref{ineq_becp} to $Z$ with above expansions, we have
\begin{align} \label{midauxeq}
    \alpha_p \frac{\ep^2}{2} p^2 \left(\norm{X}_{\si,\vp_p}^2 - \tr(\si X)^2\right) + O(\ep^3) 
    \le -  \ep^2\frac{p^2}{4} \l X, \mc{L} X\r_{\si,\vp_p} + O(\ep^3)\,.
\end{align}
Due to $\tr(\si X) = 0$, by dividing both sides of \eqref{midauxeq} by $\ep^2$ and letting $\ep \to 0$, 
it readily follows that $$2 \alpha_p \norm{X}_{\si,\vp_p}^2 \le - \l X, \mc{L} X \r_{\si,\vp_p}\,,$$ 
which gives the desired estimate: $2 \alpha_p \le \lad$. 
\end{proof}

We next generalize \cite[Proposition 2.8]{adamczak2022modified} in the classical setting to the quantum regime, which shows that the quantum Poincar\'{e} inequality \eqref{ineq_pi2} implies $p$-Beckner's inequality \eqref{ineq_becp} for $p \in (1,2]$. 

\begin{proposition} \label{prop:pb} 
Let $\mc{P}_t = e^{t\mc{L}}$ be a primitive QMS with \mb{$\si$-{\rm GNS DBC}}. 
If the Poincar\'{e} inequality \eqref{ineq_pi2} holds with constant $\lad$, then, for $p \in (1,2]$, the Beckner's inequality \eqref{ineq_becp} holds with constant $\alpha_p$ satisfying 
\begin{equation} \label{eq:poin_to_beck}
\mb{\alpha_p \ge \frac{p-1}{p} \lad\,.}
\end{equation}
\end{proposition}

\begin{proof}
We first claim that there holds 
\begin{equation} \label{auxeq:convex_bec}
      \norm{X}_{p,\si}^p -  \norm{X}_{1,\si}^p \le \l I_{\h{p},p}(X), X - \tr(\si X) \mi \r_{\si,1/2}\,,\q \forall X \ge 0\,.
\end{equation}
Indeed, a direct computation gives 
\begin{align} \label{auxeq:ipp}
     \left\l I_{\h{p},p}(X), X - \tr(\si X) \mi \right\r_{\si,1/2} = \norm{X}_{p,\si}^p - \norm{X}_{1,\si} \tr(\si^{1/p}(\gs^{1/p}X)^{p-1})\,. 
\end{align}
Then, by ALT inequality \eqref{eq_alt_2} and the ordering of $\norm{\dd}_{p,\si}$ in Lemma \ref{lem:prop:norm_power}, we have
\begin{align*}
   \tr(\gs^{1/p}(\gs^{1/p}X)^{p-1}) & \le  \tr((\gs^{1/(p-1)}X )^{p-1}) = \norm{X}_{p-1,\si}^{p-1} \le \norm{X}_{1,\si}^{p-1}\,,
\end{align*}
which, along with \eqref{auxeq:ipp},  implies the desired inequality \eqref{auxeq:convex_bec}. 
One can readily note that up to some constant, the right-hand term in \eqref{auxeq:convex_bec} is the $p$-Dirichlet form $\mc{E}_{p,\mc{L}_{\rm depol}}$ defined by \eqref{def:p_diri} 
associated with the generator $\mc{L}_{\rm depol}$ with $\gamma = 1$; see Example \ref{exp:dep_channel}. We proceed by using Corollary \ref{lem:com_diri} and find 
\begin{align} \label{auxeq:diri_depo}
     \mc{E}_{p, \mc{L}_{\rm depol}}(X) \le \frac{p^2}{4(p-1)} \Var_{2,\si}(I_{2,p}(X)) \le \frac{1}{\lad}\frac{p^2}{4(p-1)} \mc{E}_{2,\mc{L}}(I_{2,p}(X)) \le \frac{1}{\lad}\frac{p^2}{4(p-1)} \mc{E}_{p, \mc{L}}(X)\,,
\end{align}
where we have also used the assumption that \eqref{ineq_pi2} holds, and the observation:
\begin{equation*}
    \Var_{2,\si}(X) = \l X - \tr(\si X) \mi, X - \tr(\si X) \mi \r_{\si,1/2} =  \mc{E}_{2,\mc{L}_{\rm depol}}(X)\,.
\end{equation*}
Then it follows from \eqref{auxeq:convex_bec} and \eqref{auxeq:diri_depo} that 
\begin{align*}
\lad (\norm{X}_{p,\si}^p - \norm{X}_{1,\si}^p) \le   \mc{E}_{p, \mc{L}}(X)\,,
\end{align*}
that is, $\h{p}\alpha_p \ge \lad$ holds. 
\end{proof}

Note that the lower bound for Beckner constant $\alpha_p(\mc{L})$ in \eqref{eq:poin_to_beck} vanishes as $p \to 1^+$. \mb{Thanks to Proposition \ref{prop:connectpsob}, it is easy to establish a uniform lower bound for $\alpha_p(\mc{L})$, which improves the estimate in \eqref{eq:poin_to_beck}.
\begin{theorem} \label{them:lowerbeck} 
  Let $\mc{P}_t = e^{t\mc{L}}$ be a primitive QMS satisfying $\si$-{\rm GNS DBC}. Then we have, for $p \in (1,2]$, 
    \begin{equation} \label{est:lowerbeck} 
      \mb{  \alpha_p(\mc{L}) \ge \max\Big\{  \frac{1}{2p(2 - \log (\si_{\min}))}\,, \frac{p-1}{p}  \Big\} \lad (\mc{L})\,.}
    \end{equation}
\end{theorem}

\begin{proof}
Recall from \cite[Corollary 16]{beigi2020quantum} that $\w{\beta}_p(\mc{L})$ is decreasing in $p \in (0,2]$ for a QMS with $\si$-{\rm GNS DBC}. Then, we obtain, by Lemma \ref{lem:aux_func} and Proposition \ref{prop:connectpsob},
    \begin{align} \label{est:lowerbeck2} 
      \alpha_p(\mc{L}) \ge \frac{1}{p}\w{\beta}_p(\mc{L}) \ge \frac{1}{p} \w{\beta}_2(\mc{L}) \ge \frac{1}{2p}  \frac{\lad (\mc{L})}{2 - \log (\si_{\min})}\,.
        \end{align}
The proof is complete by Proposition \ref{prop:pb}. 
\end{proof}

\begin{remark}
   The estimate \eqref{est:lowerbeck2} holds for any $p \ge 1$ by using $\w{\beta}_p(\mc{L}) = \w{\beta}_{\h{p}}(\mc{L})$ in \cite[Proposition 10]{beigi2020quantum}. 
\end{remark}

\begin{remark} 
The minimal eigenvalue $\si_{\min}$ for the invariant state $\si$ has been estimated in several interesting cases; see \cite[Remark 1]{kastoryano2013quantum} and \cite[Section 7]{muller2018sandwiched}. Note that a crude estimate  
$ \si_{\min} \le 1/2$ gives $2(2 - \log(\si_{\min})) \ge c_0: =  2(2 + \log 2)$. It follows that if $p \ge 1 + c_0^{-1} \approx 1.186$, the estimate \eqref{est:lowerbeck} reduces to the one \eqref{eq:poin_to_beck}, which is clearly tight when $p \to 2^-$.  However, when $p \to 1^+$, by $2 \alpha_1 \le \lad$ in Lemma \ref{lem:aux_func}, 
we only have
$$
\alpha_p(\mc{L}) \ge \frac{\alpha_1(\mc{L})}{p(2 - \log (\si_{\min}))} \gtrsim \frac{\alpha_1(\mc{L})}{\log d}\,, 
$$
if $\si_{\min}$ is of the same order as $1/d$, which is far from tight, given \eqref{eq:opt_becktoopt_msli} below. For the symmetric QMS with invariant state $\si = \mi/d$, one can extend the results in the classical case \cite[Proposition 2.9 and Theorem 2.1]{adamczak2022modified} to have a tighter lower bound for $\alpha_p$, which goes to $C \alpha_1(\mc{L})$ as $p \to 1^+$, with constant $C$ independent of the dimension $d$ of $\mc{H}$. However, the argument involved in \cite[Proposition 2.9]{adamczak2022modified} seems hard to be generalized to general invariant states $\si \in \dhh$. 
\end{remark}

\begin{remark}
    In Appendix \ref{app:nonprimitive}, we will briefly discuss Beckner's inequalities in the non-primitive setting. In particular, we adopt the analysis framework recently proposed in \cite{gao2021complete} for CMLSI with some tools developed above to give a lower bound for the non-primitive Beckner constant, which is asymptotically worse than the one in \eqref{est:lowerbeck} but can apply to more general Lindbladian $\mc{L}$. 
\end{remark}

\subsubsection*{Relation between $\alpha_1(\mc{L})$ and $\alpha_p(\mc{L})$} Recall from Remark \ref{lem:cont_dualbeck} that  $\lim_{q \to 2^-} \beta_q(\mc{L}) = 2 \beta(\mc{L})$. Next, we give a similar result for Beckner constant $\alpha_p(\mc{L})$, which is more technical; see Theorem \ref{thm:beck_log_sobo} below.}  First, by taking the right limit $p \to 1^+$ in \eqref{ineq_becp} and using formulas \eqref{def:1diri} and \eqref{eq:dev_pnorm}, we obtain the following lemma. 

\begin{lemma} \label{lem:mb} 
Let $\mc{P}_t$ be a primitive QMS with $\si$-{\rm KMS DBC}. If \eqref{ineq_becp} holds with $\limsup_{p \to 1^+} \alpha_p > 0$, then \eqref{ineq_mlsi} holds with constant
\begin{equation} \label{eq:beckmsli_one}
    \alpha_1 \ge \limsup_{p \to 1^+} \alpha_p\,.
\end{equation}
\end{lemma}

\begin{theorem} \label{thm:beck_log_sobo}
Let $\mc{P}_t = e^{t \mc{L}}$ be a primitive QMS with \mb{$\si$-{\rm GNS DBC}}. Then 
we have
\begin{equation} \label{eq:opt_becktoopt_msli}
    \alpha_1(\mc{L}) = \lim_{p \to 1^+} \alpha_p({\mc{L}})\,.
\end{equation}
\end{theorem}

The proof of \eqref{eq:opt_becktoopt_msli} needs the following lemma that extends \cite[Theorem 6.5]{bobkov2006modified} for the discrete MLSI.   

\begin{lemma} \label{lem:extremal_func} 
Let $\mc{P}_t = e^{t \mc{L}}$ be a primitive QMS with \mb{$\si$-{\rm GNS DBC}}. If $\alpha_p(\mc{L}) < \lad(\mc{L})/2$ holds, then the following infimum is attained:
\begin{align} \label{auxeq:defoptap}
    \alpha_{p}(\mc{L}) = \inf_{\substack{X\ge 0, X \neq \mi \\ \norm{X}_{1,\si} = 1}} \frac{\mc{E}_{p,\mc{L}}(X)}{\h{p}(\norm{X}_{p,\si}^p- 1)}\,.
\end{align} 
\end{lemma}

\begin{proof}
By definition, there exists a sequence of $X_n \in \{X \ge 0\,;\ \norm{X}_{1,\si} = 1\,,\, X \neq \mi\}$ such that 
\begin{align} \label{eq:assp}
    V_p(X_n): = \frac{\mc{E}_{p,\mc{L}}(X_n)}{\h{p}(\norm{X_n}_{p,\si}^p- 1)} \to \alpha_p(\mc{L})\,,\q \text{as}\ n \to \infty\,.
\end{align}
Since the set $\{X \ge 0\,;\ \norm{X}_{1,\si} = 1\}$ is compact, without loss of generality, we assume $X_n \to X $ as $n \to \infty$ for some $X \ge 0$ with $\norm{X}_{1,\si} = 1$. Suppose that $X = \mi$. Then we can write $X_n = \mi + Y_n$ with $Y_n \to 0$ and $\tr(\si Y_n) = 0$. Recalling the asymptotic expansions \eqref{eq:asy_divi} and \eqref{eq:asy_diri}, we obtain, by \eqref{ineq_pi}, 
\begin{align*}
\liminf_{n \to \infty} V_p(X_n) = \liminf_{n \to \infty} - \frac{1}{2} \frac{\l Y_n, \mc{L} Y_n \r_{\si,\vp_p} + O(\norm{Y_n}^3_{1,\si})}{\norm{Y_n}^2_{\si,\vp_p} + O(\norm{Y_n}^3_{1,\si})} \ge \frac{1}{2}\lad(\mc{L}) > \alpha_p(\mc{L})\,,
\end{align*}
which contradicts \eqref{eq:assp}. Thus, the limiting operator $X$ is in the desired set $\{X \ge 0\,;\ \norm{X}_{1,\si} = 1\,,\, X \neq \mi\}$ and the infimum in \eqref{auxeq:defoptap} is attained. 
\end{proof}

\begin{proof}[Proof of Theorem \ref{thm:beck_log_sobo}]
To show \eqref{eq:opt_becktoopt_msli}, by Lemma \ref{lem:mb}, it suffices to prove 
\begin{align} \label{auxeq:infmb}
    \liminf_{p \to 1^+}  \alpha_p (\mc{L}) \ge  \alpha_1 (\mc{L})\,.
\end{align}
We shall prove it by contradiction. If \eqref{auxeq:infmb} does not hold, there is a sequence $p_n \to 1^+$ as $n \to \infty$ such that 
\begin{align*} 
    \lim_{n \to \infty}  \alpha_{p_n}(\mc{L}) \le \alpha_1(\mc{L}) - \ep\,, 
\end{align*}
for some small enough $\ep > 0$, that is, for any $k > 0$, there exists $N$ such that for $n \ge N$, 
\begin{align} \label{auxeq:seq}
   \alpha_{p_n}(\mc{L}) \le \alpha_1(\mc{L}) - \ep + \frac{1}{k}\,.
\end{align}
Suppose that $\alpha_{p_n}(\mc{L}) = \lad(\mc{L})/2$ holds for infinitely many $n$. It follows from \eqref{auxeq:seq} that $\lad(\mc{L})/2 \le \alpha_1(\mc{L}) - \ep$, which is a contradiction with Lemma \ref{lem:aux_func}. Thus, without loss of generality, we assume $\alpha_{p_n}(\mc{L}) < \lad(\mc{L})/2$ for all $n$. Then, Lemma \ref{lem:extremal_func} gives the existence of the minimizer $X_n$ associated with $\alpha_{p_n}(\mc{L})$. By compactness, we further assume that $X_n$ converges to some $X \ge 0$ with $\norm{X}_{1,\si} = 1$. To proceed, we consider two cases.  If $X = \mi$, similarly to the proof of Lemma \ref{lem:extremal_func} above,
we write $X_n = \mi + Y_n$ and find 
\begin{align*}
  \alpha_1(\mc{L}) - \ep + \frac{1}{k} \ge   \liminf_{n \to \infty} \alpha_{p_n}(\mc{L}) = \liminf_{n \to \infty} - \frac{1}{2} \frac{\l Y_n, \mc{L} Y_n \r_{\si,\vp_{p_n}} + O(\norm{Y_n}_{1,\si}^3)}{\norm{Y_n}^2_{\si,\vp_{p_n}} + O(\norm{Y_n}_{1,\si}^3)} \ge \frac{1}{2}\lad(\mc{L})\,,
\end{align*}
which, by letting $k \to \infty$, again contradicts with Lemma \ref{lem:aux_func}. If $X \neq \mi$, by definition \eqref{ineq_becp} and \eqref{auxeq:seq}, we have 
\begin{align*}
     \mc{E}_{p_n,\mc{L}}(X_{n}) \le \Big(\alpha_1(\mc{L}) - \ep + \frac{1}{k}\Big) \frac{p_n(\norm{X_{n}}_{p_n,\si}^{p_n}- 1)}{p_n-1}\,,
\end{align*}
for large enough $n$. It implies that, by letting $n \to \infty$ and $k \to \infty$ and using \eqref{eq:dev_pnorm} with elementary analysis, 
\begin{align*}
     \mc{E}_{1,\mc{L}}(X) \le \big(\alpha_1(\mc{L}) - \ep\big) \Ent_{1,\si}(X)\,,
\end{align*}
which contradicts the optimality of $\alpha_1(\mc{L})$. The proof is complete. 
\end{proof}

\subsubsection*{Stability of $\alpha_p(\mc{L})$}
We proceed to investigate the stability of the quantum Beckner constant $\alpha_p(\mc{L})$ with respect to the invariant state. We will compare the constants $\alpha_p(\mc{L})$
for the following two generators $\mc{L}_\si$ and $\mc{L}_{\si'}$ that satisfy the detailed balance conditions with respect to two different but commuting full-rank states $\si$ and $\si'$:
\begin{align} \label{eq:gen_1}
        \mc{L}_\si(X) = \sum_{j = 1}^{J} \big(e^{-\omega_j/2}V_j^*[X,V_j] + e^{\omega_j/2}[V_j,X]V^*_j\big)\,,
\end{align}
and 
\begin{align} \label{eq:gen_2}
        \mc{L}_{\si'}(X) = \sum_{j = 1}^{J} \big(e^{-\nu_j/2}V_j^*[X,V_j] + e^{\nu_j/2}[V_j,X]V^*_j\big)\,,
\end{align}
where $e^{-\ww_j}$ and $e^{-\nu_j}$ are the eigenvalues of $\Delta_\si$ and $\Delta_{\si'}$, respectively. For ease of  exposition, we assume that the states $\si$ and $\si'$ admit the spectral decompositions: 
\begin{align} \label{auxeq_specdecomp}
    \si = \sum_{k = 1}^{d} \si_k \pure{v_k}\,,\q  \si' = \sum_{k = 1}^{d} \si'_k \pure{v_k}\,,
\end{align}
respectively. The following result is extended from \cite[Theorem 3.1]{junge2019stability}.

\begin{theorem} \label{thm:beck_stab}
Let $\mc{L}_\si$ and $\mc{L}_{\si'}$ be the generators of two primitive QMS satisfying \mb{$\si$-{\rm GNS DBC}} and \mb{$\si'$-{\rm GNS DBC}}, given in \eqref{eq:gen_1} and \eqref{eq:gen_2}, respectively.
Then it holds that, \mb{for $p \in (1,2]$},
\begin{align} \label{est_stability}
\frac{\Lad_{\min}}{\Lad_{\max}} \min_{j} e^{-\frac{|\ww_j-\nu_j|(2-p)}{2p}} \alpha_p(\mc{L}_{\si'}) \le \alpha_p(\mc{L}_\si)\,,
\end{align}
where constants $\Lad_{\min}$ and $\Lad_{\max}$ are defined as  
\begin{align} \label{eq:const_state}
 \Lad_{\min} = \min_k \frac{\si_k}{\si'_k}  \quad \text{and} \quad \Lad_{\max} = \max_k \frac{\si_k}{\si'_k}\,.
\end{align}
\end{theorem}

\begin{remark}
\mb{The most useful case of the above result might be $\si' = \mi/d$, which allows us to reduce the estimate of $\alpha_p(\mc{L}_\si)$ for a $\si$-{\rm GNS} symmetric QMS to  $\alpha_p(\mc{L}_{\frac{\mi}{d}})$ for a symmetric QMS. In this case, the estimate \eqref{est_stability} can be simplified as follows, by the relation \eqref{eq:repbohr},
\begin{align*}
\alpha_p(\mc{L}_\si) \ge \min_{k,l} \frac{\si_l}{\si_k} \min_{j} e^{-\frac{|\ww_j|(2-p)}{2p}} \alpha_p(\mc{L}_{\frac{\mi}{d}}) = \Big(\min_{k,l} \frac{\si_l}{\si_k}\Big)^{\frac{2+p}{2p}} \alpha_p(\mc{L}_{\frac{\mi}{d}})\,.
\end{align*}
It is also worth mentioning that the assumption ($\si$ and $\si'$ commute) is restrictive in the sense that the jump operators $\{V_j\}$ in \eqref{eq:gen_1} and \eqref{eq:gen_2} are the same. A very recent work \cite{junge2022stability} by Junge and Wu gives a general stability (continuity) result for two non-primitive generators $\mc{L}$ and $\mc{L}'$ that 
satisfy $\si$-{\rm GNS DBC} for the same $\si \in \dhh$ and have the same fixed point algebra. They showed that for $\norm{\mc{L} - \mc{L}'} < \d$, the MLSI constant satisfies $(1-\ep)\alpha_1(\mc{L}) \le \alpha_1(\mc{L}')$, where the dependence of $\ep$ on $\d$ is implicit. It would be interesting to generalize the results in \cite{junge2022stability} to  Beckner constant $\alpha_p(\mc{L})$ for a more general class of generators $\mc{L}$, which is beyond the scope of this work.}
\end{remark}

\begin{proof}[Proof of Theorem \ref{thm:beck_stab}]
We first establish a comparison result for the  $p$--divergence $\mc{F}_{p,\si}(\rho)$. We define the map
\begin{align*}
    \Phi(A): = \mm
    \Lambda_{\max}^{-1} \gs \Gamma_{\si'}^{-1} (A) & 0 \\ 0 & 
    \tr(A) - \Lad_{\max}^{-1} \tr(\gs \Gamma_{\si'}^{-1} (A))
    \nn : \ \bh \to \mc{B}(\mc{H}\oplus \C)\,.
\end{align*}
It is easy to check from \eqref{eq:const_state} that $\tr \big(  \Lambda_{\max}^{-1} \gs \Gamma_{\si'}^{-1} (A) \big) \le \tr (A) $ \mb{for any $A \in  \mc{B}^+_{\rm sa}(\mc{H})$}. Then, $\Phi$ is completely positive and trace preserving by Kraus representation theorem. Let $X := \Gamma_{\si'}^{-1}(\rho) \ge 0$ for $\rho \in \dh$. By the data processing inequality \eqref{eq:dpi_divi}, we have 
\begin{align} \label{eq:dpi}
    \mc{F}_{p, \Phi(\si')}\left(\Phi(\rho)\right) \le \mc{F}_{p,\si'}\left(\rho\right) = \frac{1}{p(p-1)} \left( \norm{X}_{p,\si'}^p - 1 \right). 
\end{align}
We now compute, by definition,
\begin{align} \label{eq:auxdp_2}
      \mc{F}_{p, \Phi(\si')}\left(\Phi(\rho)\right) = \frac{1}{p(p-1)} \left( 
   \Lad^{-1}_{\max} \norm{X}_{p,\si}^p + \big(1 - \Lad^{-1}_{\max} \norm{X}_{1,\si}\big)^p \left(1 - \Lad^{-1}_{\max}\right)^{1-p} - 1
     \right).
\end{align}
Note from \eqref{eq:const_state} that $\Lad_{\max} \ge 1 \ge \Lad_{\min}$ and $\Lad^{-1}_{\max} \norm{X}_{1,\si} \le 1$. 
Then  the convexity of $x^p$, $1 < p \le 2$, gives
\begin{align*}
   &(1 - \Lad^{-1}_{\max})  \left(\frac{1 - \Lad^{-1}_{\max} \norm{X}_{1,\si}}{1 - \Lad^{-1}_{\max}} \right)^p +   \Lad^{-1}_{\max} \left(\frac{ \Lad^{-1}_{\max} \norm{X}_{1,\si}}{\Lad^{-1}_{\max}} \right)^p \\
   = &     \big(1 - \Lad^{-1}_{\max} \norm{X}_{1,\si}\big)^p \big(1 - \Lad^{-1}_{\max}\big)^{1-p} + \Lad^{-1}_{\max} \norm{X}_{1,\si}^p  \ge 1\,,
\end{align*}
which, by \eqref{eq:auxdp_2}, implies
\begin{align} \label{auxeq_stab}
     \mc{F}_{p, \Phi(\si') }\left(\Phi(\rho) \right) \ge \frac{1}{p(p-1)} \left( 
   \Lad^{-1}_{\max} \norm{X}_{p,\si}^p - \Lad^{-1}_{\max} \norm{X}^p_{1,\si}
     \right)\,.
\end{align}
Combining \eqref{eq:dpi} and \eqref{auxeq_stab}, we can find
\begin{equation} \label{eq:sta_main_est_1}
    \frac{1}{p(p-1)} \left( 
    \norm{X}_{p,\si}^p -  \norm{X}^p_{1,\si}
     \right) \le \frac{\Lad_{\max}}{p(p-1)} \left( \norm{X}_{p,\si'}^p - 1  \right)\,.
\end{equation}

We next give the comparison result for $\mc{E}_{p,\mc{L}}$. We recall \eqref{eq:rep_epl}, and, by Lemma \ref{lem:mono_norm}, obtain
\begin{align} \label{est:p-diri}
     \mc{E}_{p,\mc{L}_\si}(X) \ge \left(\inf_{j} e^{-|\ww_j - \nu_j|(2-p)/2p}\right) \frac{p^2}{4} \sum_{j = 1}^J \left\l \gs^{1/p}(\p_j X), f_p^{[1]}\left(e^{\nu_j/2p}\gs^{1/p}(X), e^{-\nu_j/2p}\gs^{1/p}(X)\right)  \gs^{1/p} (\p_j X)   \right\r,
\end{align}
since 
\begin{equation*}
e^{\pm \ww_j/2p} \gs^{1/p}(X) \le (\max_j e^{|\ww_j - \nu_j|/2p}) e^{\pm \nu_j/2p} \gs^{1/p}(X)\,.    
\end{equation*}
By the integral representation \eqref{eq:integral_thetap} of $f_p^{[1]}$, we can estimate
\small
\begin{align}  \label{auxeqq_stabbdiri}
& \gs^{1/p} f_p^{[1]}\left(e^{\nu_j/2p}\gs^{1/p}(X), e^{-\nu_j/2p}\gs^{1/p} (X)\right) \gs^{1/p} \\ = &   \frac{\sin((p-1) \pi)}{\pi} \int_0^\infty  s^{p-2} 
\gs^{1/p} g_0^{[1]}\left(s + e^{\nu_j/2p}\gs^{1/p}(X), s + e^{-\nu_j/2p}\gs^{1/p} (X)\right) \gs^{1/p}
\ ds  \notag \\
 \ge & \frac{\sin((p-1) \pi)}{\pi} \int_0^\infty  s^{p-2} \gs^{1/p} g_0^{[1]}\left(s \Lad_{\min}^{-1/p} \si^{1/p}(\si')^{-1/p} + e^{\nu_j/2p}\gs^{1/p}(X), s \Lad_{\min}^{-1/p} \si^{1/p}(\si')^{-1/p} + e^{-\nu_j/2p}\gs^{1/p} (X)\right) \gs^{1/p}\ ds\,, \notag 
\end{align}
\normalsize
where we used the following observation from \eqref{auxeq_specdecomp} and \eqref{eq:const_state}:
\begin{equation} \label{eq:simple_ob}
\Lad_{\max}^{-1/p} \si^{1/p} (\si')^{-1/p} \le \mi \le \Lad_{\min}^{-1/p} \si^{1/p} (\si')^{-1/p}\,,     
\end{equation}
and the operator monotonicity of $t^{-1}$. The inequality \eqref{eq:simple_ob} also implies that $\Lad_{\min}^{1/p} \gs^{-1/p}\Gamma_{\si'}^{1/p}$ is completely positive and trace non-increasing. Then, by \cite[Proposition 3.6]{junge2019stability}, it follows that 
\small
\begin{multline*}
\gs^{1/p} g_0^{[1]}\left(s \Lad_{\min}^{-1/p} \si^{1/p}(\si')^{-1/p} + e^{\nu_j/2p}\gs^{1/p}(X), s \Lad_{\min}^{-1/p} \si^{1/p}(\si')^{-1/p} + e^{-\nu_j/2p}\gs^{1/p} (X)\right) \gs^{1/p} \\
\ge \Lad_{\min}^{2/p} \Gamma_{\si'}^{1/p} g_0^{[1]}\left(s + \Lad_{\min}^{1/p} e^{\nu_j/2p} \Gamma_{\si'}^{1/p}(X), s  + \Lad_{\min}^{1/p} e^{-\nu_j/2p}\Gamma_{\si'}^{1/p} (X)\right) \Gamma_{\si'}^{1/p}\,. 
\end{multline*}
\normalsize
Therefore, by \eqref{eq:integral_thetap} and \eqref{auxeqq_stabbdiri}, we have 
\begin{align*}
& \gs^{1/p} f_p^{[1]}\left(e^{\nu_j/2p}\gs^{1/p}(X), e^{-\nu_j/2p}\gs^{1/p} (X)\right) \gs^{1/p} \\
\ge & \frac{\sin((p-1) \pi)}{\pi} \int_0^\infty  s^{p-2} \Lad_{\min}^{2/p} \Gamma_{\si'}^{1/p} g_0^{[1]}\left(s + \Lad_{\min}^{1/p} e^{\nu_j/2p} \Gamma_{\si'}^{1/p}(X), s  + \Lad_{\min}^{1/p} e^{-\nu_j/2p}\Gamma_{\si'}^{1/p} (X)\right) \Gamma_{\si'}^{1/p}
\ ds \\
\ge & \Lad_{\min} \Gamma_{\si'}^{1/p} f_p^{[1]}\left(e^{\nu_j/2p}\Gamma_{\si'}^{1/p}(X), e^{-\nu_j/2p}\Gamma_{\si'}^{1/p} (X)\right) \Gamma_{\si'}^{1/p}\,. 
\end{align*}
Combining the above estimate with \eqref{est:p-diri} and recalling \eqref{eq:rep_epl}, we readily have
\begin{align} \label{eq:sta_main_est_2}
     \mc{E}_{p,\mc{L}_\si}(X) \ge \left(\inf_{j} e^{-|\ww_j - \nu_j|(2-p)/2p}\right) \Lad_{\min} \mc{E}_{p,\mc{L}_{\si'}}(X)\,.
\end{align}
The proof is completed by the following simple estimate, with the help of \eqref{eq:sta_main_est_1} and \eqref{eq:sta_main_est_2},
\begin{align*}
     \mb{\frac{p}{p-1}} \left( 
    \norm{X}_{p,\si}^p -  \norm{X}_{1,\si}
     \right) & \le \frac{\Lad_{\max}}{\alpha_p(\mc{L}_{\si'})} \mc{E}_{p,\mc{L}_{\si'}}(X) \\
     & \le \frac{\Lad_{\max}}{\alpha_p(\mc{L}_{\si'})} \left(\inf_{j} e^{-|\ww_j-\nu_j|(2-p)/2p}\right)^{-1} \Lad_{\min}^{-1}  \mc{E}_{p,\mc{L}_\si}(X)\,. \qedhere
\end{align*}
\end{proof}

\subsection{Applications and examples} \label{sec:app_exp}

This section is devoted to the applications of $p$-Beckner's inequalities. We first analyze the Beckner constant $\alpha_p(\mc{L})$ for the depolarizing semigroup. We then derive a bound on the mixing time of quantum Markov dynamics in terms of $\alpha_p(\mc{L})$. For the symmetric semigroups, by borrowing the techniques from \cite{adamczak2022modified,junge2015noncommutative}, we obtain the moment estimates from Beckner's inequalities \eqref{ineq_becp}, which further allows us to derive a concentration inequality.

\subsubsection*{Beckner constant for depolarizing semigroups} In general, it is challenging to explicitly compute or estimate the optimal constant for the functional inequalities, even in the classical setting. We will consider the quantum Beckner constant $\alpha_p$ for the simplest QMS: the depolarizing semigroup \eqref{def:depolt} with $\gamma = 1$ and $\si = \mi/d$, and show that in this case, the computation of $\alpha_p$ is equivalent to the classical one for a Markov chain on the two-point space. We mention that the explicit values of LSI constant $\beta$ and MLSI constant $\alpha_1$ for $\mc{L}_{\rm depol}$ with a general invariant state $\si \in \dhh$ have been obtained in \cite{beigi2020quantum} and \cite{muller2016relative}, respectively. 

\begin{proposition} \label{prop:beck_depol}
Let $\mc{L}_{\rm depol}(X) = \tr(X/d)\mi - X$ be the Lindbladian of the depolarizing semigroup. Then we have 
\begin{align} \label{eq:const_beck}
    \alpha_p(\mc{L}_{\rm depol})  = \inf_{\substack{\theta x + (1 - \theta) y = 1 \\ x,y \ge 0,\  \theta \in \{\frac{1}{d},\ldots, 1 - \frac{1}{d}\} }} \mb{\frac{p}{4}}  \frac{ (\theta x^p + (1-\theta)y^p) - (\theta x^{p-1} + (1-\theta)y^{p-1})}{(\theta x^p + (1-\theta)y^p) - 1}\,,
\end{align}
for $p \in (1,2]$, where $d$ is the dimension of the underlying Hilbert space $\mc{H}$. 
\end{proposition}

\begin{proof} 
We first compute from definition \eqref{ineq_becp} that 
\begin{equation} \label{auxeq:expbeck}
\alpha_{p}(\mc{L}_{\rm depol}) = \inf_{X \ge 0} \frac{\h{p}^{-1}\mc{E}_{p,\mc{L}_{\rm depol}}(X)}{\norm{X}_{p,\frac{\mi}{d}}^p - \norm{X}_{1,\frac{\mi}{d}}^p} = \frac{p}{4} \inf_{X \ge 0} \frac{\norm{X}_{p,\frac{\mi}{d}}^p- \norm{X}_{p-1,\frac{\mi}{d}}^{p-1}\norm{X}_{1,\frac{\mi}{d}} }{\norm{X}_{p,\frac{\mi}{d}}^p - \norm{X}_{1,\frac{\mi}{d}}^p}\,.
\end{equation}
We only consider $p \in (1,2)$, since the case $p = 2$, corresponding to the spectral gap, is trivial. Let $\mu_i \ge 0$, $1 \le i \le d$, be the eigenvalues of $X \ge 0$. It is easy to reformulate \eqref{auxeq:expbeck} as
\begin{align*}
\alpha_{p}(\mc{L}_{\rm depol}) = \frac{p}{4} \inf_{\mu_i \ge 0}  \frac{ \sum_i \mu_i^p - d^{-1} \big(\sum_{i} \mu_i^{p-1}\big)\big(\sum_i \mu_i\big)}{\sum_i \mu_i^p - d^{-(p-1)}\big(\sum_i \mu_i\big)^p }\,,
\end{align*}
which is equivalent to, for any $\mu_i \ge 0$,  
\begin{equation} \label{auxeq:expbeck_2}
F(\mu_1,\ldots, \mu_d): =   \sum_i \mu_i^p - d^{-1} \big(\sum_{i} \mu_i^{p-1}\big)\big(\sum_i \mu_i\big) - \frac{4 \alpha_{p}(\mc{L}_{\rm depol}) }{p} \Big(\sum_i \mu_i^p - d^{-(p-1)}\big(\sum_i \mu_i\big)^p  \Big) \ge 0\,.
\end{equation}
Suppose that $(r_i)_{i =1}^d$ achieves the equality in \eqref{auxeq:expbeck_2}. We claim that all $r_i$ are strictly positive. If not, without loss of generality, we assume $r_1 = 0$, $r_2 > 0$, and
$\sum_{i = 2}^{d} r_i  = 1$. Then, from \eqref{auxeq:expbeck_2}, for small enough $\ep > 0$, we have 
\begin{align*}
F(\ep, r_2 - \ep, r_3, \ldots, r_d) = &\sum_{i = 3} r_i^p - d^{-1} \big(\sum_{i = 3}^d r_i^{p-1}\big) - C_p \Big(\sum_{i = 3}^d r_i^p - d^{-(p-1)}  \Big)  \\
& - d^{-1} \big( (r_2 - \ep) ^{p-1} + \ep^{p-1} \big) - (C_p - 1)  \big( (r_2 - \ep)^p + \ep^p \big) \\
= & - d^{-1} \ep^{p-1} + O(\ep) < 0\,,
\end{align*}
where $C_p := 4 \alpha_p/p$. 
This fact contradicts the assumption that $(r_i)_{i}$ saturates the equality, so the claim holds. 
Now, by $r_i > 0$ for all $i$, we have $\na F (r_1,\ldots,r_d) = 0$, which, by a direct computation, gives the following equations in the variables $r_i$:
\begin{equation} \label{auxeq:expbeck_3}
\big(\frac{p}{4} - \alpha_p \big)r_i^{p-1} - \frac{p-1}{4 d}\norm{X}_1 r_i^{p-2} 
  = \frac{1}{4d}\norm{X}_{p-1}^{p-1} - \alpha_p d^{1-p}\norm{X}_1^{p-1}.
\end{equation}
When $\norm{X}_1$ and $\norm{X}_{p-1}$ are fixed, the above equation clearly
has at most two solutions, denoted by $a$ and $b$, which means that $r_i$ takes the value either $a$ or $b$. Let $n$ be the number of $r_i$ equal to $a$. Then, it follows that
 \begin{align*}
\alpha_{p}(\mc{L}_{\rm depol}) & = \inf_{\substack{n a + (d - n) b = 1 \\ a,b \ge 0,\ n \in \{1,\ldots, d-1\} }}  \frac{p}{4}  \frac{ (n a^p + (d - n)b^p) - d^{-1}(n a^{p-1} + (d-n)b^{p-1})}{(n a^p + (d-n)b^p) - d^{-(p-1)}} \\  
& =  \inf_{\substack{\theta x + (1 - \theta) y = 1 \\x,y \ge 0,\ \theta \in \{\frac{1}{d},\ldots, 1-\frac{1}{d}\} }} \frac{p}{4}  \frac{ (\theta x^p + (1-\theta)y^p) - (\theta x^{p-1} + (1-\theta)y^{p-1})}{(\theta x^p + (1-\theta)y^p) - 1}\,,
\end{align*}
by setting $\theta = n/d$, $x = d a$, and $y = d b$. 
\end{proof}

To connect the expression \eqref{eq:const_beck} with the classical Beckner constant, we consider a Markov chain on $\{0,1\}$ with the transition matrix: for $\theta > 0$,
$$
P = \mm \theta & 1- \theta \\ \theta & 1- \theta \nn\,,
$$
which has the invariant measure $\pi(0) = \theta$, $\pi(1) = 1-\theta$. The Beckner constant for this chain is give by \cite[(4.1)]{bobkov2006modified}
\begin{align} \label{eq:rep_beck}
    \alpha_{p,\theta} & = \inf_{f \ge 0} \frac{p}{2} \frac{- \l f^{p-1}, (P - I) f \r_{\pi}}{\pi(f^p) - \pi(f)^p }  \notag \\
    & =  \inf_{\substack{\theta x + (1 - \theta) y = 1 \\x,y \ge 0 }}  \frac{p}{2}  \frac{ (\theta x^p + (1-\theta)y^p) - (\theta x^{p-1} + (1-\theta)y^{p-1})}{(\theta x^p + (1-\theta)y^p) - 1}\,.
\end{align}
Then we can see 
\begin{align} \label{connec_class}
\alpha_{p}(\mc{L}_{\rm depol}) = \frac{1}{2} \inf \big\{ \alpha_{p,\theta}\,;\ \theta \in \{\frac{1}{d},\frac{2}{d},\ldots, 1-\frac{1}{d}\}\big\}\,.    
\end{align}
However, although the representation \eqref{eq:rep_beck} is simple, numerical techniques are still necessary to find
the explicit values of $ \alpha_{p,\theta}$ and $\alpha_{p}(\mc{L}_{\rm depol})$. \mb{We next derive upper and lower bounds for $\alpha_{p}(\mc{L}_{\rm depol})$.} 

\begin{proposition} \label{prop:upperlowerdep}
\mb{For $\mc{L}_{\rm depol}$ given in Proposition \ref{prop:beck_depol} and $p \in (1,2]$, there holds}
\begin{align} \label{est:depol}
   \mb{ \frac{p}{4} \le \alpha_{p}(\mc{L}_{\rm depol}) \le \min \Big\{\frac{1}{2},  \frac{pd^{p-1}}{4(d^{p-1} - 1)}\Big\} \,.}
\end{align}
\end{proposition}

\begin{proof}\mb{
The lower bound follows from \eqref{auxeq:expbeck} and $\norm{X}_{p-1,\frac{\mi}{d}} \le \norm{X}_{1,\frac{\mi}{d}}$,  
    \begin{align*}
       \frac{4}{p} \alpha_{p}(\mc{L}_{\rm depol}) = \inf_{X \ge 0\,, \norm{X}_{1,\frac{\mi}{d} = 1}} \frac{\norm{X}_{p,\frac{\mi}{d}}^p- \norm{X}_{p-1,\frac{\mi}{d}}^{p-1} }{\norm{X}_{p,\frac{\mi}{d}}^p - 1} \ge  \inf_{X \ge 0\,, \norm{X}_{1,\frac{\mi}{d} = 1}} \frac{\norm{X}_{p,\frac{\mi}{d}}^p- 1 }{\norm{X}_{p,\frac{\mi}{d}}^p - 1}\,.
    \end{align*}
    For the upper bound, again by \eqref{auxeq:expbeck}, we have 
    \begin{align*}
       \frac{4}{p} \alpha_{p}(\mc{L}_{\rm depol}) \le \inf_{X \ge 0\,, \norm{X}_{1,\frac{\mi}{d} = 1}} \frac{\norm{X}_{p,\frac{\mi}{d}}^p}{\norm{X}_{p,\frac{\mi}{d}}^p - 1}\,.
    \end{align*}
     Note that $ \frac{x^p}{x^p - 1}$ is decreasing in $x \ge 1$ for $p \in (1,2]$, and that, by \eqref{eq:suprhodp} and definition of $D_p(\rho \| \si) $,
    \begin{align*}
        \sup_{X \ge 0\,, \norm{X}_{1,\frac{\mi}{d} = 1}} \norm{X}_{p,\frac{\mi}{d}} = d^{\frac{p-1}{p}}\,.
    \end{align*}
   It follows that $\alpha_{p}(\mc{L}_{\rm depol}) \le  \frac{p}{4} \frac{d^{p-1}}{d^{p-1} - 1}$. The proof is completed by Lemma \ref{lem:upppoin} and $2\alpha_{2}(\mc{L}_{\rm depol}) = \lad(\mc{L}_{\rm depol}) = 1$. 
   }
\end{proof}

\begin{remark}
\mb{When $p \to 1$, the estimate \eqref{est:depol}  gives $1/4 \le \alpha_1(\mc{L}_{\rm depol}) \le 1/2$, which recovers the known bound for MLSI constant for the depolarizing semigroup \cite[Figure 1]{muller2016relative}. It is also easy to see that \eqref{est:depol} is asymptotically tight for any fixed $p \in (1,2]$ when $d \to \infty$. Indeed, we have $\alpha_p(\mc{L}_{\rm depol}) \to p/4$ as $d \to \infty$. Another special case is $d = 2$, where we directly have 
$\alpha_p(\mc{L}_{\rm depol})  = 1/2$ from \eqref{est:depol}, which can also be implied by the relation \eqref{connec_class}: $\alpha_p(\mc{L}_{\rm depol})  = \alpha_{p,\frac{1}{2}}/2$, and $\alpha_{p,\frac{1}{2}} = 1$ in \cite[Proposition 4.3]{bobkov2006modified}.}
\end{remark}

\subsubsection*{Mixing time}  
\mb{We shall analyze the mixing time of a primitive QMS $\mc{P}_t = e^{t\mc{L}}$ from quantum $p$-Beckner's inequalities.} We define the $l_1$ mixing time for $\ep > 0$ by
\begin{align} \label{def:l1mix}
t_1(\ep) = \inf\{t > 0\,;\  \big\lVert\mc{P}^\dag_t(\rho) - \si\big\rVert_1 \le \ep \ \text{for all} \ \rho \in \dh \}\,.
\end{align}
To bound $t_1(\ep)$, we need the following lemma 
extending \cite[Theorem 4.1]{li2020complete} for the symmetric QMS, which characterizes the convergence of QMS in terms of $\si$-weighted $p$-norm. 

\begin{lemma} \label{lem:conver_p_norm}
Let $\mc{P}_t$ be a primitive QMS satisfying \mb{$\si$-{\rm KMS DBC}}. Then it holds that, \mb{for $p \in (1,2]$},
\begin{align*}
      \norm{\mc{P}_t(X) - \tr(\si X) \mi}_{p,\si} \le \mb{e^{- 2 \alpha_p(\mc{L})t}} \norm{X}_{p,\si}^{1-p/2} \sqrt{\frac{2}{p(p-1)} \left(  \norm{X}_{p,\si}^p - \norm{X}_{1,\si}^{p} \right)}\,, \q  X \ge 0\,,
\end{align*}
where $\alpha_p (\mc{L}) > 0$ is the quantum Beckner constant for $\mc{P}_t$. 
\end{lemma}

\begin{proof}
We define, for $X, Y \in \mc{B}_{sa}(\mc{H})$,  
\begin{equation*}
G_{X,Y}(s) := \norm{X + s Y}_{p,\si}^p - \frac{p(p-1)}{2} s^2 \norm{X + s Y}_{p,\si}^{p-2} \norm{Y}_{p,\si}^2.
\end{equation*}
Similarly to \cite[Theorem 4.1]{li2020complete}, by results in \cite{ricard2016noncommutative}, it is easy to prove that $G_{X,Y}''(0) \ge 0$ for any  self-adjoint $X, Y$, which implies that $G_{X,Y}(s)$ is a convex function on $\R$. We now consider $A := \tr(\si X) \mi $ and $B := X_t - \tr(\si X) \mi$ with $X_t = \mc{P}_t(X)$ for $X \ge 0$. Then, by Lemma \ref{lem:prop:norm_power}, we have, for any $s \in \R$,
\begin{equation*}
    \norm{A + s B}_{p,\si}^p \ge  \norm{A + s B}_{1,\si}^p \ge \tr(\si(A + s B))^p  = \tr(\si X)^p = \norm{A}_{p,\si}^p\,.
\end{equation*}
It then follows from definition that $G_{A,B}'(0) \ge 0$, which, along with the convexity of $G_{A,B}(s)$, yields $G'_{A,B}(s) \ge 0$ for any $s \ge 0$. Hence, we have $G_{A,B}(1) \ge G_{A,B}(0) $, that is (recalling the definitions of $A, B$), 
\begin{equation} \label{auxeq:decay}
     \norm{X_t}_{p,\si}^p - \frac{p(p-1)}{2} \norm{X_t}_{p,\si}^{p-2} \norm{X_t - \tr(\si X) \mi}_{p,\si}^2 \ge \norm{X}_{1,\si}^p\,.
\end{equation}
By \eqref{ineq_becp} and Gr\"{o}nwall’s inequality, it holds that 
\begin{align*}
    \norm{X_t}_{p,\si}^p - \norm{X_t}_{1,\si}^{p} \le \mb{e^{- 4 \alpha_p t}} \left(  \norm{X}_{p,\si}^p - \norm{X}_{1,\si}^{p} \right)\,,\q X \ge 0\,,
\end{align*}
which, along with \eqref{auxeq:decay}, implies 
\begin{align*}
     \norm{X_t - \tr(\si X) \mi}_{p,\si}^2 \le \frac{2}{p(p-1)} \mb{e^{- 4 \alpha_p t}} \norm{X}_{p,\si}^{2-p}\left(  \norm{X}_{p,\si}^p - \norm{X}_{1,\si}^{p} \right)\,.
\end{align*}
The proof is complete by taking the square root of the above inequality. 
\end{proof}

\begin{proposition} \label{prop:mixing}
Under the same assumption as in Lemma \ref{lem:conver_p_norm}, it holds that \mb{for $p \in (1,2]$},
\begin{align} \label{eq:mixing}
t_1(\ep) \le h(p,\si_{\min},\ep)\,,
\end{align}
where 
\begin{align*}
    h(p,\si_{\min},\ep) := \mb{\frac{1}{2\alpha_p(\mc{L})} }\log \bigg( \ep^{-1} \sqrt{\frac{2}{p(p-1)} \left(  \si_{\min}^{\frac{2}{p}-2} -  \si_{\min}^{ p + \frac{2}{p} -3}  \right)} \bigg)\,.
\end{align*}
\end{proposition}

\begin{proof}
Let $X_t$ be the relative density of $\rho_t = e^{t \mc{L}^\dag}(\rho)$; see \eqref{eq:evo_density}. We write
\begin{align*}
    \big\lVert\mc{P}^\dag_t(\rho) - \si\big\rVert_1  = \big\lVert \gs (\gs^{-1} \mc{P}^\dag_t(\rho) - \mi)\big\rVert_1 = \norm{X_t - \mi}_{1,\si}\,.
\end{align*}
By Lemma \ref{lem:conver_p_norm}, we have 
\begin{align}  \label{eq:sup_pnorm}
  \sup_{\rho \in \dh}  \big\lVert\mc{P}^\dag_t(\rho) - \si\big\rVert_1 \le \mb{e^{- 2 \alpha_p(\mc{L}) t}} \sup_{X \ge 0,\, \tr(\si X) = 1} \sqrt{\frac{2}{p(p-1)} \left(  \norm{X}_{p,\si}^2 - \norm{X}_{p,\si}^{2-p} \right)}\,. 
\end{align}
Recalling the formula \eqref{eq:suprhodp}, there holds 
\begin{align*}
    \sup_{X \ge 1,\, \tr(\si X) = 1} \norm{X}_{p,\si} = \si_{\min}^{\frac{1}{p}-1}\,.
\end{align*}
Also note that the function $x^2 - x^{2 - p}$ with $p \in (1,2]$ is increasing for $x \ge 1$. It follows from \eqref{eq:sup_pnorm} that 
\begin{align*}
      \sup_{\rho \in \dh}  \big\lVert\mc{P}^\dag_t(\rho) - \si\big\rVert_1 \le \mb{e^{- 2 \alpha_p(\mc{L}) t}} \sqrt{\frac{2}{p(p-1)} \left(  \si_{\min}^{\frac{2}{p}-2} -  \si_{\min}^{ p + \frac{2}{p} -3} \right)}\,. 
\end{align*}
By definition \eqref{def:l1mix} and a direct computation, we obtain the estimate \eqref{eq:mixing} for $t_1(\ep)$. 
\end{proof}
By elementary calculus, we find 
\begin{align*}
     h(2,\si_{\min},\ep) = \frac{1}{\lad(\mc{L})} \log \bigg(\ep^{-1} \sqrt{  \si_{\min}^{-1} -  1  } \bigg)\,,
\end{align*}
and, as $p \to 1^+$,
\begin{align*}
     h(p, \si_{\min}, \ep) \to \frac{1}{2\alpha_1(\mc{L})} \log \bigg(\ep^{-1} \sqrt{  2 \log \big(\si_{\min}^{-1}\big)} \bigg)\,,
\end{align*}
which are nothing else but the mixing time bounds obtained from the decays of the variance and the relative entropy, respectively \cite{kastoryano2013quantum,temme2010chi}. \mb{If the QMS satisfies the $\si$-{\rm GNS DBC}, then \eqref{ineq_becp} holds for all $p \in (1,2]$ by Theorem \ref{them:lowerbeck}. In principle, we can take the infimum in \eqref{eq:mixing} over $p \in (1,2]$ and obtain $t_1(\ep) \le \inf_{p \in (1,2]} h(p,\si_{\min},\ep)$. However, this observation might be not that useful in practice, since it is hard to have a good estimate of $\alpha_p(\mc{L})$ for a general $\mc{L}$ and answer when the infimum is attained in the interior of $[1,2]$. If the constants $\alpha_p(\mc{L})$ are of the same order, it is easy to see that when $\si_{\min} \le 1/d$ is small enough, we have $\inf_{p \in (1,2]} h(p,\si_{\min},\ep) = h(1,\si_{\min},\ep)$.}

\subsubsection*{Moment estimates and concentration inequalities} In this section, we consider the primitive symmetric QMS $\mc{P}_t = \mc{P}_t^\dag$ (see Remark \ref{eq:tracial}). We will derive a moment estimate from  $p$-Beckner's inequalities, by extending the arguments of \cite[Proposition 3.3]{adamczak2022modified} for classical Markov semigroups. This helps us to obtain concentration inequalities in a similar manner as \cite{junge2015noncommutative}, which could find applications in  quantum parameter estimation problems \cite[Section V]{rouze2019concentration}. 
We first recall  the carr\'{e} du champ operator (gradient form) associated with $\mc{P}_t = e^{t \mc{L}}$ \cite{junge2015noncommutative,wirth2021curvature}:  
\begin{align} \label{def:gamma}
    \Gamma(X,Y) = \frac{1}{2}(\mc{L}(X^*Y) - X^* (\mc{L}Y) - (\mc{L}X)^*Y)\quad \text{for}\  X,Y \in \mc{B}(\mc{H})\,.
\end{align}
As usual, we write $\Gamma(X)$ for $\Gamma(X,X)$. By the self-adjointness $\mc{L}^\dag = \mc{L}$ and $\mc{L}(\mi) = 0$, a direct computation gives the relation between the $\Gamma$ operator and the Dirichlet form $\mc{E}_{2,\mc{L}}(X,Y) = -\l X, \mc{L} Y \r_{\frac{\mi}{d}}$ (see \eqref{eq:productid} for the notation $\l \dd,\dd \r_{\frac{\mi}{d}}$) : 
\begin{align} \label{eq:rela_gamma}
\mc{E}_{2,\mc{L}}(X,Y) = \big\l \frac{\mi}{d}, \Gamma(X,Y)\big\r\,.   
\end{align}
Before we state our main result on moment estimates, we need the following useful lemma. 

\begin{lemma}\label{lem:diri_convex}
Let $\Gamma(\dd)$ be given in \eqref{def:gamma} for a symmetric Lindbladian $\mc{L}$. For any differentiable convex and increasing function $\vp: [0,\infty) \to \R$ and $c \in \R$, it holds that
\begin{align} \label{eq:diri_convex}
    \mc{E}_{2,\mc{L}}(\vp(|X + c|), |X + c|) \le 2 \l\vp'(|X + c|), \Gamma(X) \r_{\frac{\mi}{d}} \le 2 \norm{\vp'(|X+c|)}_{p,\frac{\mi}{d}} \norm{\Gamma(X)}_{\h{p},\frac{\mi}{d}}\,,
\end{align}
for $X \in \mc{B}_{sa}(\mc{H})$, where $p \ge 1$. 
\end{lemma}
\begin{proof}
Assume that $X \in \mc{B}_{sa}(\mc{H})$ has the spectral decomposition $X = \sum \lad_i E_i$, where $E_i$ are the eigen-projections associated with the eigenvalues $\lad_i$. Note from the convexity of $\vp$ that for any $x,y \ge 0$, 
\begin{align} \label{auxeq:convex_fun}
      \frac{\vp(x) - \vp(y)}{x - y} \le \max \{\vp'(x), \vp'(y)\} \le \vp'(x) + \vp'(y)\,,
\end{align} 
since $\vp'(x) \ge 0$ holds by the monotonicity of $\vp$. Recalling the formula \eqref{eq:inte_by_parts}, by the eigen-decomposition of $X$ and the inequality \eqref{auxeq:convex_fun}, we have
\begin{align} \label{auxeq_momen_1}
      \mc{E}_{2,\mc{L}}(\vp(|X+c|), |X + c|) & = \sum_{j = 1}^J \big\l \p_j \vp(|X+c|), \p_j |X + c| \big\r_{\frac{\mi}{d}} \notag \\
      & = \frac{1}{d} \sum_{j = 1}^J \sum_{i,k = 1}^d (\vp(|\lad_i+c|) - \vp(|\lad_k + c|))(|\lad_i + c| - |\lad_k+c|) \tr (E_k V_j E_i V_j) \notag \\
      & \le \frac{1}{d} \sum_{j = 1}^J \sum_{i,k = 1}^d (\vp'(|\lad_i+c|) + \vp'(|\lad_k + c|))(|\lad_i + c| - |\lad_k+c|)^2 \tr (E_k V_j E_i V_j)\,.
\end{align}
Similarly, by definition \eqref{def:gamma} of $\Gamma(X)$, we  compute 
\begin{align} \label{auxeq_momen_2}
  2 \big\l\vp'(|X+c|), \Gamma(X) \big\r_{\frac{\mi}{d}} = & \frac{1}{d} \sum_{j =1}^J \sum_{i,k = 1}^d \big\{(\vp'(|\lad_i+c|) - \vp'(|\lad_k + c|))(\lad^2_i - \lad_k^2) \notag \\
  & \qquad \qquad \  -  2(\lad_i\vp'(|\lad_i + c|) - \lad_k \vp'(|\lad_k + c|))(\lad_i - \lad_k) \big\} \tr (E_k V_j E_i V_j) \notag \\
  = & \frac{1}{d} \sum_{j = 1}^J \sum_{i,k = 1}^d (\vp'(|\lad_i+c|) + \vp'(|\lad_k + c|))(\lad_i - \lad_k)^2 \tr (E_k V_j E_i V_j)\,.
\end{align}
Since there holds $\big||\lad_i + c| - |\lad_k+c| \big| \le |\lad_i - \lad_k|$, by \eqref{auxeq_momen_1} and \eqref{auxeq_momen_2}, and using H\"{o}lder's inequality, we obtain the desired estimate \eqref{eq:diri_convex}. 
\end{proof}

\begin{proposition} \label{prop:moment}
\mb{Let $\mc{L}$ be a primitive symmetric Lindbladian and $\Gamma(\dd)$ be defined in \eqref{def:gamma}.}
Suppose that the quantum $p$-Beckner's inequality \eqref{ineq_becp} holds
for all $p \in (1,2]$ with $\alpha_p \ge a(p-1)^s$ for some $a > 0$ and $s \ge 0$. Then we have, for $X \in \mc{B}_{sa}(\mc{H})$ and $r \ge 2$,
\begin{align} \label{eq:moment}
    \norm{X - \mb{\frac{1}{d}\tr(X)} }_{r,\frac{\mi}{d}}^2 \le \frac{r^{s + 1}\kappa(s)}{\mb{2a}}\norm{\Gamma(X)}_{\frac{r}{2},\frac{\mi}{d}}\,,
\end{align}
where $\kappa(s) := (1 - e^{-(s+1)/2})^{-1}$. 
\end{proposition}
\begin{proof}
We shall prove by induction that for all positive integers $k$ and $r \in (k,k+1]$, there holds 
\begin{align} \label{eq:ineq_induc}
     \norm{X - \mb{\frac{1}{d}\tr(X)} }_{r,\frac{\mi}{d}}^2 \le c_r \norm{\Gamma(X)}_{\max\{\frac{r}{2},1\},\frac{\mi}{d}}\,, \quad \forall X \in \mc{B}_{sa}(\mc{H})\,,
\end{align}
where 
\begin{align} \label{def:constant_moment}
    c_r := \frac{r^{s + 1}\kappa_r(s)}{\mb{2a}} \,, \quad \kappa_r(s) := \Big( 1 - \big(  \frac{r-1}{r} \big)^{(s+1)r/2}\Big)^{-1}\,.
\end{align}
The desired estimate \eqref{eq:moment} is a direct consequence of \eqref{eq:ineq_induc}, since $\kappa_r$ increases in $r$ and $\kappa_r(s) \to \kappa(s)$ as $r \to \infty$.  We first note that by Lemma \ref{lem:upppoin}, \eqref{ineq_pi2} holds with $\lad(\mc{L}) \ge 2a$. For $k = 1$ and $r \in (1,2]$, by  \eqref{ineq_pi2} and \eqref{eq:rela_gamma}, we have
\begin{align*}
 \norm{X - \mb{\frac{1}{d}\tr(X)}}_{r,\frac{\mi}{d}}^2 \le \norm{X - \mb{\frac{1}{d}\tr(X)}}_{2,\frac{\mi}{d}}^2 \le \frac{1}{\lad(\mc{L})} \mc{E}_{2,\mc{L}}(X) \le c_r \norm{\Gamma(X)}_{1,\frac{\mi}{d}}\,,
\end{align*}
since $c_r$ is increasing in $r$, which gives $c_r \ge c_1 = \mb{1/2a} \ge 1/\lad(\mc{L})$. Suppose that \eqref{eq:ineq_induc} holds for all integers smaller than some $k > 1$. Now we consider $r \in (k,k+1]$. We define $Y = | X - \mb{\frac{1}{d}\tr(X)}|$, and then estimate by using \eqref{ineq_becp},
\begin{align} \label{auxeqpr_1}
\mb{r} \alpha_{\h{r}}\big(\norm{Y^{r-1}}_{\h{r},\frac{\mi}{d}}^{\h{r}} -  \norm{Y^{r-1}}_{1,\frac{\mi}{d}}^{\h{r}}\big) \le \mc{E}_{\h{r},\mc{L}}(Y^{r-1})\,.
\end{align}
By Lemma \ref{lem:diri_convex} with $\vp(x) = x^{r - 1}$, $c = - \mb{\frac{1}{d}\tr(X)}$, and $p = r/(r-2)$, we have 
\begin{align} \label{auxeqpr_2}
    \mc{E}_{\h{r},\mc{L}}(Y) = - \frac{\mb{\h{r} r}}{4}\l Y^{r-1}, \mc{L}Y \r_{\frac{\mi}{d}} & \le \frac{\mb{\h{r} r}}{2}(r - 1)\norm{Y^{r-2}}_{p,\frac{\mi}{d}} \norm{\Gamma(X)}_{\h{p},\frac{\mi}{d}} \notag \\
    & = \mb{\frac{r^2}{2}} \norm{Y}_{r,\frac{\mi}{d}}^{r-2} \norm{\Gamma(X)}_{\frac{r}{2},\frac{\mi}{d}}\,.
\end{align}
We write $l_r = \norm{Y}_{r,\frac{\mi}{d}}$ and note $\alpha_{\h{r}} \ge a(\h{r} - 1)^s = a (r - 1)^{-s}$. Then combining estimates \eqref{auxeqpr_1} and \eqref{auxeqpr_2} gives
\begin{align} \label{auxeqpr_3}
     l_r^r - l_{r-1}^{r} \le \mb{\frac{r(r-1)^s}{2a}} l_r^{r-2} \norm{\Gamma(X)}_{\frac{r}{2},\frac{\mi}{d}} \le \frac{r^{s+1}}{\mb{2a}} l_r^{r-2} \norm{\Gamma(X)}_{\frac{r}{2},\frac{\mi}{d}}\,. 
\end{align}
Applying the assumption \eqref{eq:ineq_induc} to bound $l_{r-1}$ and by \eqref{auxeqpr_3}, we obtain 
\begin{align} \label{auxeqpr_4}
     l_r^r \le \big(c_{r-1}\norm{\Gamma(X)}_{\max\{\frac{r-1}{2},1\},\frac{\mi}{d}}\big)^{\frac{r}{2}} + \frac{r^{s+1} }{\mb{2a}} l_r^{r-2} \norm{\Gamma(X)}_{\frac{r}{2},\frac{\mi}{d}}\,.
\end{align}
Note from \eqref{def:constant_moment} that 
\begin{align} \label{auxeqpr_5}
    \frac{c_{r-1} \norm{\Gamma(X)}_{\max\{\frac{r-1}{2},1\},\frac{\mi}{d}}}{c_r \norm{\Gamma(X)}_{\frac{r}{2},\frac{\mi}{d}}} \le \frac{c_{r-1}}{c_r} \le \frac{(r-1)^{s+1}\kappa_{r-1}(s)}{r^{s+1}\kappa_r(s)}\,.
\end{align}
By dividing both sides of \eqref{auxeqpr_4} by $(c_r\norm{\Gamma(X)}_{\frac{r}{2},\frac{\mi}{d}})^{r/2}$ and using \eqref{auxeqpr_5}, it follows that 
\begin{align} \label{auxeqpr_6}
    \Big( \frac{l_r^2}{c_r\norm{\Gamma(X)}_{\frac{r}{2},\frac{\mi}{d}}}\Big)^{\frac{r}{2}} \le \Big( \frac{r-1}{r} \Big)^{\frac{(s+1)r}{2}} + \frac{1}{\kappa_r} \Big( \frac{l_r^2}{c_r \norm{\Gamma(X)}_{\frac{r}{2},\frac{\mi}{d}}} \Big)^{\frac{r-2}{2}}\,.
\end{align}
To complete the proof, we need the fact from \cite[Proposition 3.1]{adamczak2022modified} that the function $h(x) = (1-1/r)^{(1+s)r/2} + \kappa_r^{-1} x^{1-2/r} - x$ is strictly concave on $[0,\infty)$ and satisfies $h(0) > 0$ and $h(1) = 0$, which means that $h(x) \ge 0$ implies $x \le 1$. This fact, along with \eqref{auxeqpr_6}, readily implies 
\begin{equation*}
   \norm{X - \mb{\frac{1}{d}\tr(X)}}_{r,\frac{\mi}{d}}^2 \le c_r \norm{\Gamma(X)}_{\frac{r}{2},\frac{\mi}{d}}\,. \qedhere
\end{equation*}
\end{proof}
\begin{remark}
In the case of $s = 0$, Proposition \ref{prop:moment} shows that if \eqref{ineq_becp} holds with $\inf_{p\in(1,2]}\alpha_p \ge a$, then
\begin{align*} 
    \norm{X - \mb{\frac{1}{d}\tr(X)}}^2_{r,\frac{\mi}{d}} \le  \frac{r\kappa}{\mb{2a}} \norm{\Gamma(X)}_{\infty}\,,\q \forall X \in \mc{B}_{sa}(\mc{H})\,,
\end{align*} 
by $\norm{\Gamma(X)}_{\frac{r}{2},\frac{\mi}{d}} \le  \norm{\Gamma(X)}_{\infty}$, where $\kappa := (1 - e^{-1/2})^{-1}$.  A similar result was given in \cite[Theorem 4.4]{junge2015noncommutative} for non-primitive symmetric QMS under the Bakry-\'{E}mery curvature-dimension condition ($\Gamma_2$-criterion): 
\begin{align}\label{def:becdc}
 \Gamma_2(X,X) \ge \alpha \Gamma(X,X)\,, \q \forall X \in \bh\,, \tag{BE$(\alpha,\infty)$}
\end{align}
for some $\alpha > 0$, where $\Gamma_2$ is the iterated carr\'{e} du champ operator: $\Gamma_2(X,Y): = - \frac{1}{2}(\Gamma(X,\mc{L} Y) + \Gamma(\mc{L}X,Y) - \mc{L}\Gamma(X,Y))$. To be precise, they used the noncommutative martingale methods and proved that
under some necessary regularity condition, if \ref{def:becdc} holds, then we have
\begin{align*}
    \norm{X - E(X)}_r^2 \le \frac{8r}{\alpha} \norm{\Gamma(X)}_{\infty}\,,\q \forall X \in \mc{B}_{sa}(\mc{H})\,,
\end{align*}
where $E$ is the conditional expectation to the fixed point algebra $\{X\,;\ \mc{P}_t(X)= X\ \text{for}\  t \ge 0\}$ and $\norm{\dd}_r$ is the Schatten norm with respect to a normal faithful tracial state. It is currently unknown whether or not Beckner's inequality can be implied from $\Gamma_2$-criterion \ref{def:becdc}.  Hence, our result is complementary to theirs.  
\end{remark}
 
Similarly to \cite[Corollary 4.13]{junge2015noncommutative}, we next show that the moment estimate \eqref{eq:moment} with $s= 0$ implies a noncommutative exponential integrability and a Gaussian concentration inequality. The idea of the proof is standard and borrowed from the commutative case \cite{ben2008poincare} by Efraim and Lust-Piquard.

\begin{corollary} \label{cor:concern}
\mb{Let $\mc{L}$ be a primitive symmetric Lindbladian and $\Gamma(\dd)$ be defined in \eqref{def:gamma}.}
Suppose that $p$-Beckner's inequality \eqref{ineq_becp} holds
 with \mb{$\inf_{p\in(1,2]}\alpha_p \ge a$} for some $a > 0$.
 Then it holds that, for $X \in \mc{B}_{sa}(\mc{H})$,
\begin{equation} \label{eq:ex_int}
    \norm{\exp\Big(\Big|X- \mb{\frac{1}{d}\tr(X)} \Big|\Big)}_{1,\frac{\mi}{d}} \le 2 \exp\bigg(  \frac{e\kappa \norm{\Gamma(X)}_{\infty}}{\mb{4 a}}  \bigg)\,,
\end{equation}
and, for any $t > 0$,
\begin{align} \label{eq:concentration}
      \frac{1}{d}\tr\Big(1_{[t,\infty)}\Big(\Big|X- \mb{\frac{1}{d}\tr(X)}\Big|\Big)\Big) \le 2 \exp \left( -
      \frac{a t^2}{\mb{e\kappa}\norm{\Gamma(X)}_{\infty}}
      \right)\,,
\end{align}
where $\kappa = (1 - e^{-1/2})^{-1}$. 
\end{corollary}

\begin{proof}
Note that $\Gamma(X) = \Gamma\big(X - \mb{\frac{1}{d}\tr(X)}\big)$. Without loss of generality, we assume $\mb{\frac{1}{d}\tr(X)} = 0$.  By functional calculus and
Proposition \ref{prop:moment}, we obtain 
\begin{align}  \label{auxeq:ex_int}
    \frac{1}{2d} \tr (e^{|X|}) &\le  \frac{1}{d}\tr(\cosh X) = 1 + \sum_{j = 1}^\infty \frac{1}{(2j)!}\norm{X}^{2j}_{2j, \frac{\mi}{d}} \notag \\
    & \le 1 + \sum_{j = 1}^\infty \frac{1}{(2j)!} \frac{(2j)^{j}\kappa^j}{\mb{(2a)}^j}\norm{\Gamma(X)}_{j,\frac{\mi}{d}}^j \notag \\
    & \le 1 + \sum_{j = 1}^\infty \frac{j^j}{j! (2 j -1)!!} \bigg( \frac{\kappa \norm{\Gamma(X)}_{j,\frac{\mi}{d}} }{\mb{2a}} \bigg)^j \notag \\
    & \le 1 + \sum_{j = 1}^\infty \frac{1}{j!} \bigg( \frac{e\kappa \norm{\Gamma(X)}_{j,\frac{\mi}{d}}}{\mb{4a}} \bigg)^j\,, 
\end{align}
where in the last inequality we have used $\frac{j^j}{(2j-1)!!} \le \left( \frac{e}{2} \right)^j$ for all $j \in \mathbb{N}$ from Stirling's formula. Then the inequality \eqref{eq:ex_int} follows from \eqref{auxeq:ex_int} and again $\norm{\Gamma(X)}_{j,\frac{\mi}{d}} \le \norm{\Gamma(X)}_\infty$. For the concentration inequality \eqref{eq:concentration}, by \eqref{eq:ex_int} and Chebyshev inequality, we have, for any $\lad > 0$, 
\begin{align*}
    \frac{1}{d}\tr\big(1_{[t,\infty)}(|X|)\big) \le e^{-\lad t} \frac{1}{d} \tr \big( e^{\lad |X|} \big) \le 2 \exp\bigg(  \frac{e\kappa \norm{\Gamma(X)}_{\infty}}{\mb{4a}} \lad^2 - t \lad \bigg)\,.
\end{align*}
Then, choosing $\lad$ to minimize the right-hand side of the above estimate gives the desired \eqref{eq:concentration}.
\end{proof}

\begin{remark}
\mb{The classical Gaussian-type concentration inequalities can be derived by various approaches; see \cite{raginsky2013concentration} for a review. In particular, Marton \cite{marton1996bounding,marton1986simple} showed that 
Gaussian concentration can be implied from a  transportation cost inequality (${\rm TC}_1$) associated with Wasserstein distance of order one $W_1$. Later, Bobkov and G\"{o}tze \cite{bobkov1999exponential} proved the exponential integrability (EI) by a log-Sobolev inequality and established the equivalence between EI and ${\rm TC}_1$ (note that Gaussian concentration is an easy consequence of EI). Such an approach  
turns out to be very useful for noncommutative concentration inequalities and has been exploited in different settings. We refer the interested readers to \cite[Theorem 4.18]{junge2015noncommutative}, \cite[Proposition 6.13]{gao2020fisher}, \cite[Theorem 3]{de2021quantum1}, and \cite[Theorem 8]{rouze2019concentration} for recent results, all of which involve a quantum version of 
 Lipschitz constant $\norm{O}_{\text{Lip}}$ for observables $O$ and an associated 1-Wasserstein distance defined by $W_1(\rho,\si) = \sup\{|\tr(\rho O) - \tr(\si O)|\,;\ O \ \text{self-adjoint}\,,\ \norm{O}_{\text{Lip}} \le 1\}$. 

In our case \eqref{eq:concentration}, the Lipschitz constant is defined by $\norm{X}_{\text{Lip},1}: = \norm{\Gamma(X)}^{1/2}_\infty$, which is a 
natural generalization from the classical case \cite[Definition 3.3.24]{bakry2014analysis}. Such a definition has been used for defining a  quantum 
$W_1$ distance \cite[Definition 4.15]{junge2015noncommutative} and deriving concentration inequalities \cite[Corollary 4.13]
{junge2015noncommutative}; see also \cite{gao2020fisher} for similar results for non-ergodic semigroups. An alternative definition of the Lipschitz constant associated with $\si$-{\rm GNS} symmetric QMS is given by Rouz\'{e} and Datta \cite[Section E.2]{rouze2019concentration}: $\norm{X}_{\text{Lip},2}: = (\sum_j (e^{-\ww_j/2} + e^{\ww_j/2})\norm{\p_j X}^2_\infty)^{1/2}$, using the differential structure in Lemma \ref{lem:struc_gene}. It is clear that both $\norm{X}_{\text{Lip},1}$ in \cite{junge2015noncommutative,gao2020fisher} and $\norm{X}_{\text{Lip},2}$ in \cite{rouze2019concentration} depend on the generator $\mc{L}$ of a QMS. In \cite{de2021quantum1}, De Palma et al. suggested another way to define the Lipschitz constant (denoted by $\norm{X}_{\text{Lip},3}$) and $W_1$ distance, which is based on the notion of neighboring states instead of the QMS. 
However, there is no consensus on the definition of $\norm{\dd}_{\text{Lip}}$, and each one has its own interest and advantage. In certain scenarios, they can be compared with each other. For instance, for a symmetric Lindbladian $\mc{L}$, by \eqref{eq:rep_gen} and the Leibniz rule for $\p_j$, we have 
$\Gamma(X) = \sum_j (\p_j X)^* (\p_j X)$, which gives
\begin{align*}
    \norm{X}_{\text{Lip},1}^2 = \Big\lVert\sum_j (\p_j X)^* (\p_j X)\Big\rVert_\infty \le \sum_j \norm{\p_j X}^2_\infty = \frac{1}{2} \norm{X}^2_{\text{Lip},2}\,.
\end{align*}
Moreover, \cite[Proposition\,10]{de2022quantum} compares $\norm{X}_{\text{Lip},2}$ and $\norm{X}_{\text{Lip},3}$ in the case of $\si$ being the Gibbs state of a local commuting Hamiltonian.}
\end{remark}

\begin{remark}
\mb{The concentration inequality \eqref{eq:concentration} above is derived from \eqref{ineq_becp}, thanks to the moment estimate \eqref{eq:moment}. Due to some technical issues, we are not able to extend \eqref{eq:moment} to general QMS satisfying $\si$-{\rm GNS DBC}, which is more desirable. Nevertheless, similarly to \cite{rouze2019concentration}, we can bypass the moment estimate to obtain a Gaussian concentration, by a geometric characterization of \eqref{ineq_becp} that will be explored in Section \ref{sec:qot_beck}. To be precise, in Proposition \ref{propa} below, we imply from \eqref{ineq_becp} a transport cost inequality associated with the distance $W_{2,p}$ in \eqref{def:wp1}. As in \cite[Lemma\,6]{rouze2019concentration}, we can also show that $W_{1,\mc{L}}(\rho,\si) \le C W_{2,p} (\rho,\si)$ for some constant $C > 0$, where $W_{1,\mc{L}}$ is 1-Wasserstein distance defined in \cite{rouze2019concentration}. Then a concentration inequality can follow in the same way as \cite[Theorems 3,\,8]{rouze2019concentration}. }
\end{remark}



\section{Generalized quantum optimal transport} \label{sec:qot_beck}

\mb{In this section, we limit our discussion to the primitive QMS satisfying $\si$-{\rm GNS DBC} for some $\si \in \dhh$ and only consider $\mc{F}_{p,\si}$ with $p \in (1,2]$ for clarity.}

\subsection{Gradient flow of quantum \texorpdfstring{$p$} --divergence} \label{sec:gradient}

This subsection aims to identify the QMS $\mc{P}_t^\dag = e^{t \mc{L}^\dag}$ with $\si$-{\rm GNS DBC} as the gradient flow 
of quantum $p$--divergence $\mc{F}_{p,\si}$ with respect to some Riemannian metric $g_{p,\mc{L}}$ constructed below. The argument follows closely with those in \cite{carlen2017gradient} (see also \cite{otto2001geometry} for the classical case). 
The necessary and sufficient conditions for the existence of such a Riemannian metric will also be discussed.

For our purpose, we first compute the functional derivative $\d_\rho \mc{F}_{p,\si}(\rho)$ of $\mc{F}_{p,\si}$, which is defined by 
\begin{align*}
    \frac{d}{d t}\Big|_{t = 0} \mc{F}_{p,\si}(\rho_t)  = \l  \d_\rho \mc{F}_{p,\si}(\rho), \dot{\rho}\r\,,
\end{align*}
where $\rho_t: (-\ep,\ep) \to \dhh$ with $\ep > 0$ is any 
smooth curve satisfying $\rho_0 = \rho$. Similarly to \eqref{eq:ep_divi}, we find 
\begin{equation} \label{eq:deriv_func}
    \d_\rho \mc{F}_{p,\si}(\rho) = \frac{1}{p-1}
\gs^{-1/\h{p}}\big(\big(\gs^{-1/\h{p}}(\rho)\big)^{p-1}\big)\,, \q \forall \mb{\rho \in \dhh}\,.
\end{equation}
Then, defining $Y := \gs^{-1/\h{p}}(\rho)$, by \eqref{eq:partialpower_2} and Lemma \ref{lem:chain_rule},
we derive, for $1 \le j \le J$, 
\begin{align} \label{eq:deriv_entropy}
    \p_j  \d_\rho \mc{F}_{p,\si}(\rho) & =  \frac{1}{p-1} \p_j\gs^{-1/\h{p}}\left(Y^{p-1}\right) \notag \\ & = \frac{1}{p-1} \gs^{-1/\h{p}}\left(V_j \left(e^{-\ww_j/2p} Y\right)^{p-1} - \left(e^{\ww_j/2p} Y\right)^{p-1} V_j\right) \notag \\
    & = \gs^{-1/\h{p}} \left(f_p^{[1]}\left(e^{\ww_j/2p} Y,  e^{-\ww_j/2p} Y \right) \left(V_j \left(e^{-\ww_j/2p} Y \right) - \left(e^{\ww_j/2p} Y\right) V_j \right)\right),
\end{align}
with $f_p$ given in \eqref{def:funfp}.
Again by \eqref{eq:partialpower_2} and $\p_{j,\si} = \gs \p_j \gs^{-1}$ in \eqref{eq:adjpjkms}, it follows that
\begin{align} \label{eq:deriv_entropy_aux}
    V_j \left(e^{-\ww_j/2p} Y\right) - \left(e^{\ww_j/2p} Y\right) V_j &= \gs^{1/p} \left( \p_j \gs^{-1/p} (Y) \right)  \notag \\ & =  \gs^{1/p} \left( \p_j \gs^{-1} (\rho) \right) =  \gs^{-1/\h{p}} (\p_{j,\si}\rho)\,.
\end{align}
Combining \eqref{eq:deriv_entropy} and \eqref{eq:deriv_entropy_aux} readily gives
\begin{align} \label{eq:double_deriv}
      \p_j  \d_\rho \mc{F}_{p,\si}(\rho) = \gs^{-1/\h{p}} \left(f_p^{[1]}\left(e^{\ww_j/2p} \gs^{-1/\h{p}}(\rho),  e^{-\ww_j/2p}  \gs^{-1/\h{p}}(\rho)\right)\left(\gs^{-1/\h{p}} (\p_{j,\si} \rho)\right)\right).
\end{align}

We next define the Riemannian structure, associated with $\mc{F}_{p,\si}(\rho)$, on the matrix manifold $\mc{D}_+(\mc{H})$, that is, a family of inner products on the tangent space $T_\rho = \mc{B}^0_{sa}(\mc{H}):= \{X \in \mc{B}_{sa}(\mc{H})\,;\ \tr(X) = 0\}$ that depends smoothly on $\rho$. 
Motivated by \eqref{eq:double_deriv}, we first define the operator $ [\rho]_{p,\ww}: \bh \to \bh$ for $\rho \in \mc{D}_+(\mc{H})$ and $\ww \in \R$ by
\begin{align} \label{def:kernel_multiplication}
    [\rho]_{p,\ww} = \gs^{1/\h{p}} \circ
    \theta_p\left(e^{\ww/2p}  \gs^{-1/\h{p}}(\rho),  e^{-\ww/2p}  \gs^{-1/\h{p}}(\rho)\right) \circ
    \gs^{1/\h{p}}\,,
\end{align}
where
\begin{align} \label{def:theta_p}
     \theta_p(x,y) := \big(f_p^{[1]}(x,y)\big)^{-1}
     = \begin{cases}
         (p-1) \frac{x - y}{x^{p-1} - y^{p-1}}\,, & x \neq y\,, \\
         x^{2-p} \,,\q  & x = y\,.
     \end{cases}
\end{align}
It is clear that when $\si = \mi/d$ and $\ww = 0$, we have $[\rho]_{p,\ww} A = d^{1-p} \rho^{2-p} A = \gs( (\gs^{-1}(\rho))^{2-p}) A$ for any $A$ that commutes with $\rho$. Thus, $ [\rho]_{p,\ww}$ can be regarded as 
a noncommutative analog of the multiplication by the relative density $(\gs^{-1}\rho)^{2-p}$ with respect to the reference state $\si$. Noting $\gs^{-1/\h{p}}(\rho) > 0$, by Lemma \ref{lem:double_inner} the operator $[\rho]_{p,\ww}$ is
evidently invertible. Then, we immediately obtain from \eqref{eq:double_deriv} that
\begin{align} \label{auxeq_deriva}
     \p_j  \d_\rho \mc{F}_{p,\si}(\rho) = [\rho]_{p,\ww_j}^{-1}\p_{j,\si}\rho\,.
\end{align}
With the help of $[\rho]_{p,\ww}$, we now introduce the operator:
\begin{align} \label{def:dprho} 
\mf{D}_{p,\rho}(A) := \sum_{j = 1}^J \p_j^\dag \left([\rho]_{p,\ww_j} \p_j A \right):\ \bh \to \bh\,,
\end{align}
which is crucial for defining the desired Riemannian metric on $\mc{D}_+(\mc{H})$; see Definition \ref{def:metric}. The next lemma gives the main properties of $[\rho]_{p,\ww}$ and $\mf{D}_{p,\rho}$ and can be proved in the same manner as \cite[Lemma 5.8]{carlen2017gradient} and \cite[Lemma 7.3]{carlen2020non}. Hence we omit its proof.

\begin{lemma} \label{lem:basic_krho} 
Let $\rho \in \mc{D}_+(\mc{H})$, $\ww \in \R$.  It holds that 
\begin{enumerate}[1.]
    \item  $\l \dd, [\rho]_{p,\ww}(\dd)\r$ gives an inner product on $\bh$. Moreover, both $[\rho]_{p,\ww}$ and $[\rho]_{p,\ww}^{-1}$ are $C^\infty$ on  $\mc{D}_+(\mc{H})$. 
    \item  $\mf{D}_{p,\rho}$ is a positive semidefinite operator on $\bh$ and preserves self-adjointness. Moreover, we have 
\begin{align*}
    \ran(\mf{D}_{p,\rho}) = \ran(\mc{L}^\dag) = \ran(\ddiv)\,,\q \ker(\mf{D}_{p,\rho}) = \ker(\mc{L})=\ker(\na)\,.
\end{align*} 
\end{enumerate}
\end{lemma}
Recalling $\ker(\mc{L}) = {\rm span}\{\mi_{\mc{H}}\}$ for a primitive QMS $\mc{P}_t = e^{t \mc{L}}$, a direct consequence of the above lemma is that $\mf{D}_{p,\rho}$ for $\rho \in \mc{D}_+(\mc{H})$ is a positive definite operator on $\mc{B}_{sa}^0(\mc{H}) = (\ker(\mf{D}_{p,\rho}))^\perp$ that depends $C^\infty$ on $\rho$.
These facts allow us to define the following class of Riemannian metrics on $\mc{D}_+(\mc{H})$.
\begin{definition}\label{def:metric}
For each $\rho \in \mc{D}_{+}(\mc{H})$, we define the metric tensor $g_{p,\rho}$ on the tangent space $T_\rho = \mc{B}_{sa}^0(\mc{H})$ by 
\begin{align} \label{def:metric_tensor}
    g_{p,\rho}(\nu_1, \nu_2) := \l \mf{D}_{p,\rho}^{-1} (\nu_1), \nu_2 \r \,,\q \nu_1, \nu_2 \in T_\rho\,.
\end{align}
\end{definition}
We introduce $U_i := \mf{D}_{p,\rho}^{-1} (\nu_i) \in \mc{B}_{sa}^0(\mc{H})$ and define the inner product on $\mc{B}(\mc{H})^J$: 
\begin{align*}
    \l {\bf A}, {\bf B}\r_{p,\rho} := \sum_{j = 1}^J \l A_j, [\rho]_{p,\ww_j} B_j\r \q \text{for} \ {\bf A}, {\bf B} \in \mc{B}(\mc{H})^J\,.
\end{align*}
Then, by \eqref{def:dprho}, the metric tensor $g_{p,\rho}$ can be rewritten as 
\begin{align} \label{def:gp2}
     g_{p,\rho}(\nu_1, \nu_2) = \sum_{j = 1}^{J} \l \p_j U_1, [\rho]_{p,\ww_j} \p_j U_2\r = \l \na U_1, \na U_2\r_{p,\rho}\,.
\end{align}
We are now ready to conclude Proposition \ref{prop:gradient_flow} below. We first recall that the Riemannian gradient of $\mc{F}_{p,\si}$ 
with respect to the metric $g_{p,\rho}$, denoted by
$\grad \mc{F}_{p,\si}(\rho)$, 
is determined by the relation:
\begin{align*}
 g_{p,\rho}({\rm grad}\mc{F}_{p,\si}(\rho), \nu) = \l \d_\rho \mc{F}_{p,\si}(\rho), \nu\r \,, \quad \forall \nu \in T_\rho\,.
\end{align*}
Then, by Definition \ref{def:metric} and formulas \eqref{auxeq_deriva} and \eqref{def:dprho}, it follows that
\begin{align} \label{eq:riemannian_grad}
    {\rm grad} \mc{F}_{p,\si}(\rho) = \mf{D}_{p,\rho} (\d_\rho \mc{F}_{p,\rho}(\rho)) = - \mc{L}^\dag(\rho)\,,
\end{align} 
and the following result holds. 
\begin{proposition}\label{prop:gradient_flow}
The dual primitive QMS, $\rho_t = e^{t \mc{L}^\dag}(\rho)$ for $\rho \in \dhh$, satisfying $\si$-{\rm GNS DBC} is the gradient flow of $p$--divergence $\mc{F}_{p,\si}$ for $p \in (1,2]$ with respect to the Riemannian metric $g_{p,\rho}$ in \eqref{def:metric_tensor}:
\begin{align*}
    \p_t{\rho}_t = - {\rm grad} \mc{F}_{p,\si}(\rho_t) = \mc{L}^\dag(\rho_t)\,.
\end{align*}
\end{proposition}

A direct consequence of Proposition  \ref{prop:gradient_flow} above is the decrease of $\mc{F}_{p,\si}(\rho)$ along the quantum dynamic $\rho_t = \mc{P}_t^\dag(\rho)$ (also recall \eqref{eq:ep_divi}):
\begin{align} \label{eq:grad_entropy}
    \frac{d}{d t} \mc{F}_{p,\si}(\rho_t) & = g_{p,\rho_t}({\rm grad} \mc{F}_{p,\si}(\rho_t),\dot{\rho}_t) = - g_{p,\rho_t}(\mc{L}^\dag(\rho_t), \mc{L}^\dag(\rho_t)) \le 0\,.
\end{align} 
We shall see in Section \ref{sec:geodesic_convexity} that the decay rate of $\mc{F}_{p,\si}(\rho_t)$, characterized by the quantum Beckner's inequality (cf.\,\eqref{eq:exp_divi} and \eqref{eq_beck}), naturally connects with the geodesic convexity of the functional $\mc{F}_{p,\si}$. 

It was proved by Dietert \cite{dietert2015characterisation} that if a Markov chain with finite states can be formulated as the gradient flow of the relative entropy, then it must be time-reversible. This result, along with \cite{maas2011gradient}, implies that the reversible Markov chains are exactly those that can be characterized by the gradient flows. We next give the necessary condition for $\mc{P}_t^\dag$ being the gradient flow of $\mc{F}_{p,\si}$ for some Riemannian metric, which extends the previous results \cite{carlen2020non,cao2019gradient}.  Before we proceed, we give the following lemma that will be useful in the sequel. 

\begin{lemma} \label{lem:key_rela}
For $\si \in \dhh$, let $[\si]_{p,0}$ and $[\si]_{p,\ww}$ be defined as in \eqref{def:kernel_multiplication} \mb{with $p \in \R \backslash \{0,1\}$}. Then there holds
\begin{equation} \label{auxeqq_rela_00}
    [\si]_{p,0} = J_\si^{\kappa^{-1}_{1/p}} = R_\si \kappa_{1/p}^{-1}(\Delta_\si)\,,
\end{equation}
where the function $\kappa_{1/p}$ is the power difference \eqref{def:pdiff}. \mb{In addition, for $\mc{L}$ satisfying $\si$-{\rm GNS DBC}, we have}
\begin{align} \label{auxeq_repinn}
     - \l X, \mc{L} X \r_{\si,\vp_p} = \mb{\sum_{j = 1}^J}  \l \gs \p_j X, [\si]_{p,\ww_j}^{-1} \gs \p_j X\r\,,
\end{align}
and 
\begin{align} \label{auxeqq_rela}
  [\si]_{p,\ww_j} \circ \p_j \circ [\si]_{p,0}^{-1}   =   \gs \circ \p_j  \circ \gs^{-1}\,.
\end{align}
We also have 
\begin{align} \label{rep_dirichlet}
  \mb{\mc{E}_{p,\mc{L}}(\gs^{-1}(\rho))
         = \frac{p^2}{4} \sum_{j = 1}^J \left\l \p_{j, \si} \rho ,  [\rho]_{p,\ww_j}^{-1} \p_{j, \si} \rho  \right\r\,, \q \rho \in \dhh\,.}
\end{align}
\end{lemma}

\begin{proof}
The first identity \eqref{auxeqq_rela_00} follows from \eqref{eq:kernel_iden}, \eqref{eq:kernel_iden_2}, \eqref{auxeq_power_diff}, and definition \eqref{def:kernel_multiplication}. 
The representation \eqref{auxeq_repinn} is a reformulation of \eqref{eq:rep_dbcl} by $[\si]_{p,\ww}$. Note from formulas \eqref{eq:rep_gen} and \eqref{auxeqq_rela_00} that 
\begin{align*} 
  - \l X, \mc{L} X \r_{\si,\vp_p} =  \sum_{j = 1}^J \l \gs X, [\si]_{p,0}^{-1} \p_j^\dag \gs \p_j X\r\,,
\end{align*}
which, along with \eqref{auxeq_repinn}, gives the identity \eqref{auxeqq_rela}. The last one \eqref{rep_dirichlet} is reformulated from \eqref{eq:rep_epl} by using operators $\p_{j,\si}$ in \eqref{eq:adjpjkms} and $[\rho]_{p,\ww}$ in  \eqref{def:kernel_multiplication}. 
\end{proof}

\begin{proposition} \label{prop:necessary}
If a primitive QMS $\mc{P}_t^\dag = e^{t \mc{L}^\dag}$ is the gradient flow of $\mc{F}_{p,\si}(\rho)$ for $p \in (1,2]$ with respect to some smooth Riemannian metric $g_\rho$, then $\mc{L}$ is self-adjoint with respect to the inner product:
\begin{equation} \label{def:nece_inner}
    \l X , Y \r_{[\si]_{p,0}}: = \l X, [\si]_{p,0} Y \r\,,\q X, Y \in \bh\,.
\end{equation}
\end{proposition}

\begin{proof}
Note that any Riemannian metric $g_\rho$ on $\dhh$ can be written as 
\begin{align*}
    g_\rho(\nu_1,\nu_2) = \l \mf{D}_\rho^{-1} \nu_1, \nu_2  \r\,,\q \nu_1,\nu_2 \in T_\rho\,,
\end{align*}
where $\mf{D}_\rho$ is a positive definite operator from the  cotangent space $T^*_\rho$ to the tangent space $T_\rho$ at $\rho$ (both $T_\rho$ and $T^*_\rho$ can be identified with $\mc{B}^0_{sa}(\mc{H})$). The dual QMS $\mc{P}_t^\dag$ is the gradient flow of $\mc{F}_{p,\si}(\rho)$ with respect to $g_\rho$ means that 
\begin{align} \label{eq:grad_1}
    \mc{L}^\dag \rho = - \mf{D}_{\rho}(\d_\rho \mc{F}_{p,\si}(\rho))\,.
\end{align}
Substituting $\rho = \si + \ep X$ for $X \in \mc{B}^0_{sa}(\mc{H})$ into \eqref{eq:grad_1} above and differentiating it with respect to $\ep$ gives 
\begin{align} \label{eq:grad_2}
    \mc{L}^\dag X = - \mf{D}_{\rho} \gs^{-1/\h{p}}f_p^{[1]}(\si^{1/p},\si^{1/p})\gs^{-1/\h{p}}(X) = - \mf{D}_{\rho}[\si]^{-1}_{p,0} X\,,
\end{align}
by \eqref{eq:deriv_func} and \eqref{eq:chain_curve}. Hence, it follows from \eqref{eq:grad_2} that, for $X, Y \in \mc{B}_{sa}^0(\mc{H})$, 
\begin{align*}
   \l \mc{L}Y,  X\r_{[\si]_{p,0}} = - \l Y, \mf{D}_\rho X \r = - \l \mf{D}_\rho Y,  X \r = \l Y, \mc{L} X \r_{[\si]_{p,0}}\,,
\end{align*}
which implies that $\mc{L}$ is self-adjoint on $\mc{B}(\mc{H})$ with respect to $\l \dd,\dd \r_{[\si]_{p,0}}$, since $\mc{L}$ and $[\si]_{p,0}$ are Hermitian-preserving and satisfy
$\mc{L}(\mi) = 0$ and $[\si]_{p,0}\mi = \si$. 
\end{proof}

In the limiting case $p \to 1^+$, which corresponds to the relative entropy $D(\rho \| \si)$, the inner product $ \l \dd , \dd \r_{[\si]_{p,0}}$ in \eqref{def:nece_inner}  reduces to the BKM inner product:
\begin{equation*}
    \l X, Y \r_{{\rm BKM}}: = \int_0^1 \tr(\si^{1-t} X^* \si^t Y)\ dt\,,
\end{equation*}
hence our result generalizes \cite[Theorem 2.9]{carlen2020non} by Carlen and Maas. In the case of $p = 2$, where $\mc{F}_{p,\si} $ is the variance (up to a constant factor), $ \l \dd , \dd \r_{[\si]_{p,0}}$ becomes the familiar KMS inner product $\l \dd,\dd \r_{\si,1/2}$. Similarly to Definition \ref{def:sidbc}, it is convenient to 
say that $\mc{L}$ satisfies the $[\si]_{p,0}$-DBC if it is self-adjoint with respect to $\l\dd,\dd\r_{[\si]_{p,0}}$ (the cases $p = 1$ and $p = 2$ are the known BKM DBC and KMS DBC, respectively). In Appendix \ref{app:dbc}, we show that the class of QMS satisfying the $[\si]_{p,0}$-DBC with $p \in (1,2)$ is strictly larger than the class of QMS with $\si$-{\rm GNS DBC}. Thus, in view of Propositions \ref{prop:gradient_flow}
and \ref{prop:necessary}, there is a gap between the necessary and sufficient conditions for a QMS being the gradient flow of  $\mc{F}_{p,\si}(\rho)$. Very recently, Brooks and Maas \cite{brooks2022characterisation} characterizes the gradient flow for a given functional on a general smooth manifold, and show that the BKM DBC is the necessary and sufficient condition for a primitive QMS being the gradient flow of $D(\rho\|\si)$, which closes the gap in the case $p = 1$. We next further exploit the arguments in \cite{brooks2022characterisation} for our $p$-divergence $\mc{F}_{p,\si}$, which complements the discussion in Proposition \ref{prop:necessary}. 

\begin{proposition}
Let $\mc{L}$ be the generator of a primitive QMS with invariant state $\si \in \dhh$ and $p \in (1,2]$. Suppose that $\mc{L}$ satisfies the $[\si]_{p,0}$-DBC and $\mc{E}_{p,\mc{L}}(\gs^{-1} (\rho))$ is strictly positive for $\rho \in \dhh$ except at $\rho = \si$. Then there exists a Riemannian metric on $\mc{D}_+(\mc{H})$ such that the QMS $\mc{P}_t^\dag = e^{t \mc{L}^\dag}$ is the gradient flow of $\mc{F}_{p,\si}(\rho)$. 
\end{proposition}

\begin{proof}
This result is a simple consequence of \cite[Corollary 2.5]{brooks2022characterisation}, which shows that for a smooth function $f$ on the manifold $M$, a vector field $\vp^\alpha$, and a co-vector field $\phi_\beta = \na_\beta f$ with some mild assumptions, there exists a Riemannian metric $g_{\alpha\beta}$ such that $\phi_\beta = g_{\alpha\beta} \vp^\alpha$ if i) the field $\phi$ has a unique zero $\bar{m}\in M$, where $f$ takes its minimum value and there holds $\vp^\alpha|_{\bar{m}} = 0$; ii) $\na_{\vp^\alpha}f|_{m} < 0$ holds for all $m \neq \bar{m}$; iii) the map $\na_{\alpha}\vp^\beta: T_{\bar{m}}M \to T_{\bar{m}}M$ is positive and symmetric with respect to the inner product on $T_{\bar{m}}M$ induced by the Hessian $h_{\alpha\beta} = \na_\alpha \na_\beta f|_{\bar{m}}$. 

For our case, we consider the manifold $M = \mc{D}_+(\mc{H})$, the vector field $\vp = \mc{L}^\dag$, and the functional $f= \mc{F}_{p,\si}$. Note from Lemma \ref{lem:prop_divi} that $\mc{F}_{p,\si}(\rho)$ is strictly positive except at its global minimum $\rho = \si$. By definition \eqref{eq:deriv_func}, we have $\l \d_\rho \mc{F}_{p,\si}(\rho), A \r = 0$, $\forall A \in T_\rho = \mc{B}^0_{sa}(\mc{H})$, if and only if $\rho = \si$. Moreover, in our case, $\vp^\alpha|_{\bar{m}} = 0$ simply means $\mc{L}^\dag(\si) = 0$. Thus, the condition i) holds. Since $\na_{\vp^\alpha}f$ is nothing than the entropy production of $\mc{F}_{p,\si}$ (i.e., the minus Dirichlet form $\mc{E}_{p,\mc{L}}$) by \eqref{eq:ep_divi}, the condition ii) is guaranteed by the assumption. It suffices to verify the condition iii). For this, note that $\na_\alpha \vp^\beta = \mc{L}^\dag$, as $\vp$ is linear, and that the Hessian can be computed as follows: for $X, Y \in \mc{B}_{sa}(\mc{H})$ and $A = \gs(X)$ and $B = \gs(Y)$,
\begin{align*}
    h(A,B)  & = \p_{\ep}|_{\ep = 0} \p_{\eta}|_{\eta = 0} \mc{F}_{p,\si}(\si + \ep A + \eta B) \\ 
    & = \p_{\ep}|_{\ep = 0} \frac{1}{p-1} \l \gs^{-1/\h{p}}((\gs^{-1/\h{p}}(\si + \ep A))^{p-1}), B \r \notag \\
    & = \p_{\ep}|_{\ep = 0} \frac{1}{p-1} \l I_{\h{p},p}(\mi + \ep X) ,  \gs(Y)\r \notag  = \l A, [\si]_{p,0}^{-1} B \r\,,
\end{align*}
where the second identity follows from \eqref{eq:ep_divi} and the last identity is by \eqref{eq:kernel_iden_2} and \eqref{eq:expansion_power}. Recall that the Lindbladian $\mc{L}$ is assumed to satisfy the $[\si]_{p,0}$-DBC, which is equivalent to the fact that $\mc{L}^\dag$ is self-adjoint with respect to the inner product $\l \dd, [\si]_{p,0}^{-1} \dd \r$. The proof is complete. 
\end{proof}

\begin{remark} \label{rem:positivediri}
Compared to \cite[Theorem 1.2]{brooks2022characterisation}, here we have an additional assumption of the strict positivity of the $p$-Dirichlet form, so there is still a small gap between the necessary and sufficient conditions in our case. When $p = 2$, it is easy to show that such an assumption can be removed by the KMS DBC of the Lindbladian. Indeed, by \eqref{eq:ep_divi} and the detailed balance condition, we have that $\mc{E}_{2,\mc{L}}(\gs^{-1}(\rho))$ is convex  and attains its minimum at $\rho = \si$. Then, the invertibility of $\mc{L}^\dag $ on the tangent space readily implies the strict convexity of  $\mc{E}_{2,\mc{L}}(\gs^{-1}(\rho))$ at $\rho = \si$, which further gives its strict positivity for any $\rho \neq \si$. However, for a general $p$, it seems not a very easy task to remove this assumption. A potential barrier may be the lack of the characterization of the QMS satisfying the $[\si]_{p,0}$-DBC; see  \cite{alicki1976detailed,amorim2021complete,fagnola2010generators,vernooij2023derivations} for related results. 
\end{remark}

\subsection{Generalized quantum transport distances} \label{sec:quantum_distance} The Riemannian distance $W_{2,p}$ induced by the metric $g_{p,\rho}$ in \eqref{def:metric_tensor} can be defined as: for $\rho_0, \rho_1 \in \dhh$, 
\begin{align} 
    W_{2,p}(\rho_0,\rho_1)^2  =\ & \inf  \left\{
    \int_0^1 g_{p, \gamma(s)}\left(\dot{\gamma}(s),\dot{\gamma}(s) \right) ds\,;\ \gamma(0) = \rho_0\,, \gamma(1) = \rho_1
    \right\}  \label{def:wp1} \\
     \overset{\eqref{def:gp2}}{=} & 
       \inf  \left\{
    \int_0^1 \norm{\na U(s)}^2_{p,\gamma(s)} ds \,;\ \dot{\gamma}(s) = \mf{D}_{p,\rho}U(s)\ \text{with}\  \gamma(0) = \rho_0\,, \gamma(1) = \rho_1
    \right\}, \label{def:wp2}
\end{align}
where the infimum is taken over the smooth $(C^\infty)$ curves $\gamma(s):[0,1] \to \dhh$. In this section, we will investigate the basic properties of the distance function $W_{2,p}$.  By the standard reparameterization techniques (cf.\cite[Lemma 1.1.4]{ambrosio2005gradient} or \cite[Theorem 5.4]{dolbeault2009new}),
we have that $W_{2,p}$ equals to the minimum arc length:
\begin{align} \label{def:wp3}
    W_{2,p}(\rho_0,\rho_1) = \inf  \left\{
    \int_0^1 g_{p, \gamma(s)}\left(\dot{\gamma}(s),\dot{\gamma}(s) \right)^{1/2} ds\,;\ \gamma(0) = \rho_0\,, \gamma(1) = \rho_1
    \right\},
\end{align}
where the infimum is taken over smooth curves of constant speed (i.e., $g_{p, \gamma(s)}\left(\dot{\gamma}(s),\dot{\gamma}(s) \right)^{1/2}$ is constant). Then it follows from \eqref{def:wp3} that the Riemannian manifold $\mc{D}_+(\mc{H})$ equipped with the distance $W_{2,p}$ is a metric space. Moreover, it turns out that $W_{2,p}$ can be extended continuously to the boundary of $\mc{D}_+(\rho)$.

\begin{lemma} \label{lem:exten_bry}
The metric $W_{2,p}$ for $p \in (1,2]$ on $\mc{D}_+(\mc{H})$ extends continuously to $\dh$. 
\end{lemma}

\begin{proof}
As the proof is similar to that of \cite[Proposition 9.2]{carlen2020non}, we only sketch its main steps below. Consider $\rho_0,\rho_1 \in \dh$ and let $\{\rho_0^n\}, \{\rho_1^n\}$ be any sequences in $\dhh$ such that $\norm{\rho_i^n - \rho_i}_2 \to 0$ as $n \to \infty$ for $i = 0,1$. It suffices to show that $W_{2,p}(\rho^n_0,\rho^n_1)$ is a Cauchy sequence. For this, by the triangle inequality, we only need to show $W_{2,p}(\rho_i^n, \rho_i^m) \to 0$ as $n, m \to \infty$ for $i = 0,1$. For $\ep \in (0,1)$, 
we define $\w{\rho}_0 = (1- \ep) \rho_0 + \ep \mi$, and the linear interpolation $\gamma_n(s):= (1 - s) \rho_0^n + s \w{\rho}_0$ which satisfies $\gamma_n(s) \ge s \ep \mi$. Then it is easy to see 
\begin{align*} 
e^{\pm \ww_j/2p}\gs^{-1/\h{p}}(\gamma_n(s)) \ge \inf_j \{e^{-|\ww_j|/2p}\} s \ep \si^{-1/\h{p}}\,,
\end{align*}
by which, Lemma \ref{lem:mono_norm} implies that $[\gamma_n(s)]_{p,\ww}^{-1} \le C (s \ep)^{p-2} \id_{\mc{H}}$ holds for some constant $C > 0$, and hence that
\begin{align*}
    \mf{D}_{p,\gamma_n(s)}^{-1} \le C (s \ep)^{p-2} \norm{(-\ddiv \circ \na(\dd))^{-1}} \id_{\mc{H}}\,.
\end{align*}
Recalling the expression \eqref{def:wp3}, we obtain 
\begin{align} \label{auxeq:bry}
    W_{2,p}(\rho_0^n, \w{\rho}_0) &\le \int_0^1 \left\l \dot{\gamma}_n(s), \mf{D}_{p,\gamma_n(s)}^{-1} \dot{\gamma}_n(s)  \right\r^{1/2}\, ds \notag \\
    & \le C (s \ep)^{p-2} \norm{(-\ddiv \circ \na(\dd))^{-1}} \int_0^1 \norm{ \dot{\gamma}_n(s)}_2\, ds\,.
\end{align}
Note the estimate:
\begin{align*}
    \norm{ \dot{\gamma}_n(s)}_2 = \norm{\w{\rho}_0 - \rho^n_0}_2 \le \ep \norm{\rho_0 - \mi}_2 + \norm{\rho_0 - \rho_0^n}_2\,,
\end{align*}
which, by \eqref{auxeq:bry}, gives
\begin{align*}
   \limsup_{n \to \infty}  W_{2,p}(\rho_0^n, \w{\rho}_0) \le C (s \ep)^{p-1} \norm{(-\ddiv \circ \na(\dd))^{-1}} \norm{\rho_0 - \mi}_2\,.
\end{align*}
Since $p \in (1,2]$ and $\ep$ is arbitrary, it follows that $ \lim_{n,m \to \infty} W_{2,p}(\rho_0^n, \rho_0^m) = 0$, by triangle inequality.
\end{proof} 

In analogy with the classical 2-Wasserstein distance, it is helpful to introduce the quantum moment variable ${\bf B} = ([\rho]_{p,\ww_1}\p_1 U, \ldots, [\rho]_{p,\ww_J}\p_J U )$ for $U \in \mc{B}_{sa}^0$ and reformulate \eqref{def:wp2} as a convex optimization problem: 
\begin{align} \label{def:wp4}
    W_{2,p}(\rho_0,\rho_1)^2 = \inf \Big\{
    \int_0^1 \norm{{\bf B}}^2_{-1,p,\gamma(s)} ds \,;\ &\dot{\gamma}(s) + \ddiv {\bf B} (s) = 0\ \text{with}\  \gamma(0) = \rho_0\,, \gamma(1) = \rho_1,  \notag \\ 
    & \gamma(s) \in C\left([0,1]; \dh\right),\ \text{and}\  {\bf B} (s)\in L^1\left([0,1]; \bh^J\right) 
    \Big\}\,,
\end{align}
where the continuity equation $\dot{\gamma}(s) + \ddiv {\bf B}(s) = 0$ is understood in the sense of distributions, and
$\norm{\dd}_{-1,p,\rho}$ is the norm from 
the inner product $\l\dd,\dd\r_{-1,p,\rho}$  on $\mc{B}(\mc{H})^J$ defined by
\begin{align*}
    \l {\bf A}, {\bf B}\r_{-1,p,\rho} := \sum_{j = 1}^J \l A_j, [\rho]_{p,\ww_j}^{-1} B_j\r\,.
\end{align*}
Indeed, by approximation techniques in \cite[Proposition 1]{wirth2022dual} and a mollification argument from \cite[Lemma 2.9]{erbar2012ricci}, we can show that the infimum in the above representation \eqref{def:wp4} of $W_{2,p}$ can be equivalently taken over smooth curves $\gamma \in C^\infty([0,1];\dhh)$. Then, the equivalence between formulations \eqref{def:wp2} and \eqref{def:wp4} follows from the same arguments as in \cite[Lemma 9.1]{carlen2020non}, while the convexity of the optimization problem \eqref{def:wp4} is a simple consequence of the following lemma and \cite[Theorem 9.7]{carlen2020non}.

\begin{lemma}\label{lem:joint_convexity}
$\l X, [\rho]_{p,\ww}^{-1}X \r$ with $p \in (1,2]$ is jointly convex in $(\rho,X) \in \mc{D}_+(\mc{H}) \t \mc{B}(\mc{H})$. 
\end{lemma}
\begin{proof}
Note from \eqref{def:kernel_multiplication} that  $$\l X, [\rho]_{p,\ww}^{-1} X\r =  \l \gs^{-1/\h{p}}(X), f_p^{[1]}(e^{\ww/2p} \gs^{-1/\h{p}}(\rho),  e^{-\ww/2p}  \gs^{-1/\h{p}}(\rho))(\gs^{-1/\h{p}}(X))\r\,.$$ Then the statement follows from \cite[Lemma 2.3]{zhang2020optimal}, which shows the joint convexity of $\l Y, f_p^{[1]}(A,B) Y \r$ for $Y \in \bh$ and full-rank $A,B \in \mc{B}_{sa}^+(\mc{H})$. 
\end{proof}

We summarize the above discussion in the following proposition.

\begin{proposition} \label{prop:dist_convexity}
$W_{2,p}^2$ has an equivalent convex optimization formulation \eqref{def:wp4}. More precisely, let $\rho_0^i, \rho_1^i\in \mc{D}(\mc{H})$ for $i = 0,1$, and set $\rho_0^s := (1-s)\rho_0^0 + s\rho_0^1$ and $\rho_1^s := (1-s)\rho_1^0 + s\rho_1^1$ for $s \in [0,1]$. Then there holds 
\begin{equation*}
W_{2,p}(\rho^s_0,\rho^s_1)^2 \le (1-s) W_{2,p}(\rho^0_0,\rho^0_1)^2 + s W_{2,p}(\rho^1_0,\rho^1_1)^2\,.
\end{equation*}
\end{proposition}

The main result of this section is Theorem \ref{thm:main_wasser} below.

\begin{theorem}\label{thm:main_wasser}
$(\mc{D}(\mc{H}), W_{2,p})$ for $p \in (1,2]$ is a complete geodesic space. In particular, for any $\rho_0, \rho_1 \in \dh$, the minimizer to \eqref{def:wp4} exists and gives the minimizing geodesic  $(\gamma_*(s))_{s \in [0,1]}$ connecting $\rho_0$ and $\rho_1$ and satisfying 
\begin{equation} \label{eq:geo_mini}
 W_{2,p}(\gamma_*(s),\gamma_*(t)) = |s - t| W_{2,p}(\rho_0,\rho_1)\,, \q \forall s,t \in [0,1]\,.
\end{equation}
\end{theorem}
The completeness of the metric space $(\mc{D}(\mc{H}), W_{2,p})$
needs the following lemma, which is of interest itself.

\begin{lemma} \label{lem:lower_bound}
There exists a constant $C > 0$, independent of $p \in (1,2]$, such that for any $\rho_0, \rho_1 \in \dh$,
\begin{align*}
    \norm{\rho_1 - \rho_0}_1 \le C W_{2,p}(\rho_0,\rho_1)\,.
\end{align*}
\end{lemma}

We now give the proofs of Theorem \ref{thm:main_wasser} and Lemma \ref{lem:lower_bound}. 

\begin{proof}[Proof of Theorem \ref{thm:main_wasser}]
We will first show the existence of the minimizer of \eqref{def:wp4} by direct method. Let $\left\{(\gamma^{(n)},{\bf B}^{(n)})\right\}$ be the minimizing sequence such that 
$
\sup_n \int_0^1 \big\lVert{\bf B}^{(n)}\big\rVert^2_{-1,p,\gamma^{(n)}}\ ds < + \infty\,.
$
We claim that there exists constant $C_j$, depending on $\si \in \dhh$ and $\ww_j \in \R$, such that 
\begin{align} \label{claim_one}
\l X, [\rho]_{p,\ww_j} X\r \le C_j \norm{X}_2^2\,,\q \forall \rho \in \dh\,.
\end{align}
To show this, by the representation \eqref{eq:repre_rpw} of $[\rho]_{p,\ww}$ with notations from Proposition \ref{prop:repkernel} below, we first have 
\begin{align*}
    \l X,  [\rho]_{p,\ww} X \r \le (p-1)  \bigg( \sup_{i,k} \theta_p\left(e^{\ww/2p}\lad_{k,p}, e^{-\ww/2p}\lad_{i,p}\right)
    \bigg) \norm{\gs^{1/\h{p}}(X)}_2^2\,.
\end{align*}
We also note that there exists a closed interval $I$ containing $e^{\pm \ww/2p} \lad_{k,p}$ for all $\rho \in \dh$, and that the function $\theta_p(x,y)$ is bounded for $x,y \in I$. Then the claim \eqref{claim_one} readily follows. Since ${\bf B}^{(n)} \in L^1\left([0,1],\bh^J\right)$, we
can consider the $\bh^J$-valued measure $\mathsf{B}^{(n)}(ds) := {\bf B}^{(n)}(s) ds$. Then, for every Borel set $E \subset [0,1]$, we estimate 
\begin{align} \label{eq:estimate_minimizing}
    |\mathsf{B}^{(n)}|(E) & \le \int_E \bigg(\sum_{j = 1}^J \norm{B_j^{(n)}(s)}_2^2 \bigg)^{1/2} ds \notag \\
    & = \int_E \bigg(\sum_{j = 1}^J \left\l [\gamma^{(n)}]_{p,\ww_j}^{-1/2} B_j^{(n)}, [\gamma^{(n)}]_{p,\ww_j} [\gamma^{(n)}]_{p,\ww_j}^{-1/2} B_j^{(n)} \right\r \bigg)^{1/2} ds \notag  \\
    & \le C \mathscr{L}(E)^{1/2} \bigg( \int_E \sum_{j = 1}^J \left\l  B_j^{(n)}, [\gamma^{(n)}]^{-1}_{p,\ww_j}  B_j^{(n)} \right\r  ds \bigg)^{1/2}  \notag  \\
     & \le C \mathscr{L}(E)^{1/2} \bigg(\int_0^1 \big\lVert{\bf B}^{(n)}\big\rVert^2_{-1,p,\gamma^{(n)}} ds \bigg)^{1/2} < +\infty\,,
\end{align}
for $C := \max_j\{C_j^{1/2}\}$, where $\mathscr{L}$ is the Lebesgue measure on $\R$ and the third line is from H\"{o}lder's inequality and the estimate \eqref{claim_one}. 
It follows that the total variations of the measures $\mathsf{B}^{(n)}$ are uniformly bounded in $n$. Hence there exists a subsequence, still denoted by $\mathsf{B}^{(n)}$, converging weakly* to a $\bh^J$-valued measure $\mathsf{B}_*$. Then, by \eqref{eq:estimate_minimizing}, we can obtain
\begin{equation*}
    |\mathsf{B}_*|(E) \le \liminf_{n \to \infty}  |\mathsf{B}^{(n)}|(E) \le C \mathscr{L}(E)^{1/2}\,,
\end{equation*}
for some constant $C$, which implies that $\mathsf{B}_* \ll \mathscr{L}$ and hence the R-N derivative ${\bf B}_*: = \frac{d \mathsf{B}_*}{d \mathscr{L}}$ exists. We next prove that $\gamma^{(n)}(s)$ converges pointwise to some $\gamma(s): [0,1] \to \dh$. For this, we note 
\begin{equation*}
    \gamma^{(n)}(t) - \gamma^{(n)}(0) = - \int_0^t \ddiv {\bf B}^{(n)}(s)\ ds. 
\end{equation*}
Then by the weak* convergence of $\mathsf{B}^{(n)}$ and $\gamma^{(n)}(0) = \rho_0$, we have the pointwise convergence of $\gamma^{(n)}(s)$ with the limit denoted by $\gamma_*(s)$.
Moreover, it is easy to check that $\gamma_*(s) \in C([0,1],\dh) $ and ${\bf B}_* \in L^1([0,1],\bh^J)$ satisfy the continuity equation. By dominated convergence theorem, we also have $\mc{B}(\bh)$-valued measure $[\gamma^{(n)}(s)]_{p,\ww_j} ds$ weakly* converge to $[\gamma_*(s)]_{p,\ww_j} ds$. Finally, noting the integral representation for $[\rho]^{-1}_{p,\ww}$ from \eqref{eq:inte_divi_p} and \eqref{def:kernel_multiplication}: 
\begin{align*}
      [\rho]^{-1}_{p,\ww}(\dd)
    & =
    \frac{\sin((p-1) \pi)}{\pi} \gs^{-1/\h{p}} \int_0^\infty s^{p - 2}   g_s^{[1]} \left(e^{\ww/2p}  \gs^{-1/\h{p}}(\rho),  e^{-\ww/2p}  \gs^{-1/\h{p}}(\rho) \right)  \gs^{-1/\h{p}}(\dd)\ ds\,,
\end{align*}
 we have 
\small
\begin{align} \label{auxeq_final}
  & \liminf_{n \to \infty} \int_0^1 \big\lVert{\bf B}^{(n)}\big\rVert^2_{-1,p,\gamma^{(n)}}\ ds  
  \notag \\
  \ge & \frac{\sin((p-1)\pi)}{\pi} \sum_{j=1}^J \int_0^\infty t^{p-2} \liminf_{n \to \infty} \int_0^1  \left\l \gs^{-1/\h{p}}\big(B_j^{(n)}(s)\big) , g_t^{[1]} \left(e^{\ww/2p}  \gs^{-1/\h{p}}\big(\gamma^{(n)}(s)\big),  e^{-\ww/2p}  \gs^{-1/\h{p}}\big(\gamma^{(n)}(s)\big) \right)  \gs^{-1/\h{p}}\big(B_j^{(n)}(s)\big) \right\r ds dt    \notag \\
   \ge & \frac{\sin((p-1)\pi)}{\pi} \sum_{j=1}^J \int_0^\infty t^{p-2} \int_0^1  \left\l \gs^{-1/\h{p}}\big(B_{*,j}(s)\big) , g_t^{[1]} \left(e^{\ww/2p}  \gs^{-1/\h{p}}\big(\gamma_*(s)\big),  e^{-\ww/2p}  \gs^{-1/\h{p}}\big(\gamma_*(s)\big) \right)  \gs^{-1/\h{p}}\big(B_{*,j}(s)\big) \right\r ds dt   \notag \\
  \ge  & \int_0^1 \big\lVert{\bf B}_*\big\rVert^2_{-1,p,\gamma_*} ds\,,
\end{align}
\normalsize
where in the first inequality we have used Fatou's lemma, and in the second inequality we have used Theorem 3.4.3 of \cite{buttazzo1989semicontinuity} on the lower-semicontinuity of integral functionals. Then the estimate \eqref{auxeq_final} directly implies that $(\gamma_*, {\bf B}_*)$ is a minimizer to \eqref{def:wp4}. 

Recalling formulations \eqref{def:wp1} and \eqref{def:wp3}, by Jensen's inequality, we find 
\begin{align*}
W_{2,p}(\rho_0,\rho_1) =  \Big(\int_0^1 g_{p, \gamma_*(s)}\left(\dot{\gamma}_*(s),\dot{\gamma}_*(s) \right) ds \Big)^{1/2} = \int_0^1 g_{p, \gamma_*(s)}\left(\dot{\gamma}_*(s),\dot{\gamma}_*(s) \right)^{1/2} ds\,.
\end{align*}
It follows that $W_{2,p}(\rho_0,\rho_1)^2 = g_{p, \gamma_*(s)}\left(\dot{\gamma}_*(s),\dot{\gamma}_*(s) \right)$ for $s \in [0,1]$ a.e.. Then, by definition of $W_{2,p}$ and a time scaling, it is easy to check that the property \eqref{eq:geo_mini} holds. Therefore, we have proved that $(\mc{D}(\mc{H}), W_{2,p})$ is a geodesic space. The completeness of the metric space $(\mc{D}(\mc{H}), W_{2,p})$ is a simple consequence of Lemma \ref{lem:lower_bound}. Indeed, let $\{\rho_n\}$ be a Cauchy sequence such that $W_{2,p}(\rho_n,\rho_m) \to 0$ as $n,m \to \infty$. Then, by Lemma \ref{lem:lower_bound}, it is also Cauchy in the complete metric space $(\dh,\norm{\dd}_1)$, which implies that there exists $\rho_\infty \in \dh$ such that $\norm{\rho_n - \rho_\infty}_1 \to 0$ as $n \to \infty$. Similarly to the proof of Lemma \ref{lem:exten_bry}, by $\norm{\rho_n - \rho_\infty}_1 \to 0$, we can conclude $W_{2,p}(\rho_n,\rho_\infty) \to 0$ as $n \to \infty$. 
\end{proof}

\begin{proof}[Proof of Lemma \ref{lem:lower_bound}]
For any $\d > 0$ and $\rho_0, \rho_1 \in \dhh$, by Theorem \ref{thm:main_wasser}, there exists a curve $(\gamma(s),{\bf B}(s))$, $s \in [0,1]$, satisfying $\dot{\gamma}(s) + \ddiv {\bf B}(s) = 0$ with  $\gamma(0) = \rho_0$ and $\gamma(1) = \rho_1$, such that
\begin{align*}
    \int_0^1 \norm{{\bf B}(s)}^2_{-1, p, \gamma(s)} ds \le W_{2,p}(\rho_0,\rho_1)^2 + \d\,.
\end{align*}
It then follows that, by Cauchy's inequality,
\begin{align} \label{auxeq_bound1}
    \tr(X(\rho_1 - \rho_0))  = \tr \left(X \int_0^1 \dot{\gamma}(s) ds \right) \notag
     & = \int_0^1  \l \na X, {\bf B}(s) \r ds \notag \\
    &\le \left(\int_0^1 \norm{\na X}_{p,\gamma(s)}^2 d s \right)^{1/2} \left( \int_0^1 \norm{{\bf B}(s) }_{-1,p,\gamma(s)}^2 d s \right)^{1/2} \notag \\
    &\le \bigg(\int_0^1  \sum_j \l \p_j X, [\gamma(s)]_{p,\ww_j} \p_j X  \r  d s 
 \bigg)^{1/2} \left( W_{2,p}(\rho_0,\rho_1)^2 + \d \right)^{1/2}. 
\end{align}

To deal with the term $\l \p_j X, [\gamma(s)]_{p,\ww_j} \p_j X  \r$, we next estimate the kernel operator $[\rho]_{p,\ww}$ for $\rho \in \mc{D}_+$ and $\ww \in \R$. For this, recalling \eqref{eq:repre_rpw} below:
\begin{align} \label{eq:kernel2}
      [\rho]_{p,\ww}(\dd) = (p-1)\sum_{i,k}\frac{e^{\ww/2p}\lad_{k} - e^{-\ww/2p}\lad_{i}}{\left(e^{\ww/2p}\lad_{k}\right)^{p-1} - \left(e^{-\ww/2p}\lad_{i}\right)^{p-1}} \gs^{1/\h{p}} \left(E_{k} \gs^{1/\h{p}}(\dd) E_{i}\right),
\end{align}
where $\lad_{i}$ and $E_{i}$ are eigenvalues and the associated eigen-projections of $\gs^{-1/\h{p}}(\rho)$, respectively (we omit the subscript $p$ of $\lad_i$ and $E_i$ for simplicity). 
By the integral formula \eqref{eq:inte_divi_p}, we can estimate 
\begin{align} \label{eq:ker_estimate}
  &\frac{1}{p-1}  \frac{\left(e^{\ww/2p}\lad_{k}\right)^{p-1} - \left( e^{-\ww/2p}\lad_{i}\right)^{p-1}}{e^{\ww/2p}\lad_{k} - e^{-\ww/2p}\lad_{i}}
\notag \\  
= &  \frac{\sin((p-1) \pi)}{\pi} \int_0^\infty  s^{p-2} 
\frac{\log (s + e^{\ww/2p} \lad_k) - \log (s + e^{-\ww/2p}\lad_i)  }{e^{\ww/2p} \lad_k - e^{-\ww/2p}\lad_i}
\  ds \notag\\
\ge &  \frac{\sin((p-1) \pi)}{\pi} \int_0^\infty  s^{p-2} \frac{2}{2s + e^{\ww/2p} \lad_k + e^{-\ww/2p}\lad_i}\  ds \notag \\
= & \frac{\sin((p-1) \pi)}{\pi} \left(e^{\ww/2p}\lad_k + e^{-\ww/2p} \lad_i\right)^{p-2} \int_0^\infty s^{p-2}  \frac{2}{1+ 2 s}\  ds\,, 
\end{align}
where the third line is from the elementary inequality: 
\begin{equation*}
    \frac{x - y}{\log x - \log y} \le \frac{x + y}{2}\q \text{for all}\  x, y > 0\,,
\end{equation*}
and the last line is by change of variable.  We define the constant:
\begin{equation}\label{def:cons_cp}
     C_p := \frac{\sin((p - 1) \pi)}{\pi}  \int_0^\infty s^{p-2}  \frac{2}{1+ 2 s} \ d t\,. 
\end{equation}
By \eqref{eq:kernel2} and \eqref{eq:ker_estimate}, there holds 
\begin{align*}
\left\l \p_j X,    [\rho]_{p,\ww}\p_j X \right\r \le & C_p^{-1} \sum_{i,k} \left(e^{\ww/2p}\lad_k + e^{-\ww/2p} \lad_i\right)^{2-p}  \left\l \gs^{1/\h{p}} (\p_j X),  \left(E_{k} \gs^{1/\h{p}}(\p_j X) E_{i}\right) \right\r \\
    \le &  C_p^{-1} \sum_{i,k} \left(\left(e^{\ww/2p}\lad_k \right)^{2-p} + \left(e^{-\ww/2p} \lad_i\right)^{2-p}\right)   \left\l \gs^{1/\h{p}} (\p_j X),  \left(E_{k} \gs^{1/\h{p}}(\p_j X) E_{i}\right) \right\r \\ 
     = &  C_p^{-1}  \left\l \gs^{1/\h{p}} (\p_j X),  \left(L_{\left(e^{\ww/2p}\gs^{-1/\h{p}}(\rho)\right)^{2-p}} + R_{\left(e^{-\ww/2p}  \gs^{-1/\h{p}}(\rho) \right)^{2-p}}\right)  \gs^{1/\h{p}}(\p_j X)  \right\r \\
      \le &  C_p^{-1}  \left\l \gs^{1/\h{p}} (\p_j X),  \left(L_{\left(e^{\ww/2p}\si^{-1/\h{p}}\right)^{2-p}} + R_{\left(e^{-\ww/2p}  \si^{-1/\h{p}} \right)^{2-p}}\right)  \gs^{1/\h{p}}(\p_j X)  \right\r  \\
     \le & C_p^{-1} \left( e^{(2-p)\ww/2p} + e^{(p-2)\ww/2p}  \right) \tr \left( \si^{(p-2)/\h{p}} \right)  \norm{\gs^{1/\h{p}}(\p_j X)}_{\infty}^2 \\
     \le & 4 C_p^{-1} \left(e^{(2-p)\ww/2p} + e^{(p-2)\ww/2p}  \right) \tr \left( \si^{(p-2)/\h{p}} \right)  \norm{\si}_\infty^{2/\h{p}}\norm{V_j}^2_\infty \norm{X}^2_\infty\,,
\end{align*}
where the second inequality is by $(x + y)^p \le x^p + y^p$ for $p \in (0,1)$ and $x, y > 0$; the third inequality is by $\rho \le \mi$ for all $\rho \in \dh$ and the operator monotonicity of $t^p$ for $0 \le p \le 1$; the fourth inequality is by
H\"{o}lder's inequality.  Then we arrive at, by \eqref{auxeq_bound1}, 
\begin{align} \label{eq:est_lower}
     \norm{\rho_1 - \rho_0}_1 & \le \bigg( 4 C_p^{-1} \tr \left( \si^{(p-2)/\h{p}} \right) \norm{\si}_\infty^{2/\h{p}} \sum_j \left(e^{(2-p)\ww_j/2p} + e^{(p-2)\ww_j/2p}  \right) \norm{V_j}_\infty^2 \bigg)^{1/2}  W_{2,p}(\rho_0,\rho_1)\,. 
\end{align}
 
We finally prove the uniform boundedness of the prefactor in \eqref{eq:est_lower} for $p \in (1,2]$. It suffices to consider the constant $C_p$ in \eqref{def:cons_cp}. By elementary calculation, we derive 
\begin{align*}
\frac{2}{3}\left( \frac{1}{p - 1} + \frac{1}{2 - p} \right) \le \int_0^\infty \frac{2 s^{p-2} }{1+ 2 s} d s  &= \int_0^1  \frac{2 s^{p-2}}{1+ 2 s} d s + \int_1^\infty   \frac{2 s^{p-2}}{1+ 2 s} d s \le \frac{2}{p - 1} + \frac{1}{2 - p}\,,
\end{align*}
which immediately gives the following estimates: 
\begin{align*}
\frac{2}{3} + O(p - 1) \le  C_p \le 2 + O(p-1) \q \text{as} \ p \to 1^+\,,
\end{align*}
and 
\begin{align*}
\frac{2}{3} + O(p - 1) \le  C_p \le 1 + O(p-2) \q \text{as} \ p \to 2\,.
\end{align*}
The proof is complete. 
\end{proof}

We next derive the geodesic equations for the Riemannian manifold $(\dhh, g_{p,\rho})$. Instead of regarding the geodesic equation as the Euler-Lagrange equation for the minimization problem \eqref{def:wp2} as in \cite{datta2020relating,carlen2020non}, we interpret it as the Hamiltonian flow of the Hamiltonian associated with the metric $g_{p,\rho}$: 
$$
H(\rho,U) := \frac{1}{2} g^{-1}_{p,\rho}(U,U) = \frac{1}{2}\l \mf{D}_{p,\rho}U, U\r \q \text{for}\  \rho \in \dhh\ \text{and}\  U \in \mc{B}_{sa}^0(\mc{H})\,,
$$
where $g^{-1}_{p,\rho}$ is the inverse of the metric tensor. Then we can write the geodesic equations:
\begin{align} \label{def:geo_eq}
&  \dot{\rho} = \d_{U} H(\rho,U)\,,\q   \dot{U} = -\d_\rho H(\rho,U)\,.
\end{align}
By definition, it is clear that $\d_{U} H(\rho,U) = \mf{D}_{p,\rho}U$. To find $\d_\rho H$, by Lemma \ref{lem:high_order}, we have 
\begin{align} \label{eq:func_dev}
   \l \d_\rho H(\rho, U), A\r  =  & \mb{\lim_{\ep \to 0} \frac{1}{2} \big(\left\l \mf{D}_{p, \rho + \ep A} U, U\right\r - \left\l \mf{D}_{p, \rho} U, U\right\r \big)/\ep} \notag \\
    =  & \mb{\lim_{\ep \to 0} \frac{1}{2} \big(\left\l \p_j U, [\rho+\ep A]_{p,\ww_j} \p_j U \right\r - \left\l \p_j U, [\rho]_{p,\ww_j} \p_j U \right\r \big)/\ep} \notag \\ 
     =  & \frac{1}{2} \sum_{j = 1}^J\left\l \gs^{1/\h{p}} (\p_j U), (\d_1 \theta_p)\big((l_j(\rho),l_j(\rho)), r_j(\rho)\big)\big(l_j(A), \gs^{1/\h{p}}(\p_j U)\big)\right\r  \notag 
    \\
        & +  \frac{1}{2} \sum_{j=1}^J \left\l \gs^{1/\h{p}} (\p_j U), (\d_2 \theta_p)\big(l_j(\rho), ( r_j(\rho), r_j(\rho))\big)\big(\gs^{1/\h{p}}(\p_j U), r_j(A)\big)\right\r\,, 
\end{align}
where for any $X \in \bh$, 
\begin{align*}
 l_j(X): = e^{\ww_j/2p} \gs^{-1/\h{p}}(X)\,,\q  r_j(X): = e^{-\ww_j/2p} \gs^{-1/\h{p}}(X)\,.  
\end{align*}
We also note 
\begin{align} \label{eq:observation}
   & \left\l \gs^{1/\h{p}} (\p_{j'} U), (\d_1 \theta_p)\big((l_{j'}(\rho),l_{j'}(\rho)), r_{j'}(\rho)\big)\big(l_{j'}(A), \gs^{1/\h{p}}(\p_{j'} U)\big)\right\r  \notag \\
= &  \left\l \gs^{1/\h{p}} ((\p_{j} U)^*), (\d_1 \theta_p)\big((r_j(\rho),r_j(\rho)), l_j(\rho)\big)\big(r_j(A), \gs^{1/\h{p}}((\p_{j} U)^*)\big)\right\r \notag \\
= &  \left\l \gs^{1/\h{p}} (\p_{j} U), (\d_2 \theta_p)\big(l_j(\rho), (r_j(\rho),r_j(\rho))\big)\big(\gs^{1/\h{p}}(\p_{j} U), r_j(A)\big)\right\r \,,
\end{align}
where the first equality is from
the relations $l_j(\rho) = r_{j'}(\rho)$, $r_j(\rho) = l_{j'}(\rho)$, and $(\p_j U)^* = - \p_{j'}U$ by 
\eqref{eq:adjoint_index}, and the second equality is by the following formula from definitions \eqref{def_1:divi_diff} and \eqref{def_2:multiple}: 
\begin{equation*}
    \l X^*, (\d_1 f)((A,A),B)(Y,X^*) \r =  \l X, (\d_2 f)(B,(A,A))(X,Y)\r \,,
\end{equation*}
for any symmetric $f$: $f(s,t) = f(t,s)$, $X \in \bh$, $Y \in \mc{B}_{sa}(\mc{H})$, and commuting matrices $A,B \in \mc{B}_{sa}(\mc{H})$. The identity \eqref{eq:observation} implies that the two sums in \eqref{eq:func_dev} are actually equal. Therefore, we obtain from \eqref{def:geo_eq} and \eqref{eq:func_dev} the following proposition, where the local existence and uniqueness of geodesics follow from the standard theory of ODE. 

\begin{proposition}
On the Riemannian manifold $(\dhh, g_{p,\rho})$, the unique constant speed geodesic $(\rho(t))_{t \in (-\ep,\ep)}$, $\ep > 0$, with initial data: $\rho(0) = \rho_0 \in \dhh$ and $\dot{\rho}(0) = \nu_0 \in \mc{B}_{sa}^0(\mc{H})$, satisfies the equation:
\begin{equation} \label{eq:geod_equa}
\begin{split}
    & \dot{\rho} = \mf{D}_{p,\rho}U \,,\\
    &   \l \dot{U}, A\r  = - \sum_{j = 1}^J\left\l \p_j U,  \mc{K}_{\rho,A}^{(i),j} \left[ \p_j U \right] \right \r\,, \q \forall A \in \mc{B}^0_{sa}(\mc{H})\,,
\end{split}
\end{equation}
with $\rho(0) = \rho_0$ and $U(0) = \mf{D}_{p,\rho_0}^{-1} \nu_0$,
where $i = 1$ or $2$, and $\mc{K}_{\rho,A}^{(i),j}$ is defined by, for $\rho \in \dh$ and $A \in \mc{B}_{sa}^0(\mc{H})$, 
\begin{equation} \label{def:kgeode}
\begin{split}
 &\mc{K}_{\rho,A}^{(1),j}[\dd] = \gs^{1/\h{p}} \circ
 (\d_1 \theta_p)\big((l_j(\rho),l_j(\rho)), r_j(\rho)\big)\big[l_j(A), \gs^{1/\h{p}}(\dd) \big]: \bh \to \bh\,,
 \\ & \mc{K}_{\rho,A}^{(2),j}[\dd] = \gs^{1/\h{p}} \circ (\d_2 \theta_p)\big(l_j(\rho), ( r_j(\rho), r_j(\rho))\big)\big[\gs^{1/\h{p}}(\dd), r_j(A)\big]: \bh \to \bh\,.
\end{split}
\end{equation}
\end{proposition}

We end this section with the observation that the Riemannian metrics $g_{p,\rho}$ with $1 < p\le 2$ serve as an interpolating family between the metric defined by Carlen and Mass in \cite{carlen2017gradient} and the one induced from the KMS inner product $\l \dd, \dd\r_{\si,1/2}$. It can be easily proved by an elementary analysis with
the fact, from the analytic perturbation theory \cite{rellich1969perturbation}, that the eigenvalues and eigenfunctions of $\gs^{\sss -1/\h{p}}(\rho)$ are differentiable with respect to $p$.

\begin{proposition}\label{prop:repkernel}
Suppose that $\gs^{-1/\h{p}}(\rho)$ has the eigen-decomposition:
$
    \gs^{-1/\h{p}}(\rho) = \sum_{j = 1}^{d}  \lad_{j,p} E_{j,p}. 
$ Then, we have 
\begin{align} \label{eq:repre_rpw}
    [\rho]_{p,\ww}A = \sum_{i,k = 1}^{d}
    \theta_p\left(e^{\ww/2p}\lad_{k,p}, e^{-\ww/2p}\lad_{i,p} \right)
    \gs^{1/\h{p}} \left(E_{k,p} \gs^{1/\h{p}}(A) E_{i,p}\right)\,, 
\end{align}
for $A \in \bh$ and $p \in (1,2]$ with $\theta_p$ given in \eqref{def:theta_p}. Moreover, $[\rho]_{p,\ww}$ is continuous in $p$, and it holds that 
\begin{align*}
    [\rho]_{2,\ww}A = \gs(A)\,,
\end{align*}
and when $p \to 1^+$, 
\begin{align} \label{def:metricpeq1}
     [\rho]_{p,\ww}A \to [\rho]_{\ww}A :=  \sum_{i,k = 1}^{d}
     \theta_{log}\left(e^{\ww/2}\lad_k, e^{-\ww/2}\lad_i \right)
     E_k A E_i\,,
\end{align}
where $\theta_{log}$ is the logarithmic mean:
\begin{align*}
    \theta_{log}(x,y): = 
    \begin{cases}
        \frac{x - y}{\log x - \log y} \,, & x \neq y\,, \\
        x \,,\q  & x = y\,,
    \end{cases}   
\end{align*}
and $\lad_j$ are eigenvalues of $\rho$ with $E_j$ being the associated rank-one eigen-projections: $\rho = \sum_{j = 1}^{d} \lad_j E_j$. 
\end{proposition}

Recalling the integral formula: $\frac{x - y}{\log x - \log y} = \int_0^1 x^{1-s} y^s \, ds $, we can write the operator $[\rho]_\ww$ in \eqref{def:metricpeq1} as
\begin{align*}
     [\rho]_\ww(A) = \int_0^1 e^{\ww(1/2-s)} \rho^{1-s} A \rho^{s}\, ds\,,\q A \in \bh\,, 
\end{align*}
which is nothing but the noncommutative multiplication by $\rho$ involved in the definition of the quantum 2-Wasserstein distance $W_2$ \cite{carlen2014analog,carlen2017gradient}. It follows that in the limiting case $p \to 1^+$, our transport distance $W_{2,p}$ reduces to the one $W_2$ introduced by Carlen and Maas. When $p = 2$, the representation \eqref{def:wp4} of $W_{2,p}$ gives
\begin{align} \label{eq:distance_1}
    W_{2,2}(\rho_0,\rho_1)^2 = \inf \Big\{
    \int_0^1 \sum_{j = 1}^J \l B_j, \gs^{-1} B_j  \r \, ds \,;\ &\dot{\gamma}(s) + \ddiv {\bf B} (s) = 0\ \text{with}\  \gamma(0) = \rho_0\,, \gamma(1) = \rho_1 
    \Big\}\,,
\end{align}
which has a similar form as the classical distance $\mathscr{W}_{2,\alpha,\gamma}$ in \eqref{def:distance_savare} with $\alpha = 0$, and thus can be regarded as a quantum analog of dual $\dot{H}^{-1}$ Sobolev distance. These facts allow us to (at least formally) conclude that our new family of quantum distances $W_{2,p}$ interpolates the noncommutative $2$-Wasserstein and the $\dot{H}^{-1}$ Sobolev ones. One may further expect a stronger result that for $\rho_0,\rho_1 \in \dh$, $W_{2,p}(\rho_0,\rho_1)$ is continuous in $p \in (1,2]$ and $W_{2,p}(\rho_0,\rho_1) \to W_2(\rho_0,\rho_1)$ as $p \to 1^+$, but this task seems to be challenging, and we leave it for future investigation.


\subsection{Ricci curvature and functional inequalities} \label{sec:geodesic_convexity} 
In this section, we will investigate the entropic Ricci curvature lower bound in terms of $p$-divergence $\mc{F}_{p,\si}(\rho)$ with $p \in (1,2]$. Thanks to the gradient flow structure obtained in Section \ref{sec:gradient}, we are allowed to derive some new functional inequalities from the Ricci curvature lower bound: a quantum HWI-type inequality and a transportation cost inequality, and connect them with the quantum Beckner's inequality \eqref{ineq_becp}, in the spirit of Otto and Villani \cite{otto2000generalization}.

Let us first introduce the Ricci curvature in our setting. We follow the terminology of \cite{erbar2012ricci,datta2020relating}, and say that a primitive QMS $\mc{P}_t$ with $\si$-{\rm GNS DBC} has the Ricci curvature lower bound $\kappa \in \R$ associated with $\mc{F}_{p,\si}$ if
 \begin{equation} \label{eq:ricci_p}
    \frac{d^2}{ds^2}\Big|_{s = 0} \mc{F}_{p,\si}(\gamma(s)) \ge \kappa g_{p,\gamma(0)}(\dot{\gamma}(0),\dot{\gamma}(0))\,,
\end{equation}
where $\gamma(s)$, $s \in (-\ep,\ep)$, is a geodesic satisfying $\gamma(0) = \rho \in \dhh$. We compute the second derivative of $\mc{F}_{p,\si}$ along the constant-speed geodesic $\gamma(s)$. For this, noting the Riemannian gradient \eqref{eq:riemannian_grad} of $\mc{F}_{p,\si}$, by definition, we have 
\begin{align} \label{eq:first_dev_f}
     \frac{d}{ds} \mc{F}_{p,\si}(\gamma(s)) &= - g_{p,\gamma(s)} (\dot{\gamma}(s), \mc{L}^\dag(\gamma(s)))  = - \l U(s), \mc{L}^\dag(\gamma(s))\r\,,
\end{align}
where $(\gamma,U)$ is the unique solution to \eqref{eq:geod_equa}. Then, differentiating \eqref{eq:first_dev_f} again with respect to $s$, we obtain  
\begin{align} \label{eq:second_dev_f}
     \frac{d^2}{ds^2}\Big|_{s = 0} \mc{F}_{p,\si}(\gamma(s)) & = - \l \dot{U}(s), \mc{L}^\dag(\gamma(s))\r - \l U(s), \mc{L}^\dag(\dot{\gamma}(s))\r \Big|_{s = 0} \notag \\
     & = \sum_{j = 1}^J \big\l \p_j U(0),  \mc{K}_{\rho,\mc{L}^\dag(\gamma(0))}^{(i),j} \left[ \p_j U(0) \right] \big \r - \left\l U(0), \mc{L}^\dag(\mf{D}_{p,\gamma(0)} U(0)) \right \r\,.
\end{align}
Recall that the Riemannian Hessian of $\mc{F}_{p,\si}$ at $\rho \in \dhh$ is defined by, for $U \in \mc{B}_{sa}^0(\mc{H})$, 
\begin{align*}
\hess \mc{F}_{p,\si}(\rho)[U,U] := \frac{d^2}{ds^2}\Big|_{s = 0} \mc{F}_{p,\si}(\gamma(s))\,,
\end{align*}
where $(\gamma(s),U(s))$ satisfies the geodesic equation \eqref{eq:geod_equa} with initial conditions $\gamma(0) = \rho$, $U(0) = U$, and $\dot{\gamma}(0) = \mf{D}_{p,\rho}U$. We readily conclude from the formula \eqref{eq:second_dev_f} that 
\begin{align} \label{eq:hess_of_f}
    \hess \mc{F}_{p,\si}(\rho)[U,U] = \sum_{j = 1}^J \big\l \p_j U,  \mc{K}_{\rho,\mc{L}^\dag(\rho)}^{(i),j} \left[ \p_j U \right] \big \r - \left\l U, \mc{L}^\dag(\mf{D}_{p,\rho} U) \right \r\,, \q \forall U \in \mc{B}_{sa}^0(\mc{H})\,.
\end{align}
Hence, it follows from definition that \eqref{eq:ricci_p} is equivalent to 
\begin{align} \label{eq:ricci_hess}
    \hess\mc{F}_{p,\si}(\rho)[U,U] \ge \kappa \l U, \mf{D}_{p,\rho} U\r\,,  \tag{${\rm Ric}_p(\mc{L}) \ge \kappa$}
\end{align}
for $\rho \in \dhh$ and $U \in \mc{B}^0_{sa}(\mc{H})$. The next proposition provides several equivalent characterizations of \eqref{eq:ricci_hess}, in terms of the $\kappa$-geodesic convexity of $\mc{F}_{p,\si}(\rho)$, the gradient estimates, and the evolution variational inequality. In what follows, we will use the notation:
\begin{align} \label{def:upper_deriv}
    \frac{d^+}{d t} f(t) = \limsup_{h \to 0^+} \frac{f(t+h) - f(t)}{h}\,.
\end{align}



\begin{proposition} \label{prop:geodesic_convex}
Let $\mc{P}_t$ be a primitive QMS satisfying $\si$-{\rm GNS DBC} with $\si \in \dhh$. For $\kappa \in \R$, \eqref{eq:ricci_hess} is equivalent to the following statements:
\begin{enumerate}[(i)]
    \item $\mc{F}_{p,\si}(\rho)$ is geodesically $\kappa$-convex on $(\mc{D}(\mc{H}), W_{2,p})$, that is, for any constant-speed geodesic $(\gamma(s))_{s \in [0,1]} \subset \dh$, 
    \begin{align} \label{eq:geo_conv}
        \mc{F}_{p,\si}(\gamma(s)) \le (1-s) \mc{F}_{p,\si}(\gamma(0)) + s \mc{F}_{p,\si}(\gamma(1)) - \frac{\kappa}{2}s(1-s) W_{2,p}(\gamma(0),\gamma(1))^2\,.
    \end{align}
    \item For any $\rho_0, \rho_1 \in \dh$, the following evolution
     variational inequality (EVI) holds: $\forall t \ge 0$,
     \begin{align} \label{eq:evi}
         \frac{1}{2}\frac{d^+}{d t} W_{2,p}(\mc{P}_t^\dag \rho_0 ,\rho_1)^2 + \frac{\kappa}{2} W_{2,p}(\mc{P}_t^\dag \rho_0 ,\rho_1)^2  \le \mc{F}_{p,\si}(\rho_1) - \mc{F}_{p,\si}(\mc{P}_t^\dag \rho_0)\,.
    \end{align}
\item The following gradient estimate holds: for $\rho \in \dhh$ and $U \in \mc{B}_{sa}^0(\mc{H})$, $\forall t > 0$,
\begin{align*}
    \norm{\na \mc{P}_t(U)}^2_{p,\rho} \le e^{-2 \kappa t} \norm{\na U}^2_{p, \mc{P}_t^\dag\rho}\,.
\end{align*}
\item The following contraction of the transport distance $W_{2,p}$ along the gradient flow holds:
\begin{equation*}
  W_{2,p}(\mc{P}_t^\dag \rho_0, \mc{P}_t^\dag \rho_1) \le e^{-\kappa t} W_{2,p}(\rho_0,\rho_1) \q \forall \rho_0,\rho_1 \in \dhh \,.
\end{equation*}
\end{enumerate}
\end{proposition}

\begin{proof}
The equivalence: \eqref{eq:ricci_hess} $\eqq (i) \eqq (ii)$ can be proved in the same manner as \cite[Theorem 3]{datta2020relating} with the gradient flow techniques from \cite{otto2005eulerian,daneri2008eulerian,dolbeault2012poincare}; see also \cite[Theorem 10.2]{carlen2020non}. \eqref{eq:ricci_hess} $\eqq (iii)$ follows from a similar semigroup interpolation argument as in the proof of \cite[Theorem 10.4]{carlen2020non}, while the proof of \eqref{eq:ricci_hess} $\eqq (iv)$ can be easily modified from those of 
Proposition 3.1 and (2.12) of \cite{daneri2008eulerian}. 
\end{proof}

With the notion of Ricci curvature, we next prove some interesting implications between functional inequalities, following the arguments of Otto and Villani \cite{otto2000generalization} (see also \cite{erbar2012ricci,datta2020relating}). 
 We start with an HWI-type inequality, which relates the generalized quantum transport distance $W_{2,p}$, the quantum $p$-divergence $\mc{F}_{p,\si}$, and the entropy production ($p$-Dirichlet form) $\mc{E}_{p,\mc{L}}$. The following lemma will be used later on.

\begin{lemma} \label{lem:deriva_wasser}
Let $\rho, \w{\rho} \in \dhh$ and $\rho_t = \mc{P}_t^\dag \rho$. Then there holds, for $t \ge 0$,
\begin{align*}
    \frac{d^+}{d t} W_{2,p}(\rho_t,\w{\rho}) \le  \frac{2}{p} \sqrt{\mc{E}_{p,\mc{L}}(\gs^{-1}(\rho_t))}\,.
\end{align*}
\end{lemma}

\begin{proof}
By definition \eqref{def:upper_deriv} and the triangle inequality, we have 
\begin{align*} 
    \frac{d^+}{dt}W_{2,p}(\rho_t,\w{\rho}) = \limsup_{s \to 0^+} \frac{1}{s}\big(W_{2,p}(\rho_{t + s},\w{\rho}) - W_{2,p}(\rho_t,\w{\rho})\big) \le \limsup_{s \to 0^+} \frac{1}{s}W_{2,p}(\rho_t,\rho_{t+s})\,.
\end{align*}
The expression \eqref{def:wp3} with a time scaling gives 
\begin{align*}
W_{2,p}(\rho_t,\rho_{t+s}) \le \int_t^{t+s} \sqrt{g_{p,\gamma(\tau)}(\dot{\gamma}(\tau),\dot{\gamma}(\tau))} \ d\tau\,, 
\end{align*}
for any smooth curve $\gamma$ in $\dhh$ such that $\gamma(t) = \rho_t$ and $\gamma(t+s) = \rho_{t+s}$.  Note from \eqref{eq:ep_divi} and \eqref{eq:grad_entropy} that 
\begin{align*}
    g_{p,\gamma(\tau)}(\dot{\gamma}(\tau),\dot{\gamma}(\tau)) = \frac{4}{p^2} \mc{E}_{p,\mc{L}}(\gs^{-1}(\gamma(\tau)))\,.
\end{align*}
Therefore, we can find 
\begin{equation*}
    \frac{d^+}{dt}W_{2,p}(\rho_t,\w{\rho}) \le \limsup_{s \to 0^+} \frac{1}{s} \int_t^{t+s} \frac{2}{p} \sqrt{\mc{E}_{p,\mc{L}}(\gs^{-1}(\gamma(\tau)))}  \ d\tau = \frac{2}{p} \sqrt{\mc{E}_{p,\mc{L}}(\gs^{-1}(\rho_t))}\,.   \qedhere
\end{equation*}
\end{proof}

\begin{theorem} \label{them:riccitohwi}
If \eqref{eq:ricci_hess} holds for some $\kappa \in \R$, then the following HWI-type inequality holds:
\begin{equation} \label{eq:HWI}
    \mc{F}_{p,\si}(\rho) \le \frac{2}{p} W_{2,p}(\rho, \si) \sqrt{\mc{E}_{p,\mc{L}}(\gs^{-1}(\rho))}- \frac{\kappa}{2} W_{2,p}(\rho, \si)^2\,, \q \text{for all}\ \rho \in \dh\,. 
\end{equation}
\end{theorem}

\begin{proof}
It suffices to consider $\rho \in \dhh$, since $\dhh$ is dense in $\dh$ and the inequality \eqref{eq:HWI} is continuous with respect to $\rho$. Letting $\rho_t = \mc{P}_t^\dag \rho$ for $\rho \in \dhh$, we derive from Lemma \ref{lem:deriva_wasser} that
\begin{align*}
    - \frac{1}{2} \frac{d^+}{d t}\Big|_{t =0} W_{2,p}(\rho_t, \si)^2 = &\liminf_{t \to 0^+} \frac{1}{2t} \big(W_{2,p}(\rho, \si)^2 - W_{2,p}(\rho_t, \si)^2\big) \\
    \le & \limsup_{t \to 0^+} \frac{1}{2t} \big( W_{2,p}(\rho, \rho_t)^2 + 2 W_{2,p}(\rho, \rho_t) W_{2,p}(\rho_t,\si) \big) \\
    \le &  \frac{2}{p} \sqrt{\mc{E}_{p,\mc{L}}(\gs^{-1}(\rho))} W_{2,p}(\rho,\si)\,.
\end{align*}
Then, by above estimate, recalling the EVI \eqref{eq:evi} in Proposition \ref{prop:geodesic_convex} with $\rho_0 = \rho$ and $\rho_1 = \si$, we obtain
\begin{align*}
        \mc{F}_{p,\si}(\rho) & \le - \frac{1}{2}\frac{d^+}{d t}\Big|_{t= 0} W_{2,p}(\rho_t, \si)^2 - \frac{\kappa}{2} W_{2,p}(\rho, \si)^2 \\
        & \le \frac{2}{p} W_{2,p}(\rho,\si) \sqrt{\mc{E}_{p,\mc{L}}(\gs^{-1}(\rho))} - \frac{\kappa}{2} W_{2,p}(\rho, \si)^2 \,.  \qedhere
\end{align*}
\end{proof}

As a direct consequence of Theorem \ref{them:riccitohwi} above, in the case of positive Ricci curvature lower bound, we can obtain the quantum Beckner's inequality.

\begin{corollary} \label{coro:poricci_beck}
If \eqref{eq:ricci_hess} holds for some $\kappa > 0$, then the quantum $p$-Beckner's inequality \eqref{eq_beck} holds with constant \mb{$\alpha_p \ge \kappa /2 $}. 
\end{corollary}

\begin{proof}
By Theorem \ref{them:riccitohwi} and Young's inequality, we have 
\begin{align*}
       \mc{F}_{p,\si}(\rho) \le \frac{2}{p} \Big( \frac{1}{2 C} W_{2,p}(\rho, \si)^2 + \frac{C}{2} \mc{E}_{p,\mc{L}}(\gs^{-1}(\rho)) \Big) - \frac{\kappa}{2} W_{2,p}(\rho, \si)^2\,.
\end{align*}
Letting $C = 2/(p\kappa)$, by definition \eqref{eq_beck}, we complete the proof. 
\end{proof}

Another implication of positive Ricci curvature is the finite diameter of the metric space $(\dh, W_{2,p})$, which can be viewed as a noncommutative Bonnet–Myers theorem.

\begin{corollary}
If \eqref{eq:ricci_hess} holds for some $\kappa > 0$, then it holds that
\begin{equation} \label{est:diameter}
    \sup_{\rho_0, \rho_1\in \mc{D}(\mc{H})} W_{2,p}(\rho_0,\rho_1)^2 \le \frac{8}{\kappa p (p-1)} (\si_{\min}^{1-p} - 1)\,,
\end{equation}
where $\si_{\min}$ is the minimal eigenvalue of the invariant state $\si \in \dhh$. 
\end{corollary}

\begin{proof}
Note that the geodesic convexity \eqref{eq:geo_conv} gives 
\begin{align} \label{auxest:diameter}
     \frac{\kappa}{8} W_{2,p}(\rho_0,\rho_1)^2 \le \frac{1}{2} \mc{F}_{p,\si}(\rho_0) + \frac{1}{2} \mc{F}_{p,\si}(\rho_1) \,.
\end{align}
Then the estimate \eqref{est:diameter} follows from \eqref{eq:rela_sandpdivi}, \eqref{eq:suprhodp}, and \eqref{auxest:diameter}.
\end{proof}

We say that a primitive QMS with $\si$-{\rm GNS DBC} satisfies a transport cost inequality associated with the distance $W_{2,p}$ with constant $c > 0$ if for all $\rho \in \dh$,
\begin{equation} \label{ineq:tc}
    W_{2,p}(\rho,\si) \le \sqrt{c \mc{F}_{p,\si}(\rho)}\,. \tag{TCp}
\end{equation}
We will show the chain of quantum functional inequalities \eqref{eq:chain}.

\begin{proposition} \label{propa}
Suppose that $p$-Beckner's inequality \eqref{eq_beck} holds for some $p \in (1,2]$. Then the transport cost inequality  \eqref{ineq:tc} holds with constant \mb{$c \ge 1/\alpha_p$}. 
\end{proposition}

\begin{proof}
Again, it suffices to consider $\rho \in \dhh$. Let $\rho_t = \mc{P}_t^\dag \rho$ and define the function 
\begin{align*}
    h(t) := W_{2,p}(\rho_t,\rho) + \sqrt{c \mc{F}_{p,\si}(\rho_t)}\,, \q  t \ge 0\,.
\end{align*}
Clearly, $h(t)$ satisfies that $h(0)= \sqrt{c \mc{F}_{p,\si}(\rho)}$ and $h(t) \to W_{2,p}(\si,\rho)$ as $t \to \infty$ by \eqref{eq:conver_qms}. We now claim that when \mb{$c \ge 1/\alpha_p$}, $\frac{d^+}{d t} h(t) \le 0$ holds for $t \ge 0$, which completes the proof. By Lemma \ref{lem:deriva_wasser} and \eqref{eq:ep_divi}, when $\rho_t \neq \si$, we compute 
\begin{align*}
    \frac{d^+}{d t} h(t) &\le \frac{2}{p} \sqrt{\mc{E}_{p,\mc{L}}(\gs^{-1}(\rho_t))} -  \frac{2\sqrt{c}}{p^2 \sqrt{\mc{F}_{p,\si}(\rho_t)}} \mc{E}_{p,\mc{L}}\big(\gs^{-1}(\rho_t)\big) \\
    & = \frac{2}{p} \sqrt{\mc{E}_{p,\mc{L}}(\gs^{-1}(\rho_t))} \Big(1 -  \frac{\sqrt{c}}{p \sqrt{\mc{F}_{p,\si}(\rho_t)}} \sqrt{\mc{E}_{p,\mc{L}}(\gs^{-1}(\rho_t))}\big) \Big) \le 0\,,
\end{align*}
where the last inequality follows from \mb{$c \ge 1/\alpha_p$} and the Beckner's inequality \eqref{eq_beck}. If $\rho_{t_0} = \si$ for some $t_0$, then $\rho_t = \si$ for $t \ge t_0$ and hence $\frac{d^+}{d t} h(t) = 0$ for $t \ge t_0$. 
\end{proof}

\begin{proposition} \label{propb}
If the transport cost inequality \eqref{ineq:tc} holds with constant $c$, then the Poincar\'{e} inequality \eqref{ineq_pi} holds with $f = \vp_p$ and constant $\lad \ge 2/c$.
\end{proposition}

\begin{proof}
We consider $X \in \mc{B}_{sa}(\mc{H})$ with $\tr(\si X) = 0$, and define $\rho_\ep = \gs(\mi + \ep X)$.  Recall Theorem \ref{thm:main_wasser} and let 
$(\gamma_\ep, {\bf B}_\ep)$ be the minimizer to \eqref{def:wp4} for $W_{2,p}(\rho_\ep,\si)$. Then, note from \eqref{eq:chi_heisen} that 
\begin{align} \label{auxeqq_tctopi}
    \norm{X}_{\si,\vp_p}^2 & =  \frac{1}{\ep^2} \Big\l \Omega_\si^{\kappa_{1/p}}(\rho_\ep - \si), \int_0^1 - \ddiv {\bf B}_\ep(s)  ds    \Big\r \notag \\
    & \le \frac{1}{\ep^2} \Big( \int_0^1 \norm{\na \Omega_\si^{\kappa_{1/p}}(\rho_\ep - \si)}^2_{p,\gamma_\ep(s)} ds  \Big)^{1/2} \Big(  \int_0^1 \norm{{\bf B}_\ep(s)}^2_{-1,p,\gamma_\ep(s)}  ds  \Big)^{1/2} \notag \\
     & \le \frac{1}{\ep^2} \Big( \int_0^1 \norm{\na \Omega_\si^{\kappa_{1/p}}(\rho_\ep - \si)}^2_{p,\gamma_\ep(s)} ds  \Big)^{1/2} W_{2,p}(\rho_\ep,\si)\,,
\end{align}
by using the continuity equation in the first line, and Cauchy's inequality in the second line. Applying \eqref{ineq:tc} with the expansion \eqref{eq:asy_divi}, by definition \eqref{def:quanpdivi} of $\mc{F}_{p,\si}(\rho)$ and the relation \eqref{eq:chi_heisen}, 
we find 
\begin{align} \label{auxeq:wasserexp}
    \frac{1}{\ep} W_{2,p}(\rho_\ep,\si) \le \sqrt{ \frac{c}{2} \norm{X}_{\si,\vp_p}^2 + O(\ep)}\,.
\end{align}
A direct computation with \eqref{eq:kernel_iden_2} and Lemma \ref{lem:key_rela} gives  
\begin{align} \label{newlabel_aux}
    \p_j \Omega_\si^{\kappa_{1/p}}(\rho_\ep - \si) & = \ep \p_j \gs^{-1}\vp_p(\Delta_\si) R_\si \gs^{-1} \gs X \notag \\ 
    & = \ep \p_j [\si]_{p,0}^{-1} \gs X  \notag\\
    & = \ep [\si]_{p,\ww_j}^{-1} \gs \p_j  X\,.
\end{align}
By the proof of Lemma \ref{lem:exten_bry}, $\norm{\rho_\ep -\si}_1 \to 0$ as $\ep \to 0$ implies $W_{2,p}(\rho_\ep,\si) \to 0$, which further yields $W_{2,p}(\gamma_\ep(t),\si) = |1 - t| W_{2,p}(\rho_\ep,\si) \to 0$, as $\ep \to 0$, for all $t \in (0,1)$. Moreover, using Lemma \ref{lem:lower_bound}, we have  $\norm{\gamma_\ep(t)-\si}_1  \to 0$, as $\ep \to 0$, for  $t \in (0,1)$. 
It follows from the dominated convergence theorem that
\begin{align*}
   \frac{1}{\ep^2} \int_0^1 \norm{\na \Omega_\si^{\kappa_{1/p}}(\rho_\ep - \si)}^2_{p,\gamma_\ep(s)} ds  & \overset{\eqref{newlabel_aux}}{=} \sum_{j = 1}^J \int_0^1 \l [\si]_{p,\ww_j}^{-1} \gs \p_j  X, [\gamma_\ep(s)]_{p,\ww_j} [\si]_{p,\ww_j}^{-1} \gs \p_j  X \r \\
   & \to  \sum_{j = 1}^J \int_0^1 \l  \gs \p_j  X,  [\si]_{p,\ww_j}^{-1} \gs \p_j  X \r \overset{\eqref{auxeq_repinn}}{=}  - \l X, \mc{L} X \r_{\si,\vp_p}\,, \q \text{as}\ \ep \to 0\,.
\end{align*}
 Combining the above formula with \eqref{auxeqq_tctopi} and \eqref{auxeq:wasserexp}, we conclude 
\begin{equation*}
     \norm{X}_{\si,\vp_p}^2 \le - \frac{c}{2}  \l X, \mc{L} X \r_{\si,\vp_p}\,. \qedhere
\end{equation*}
\end{proof}

We have seen that the entropic Ricci curvature lower bound \eqref{eq:ricci_hess} can imply a sequence of quantum functional inequalities. A natural and important following-up question is how to estimate the lower bound $\kappa$ for the Ricci curvature. Following closely the arguments in \cite[Theorem 10.6]{carlen2020non}, we can explicitly estimate the Ricci curvature lower bound for the depolarizing semigroup by definition \eqref{eq:ricci_hess}. 

\begin{proposition}  \label{prop:ricci_depol}
Let $\mc{L}_{\rm depol}$ be the generator \eqref{def:depol} for the depolarizing semigroup with $\gamma > 0$ and $\si = \mi/d$. Then the Ricci curvature of $\mc{L}_{\rm depol}$ is bounded below by $\gamma p/2$.
\end{proposition}

\begin{proof}
Note from the definition of $\mc{L}_{\rm depol}$ that $\p_j \mc{L}_{\rm depol} = - \gamma \p_j$.
Recalling \eqref{eq:hess_of_f}, and we compute 
\begin{align*}
    -\l \mc{L}_{\rm depol} U, \mf{D}_{p,\si} U \r =  \gamma d^{1-p} \sum_{j = 1}^J  \l \p_j U, \theta_p(\rho,\rho)   \p_j U\r\,.
\end{align*}
By definition \eqref{def:kgeode}, we can also calculate 
\begin{align*}
    \sum_{j = 1}^J \big\l \p_j U,  \mc{K}_{\rho,\mc{L}_{\rm depol}^\dag(\rho)}^{(1),j} \left[ \p_j U \right] \big \r & = \gamma d^{1-p} \sum_{j =1}^J \big\l  \p_j U,
 (\d_1 \theta_p)\big((\rho,\rho), \rho\big)\big[ \frac{\mi}{d} - \rho, \p_j U \big] \big\r \\
 & =  \gamma d^{1-p} \sum_{j =1}^J \big\l  \p_j U, L_{\frac{\mi}{d} - \rho} (\p_1 \theta_p)(\rho,\rho)[\p_j U] \big\r\,,
\end{align*}
since $\mc{L}^\dag_{\rm depol} = \gamma\big(\frac{\mi}{d} - \rho \big)$. Similarly, we have 
\begin{align*}
    \sum_{j = 1}^J \big\l \p_j U,  \mc{K}_{\rho,\mc{L}_{\rm depol}^\dag(\rho)}^{(2),j} \left[ \p_j U \right] \big \r 
 =  \gamma d^{1-p} \sum_{j =1}^J \big\l  \p_j U, R_{\frac{\mi}{d} - \rho} (\p_2 \theta_p)(\rho,\rho)[\p_j U] \big\r\,.
\end{align*}
It follows that 
\begin{align*}
      \hess \mc{F}_{p,\si}(\rho)[U,U] &= \frac{1}{2} \sum_{j = 1}^J \Big\l \p_j U,  \big(\mc{K}_{\rho,\mc{L}_{\rm depol}^\dag(\rho)}^{(1),j} + \mc{K}_{\rho,\mc{L}_{\rm depol}^\dag(\rho)}^{(1),j}\big) \left[ \p_j U \right] \Big \r - \left\l U, \mc{L}_{\rm depol}^\dag(\mf{D}_{p,\rho} U) \right \r \\
      & =  \gamma d^{1-p} \sum_{j =1}^J \Big\l  \p_j U, \big(\frac{1}{2d}\p_1 \theta_p + \frac{1}{2d}\p_2 \theta_p + \frac{p}{2}\theta_p \big)     (\rho,\rho)[\p_j U] \Big\r \\
      & \ge d^{1-p} \frac{p\gamma }{2} \sum_{j =1}^J \Big\l  \p_j U, \theta_p     (\rho,\rho)[\p_j U] \Big\r = \frac{p\gamma}{2} \l U, \mf{D}_{p,\rho} U\r\,,
\end{align*}
where the second line is from $ x \p_x \theta_p(x,y) +  y \p_y \theta_p(x,y) = (2-p) \theta_p(x,y)$; the third inequality follows from $ \p_x \theta_p(x,y) +  \p_y \theta_p(x,y) \ge 0$ by the concavity of $x^{p-1}$.
\end{proof}

However, similarly to the case of estimating the functional inequality constant, there are very few examples where the explicit expressions of the Ricci curvature lower bounds can be obtained. To avoid the complicated computation and estimation based on the definition \eqref{eq:ricci_hess}, Carlen and Maas \cite{carlen2017gradient} consider the following intertwining property of a QMS: for some $\kappa \in \R$ and all $j$, 
\begin{align} \label{eq:intertwin}
    \p_j \mc{P}_t = e^{- \kappa t} \mc{P}_t \p_j\,,
\end{align}
which can be verified for many interesting cases, e.g., Fermi and Bose Ornstein-Uhlenbeck semigroups \cite[Section 6]{carlen2017gradient}. They showed that under the condition \eqref{eq:intertwin}, the Ricci curvature of the QMS $\mc{P}_t$ associated with the relative entropy $D(\rho\|\si)$ is bounded from below by $\kappa$. The key step in their argument is the monotonicity of the action functional:  
\begin{align} \label{eq:monotone}
    \l \mc{P}^\dag_t A, [\mc{P}^\dag_t\rho]_{\ww}^{-1} \mc{P}^\dag_t A\r \le  \l A, [ \rho]_{\ww}^{-1}   A\r\,, 
\end{align}
where $\mc{P}_t$ is the primitive QMS satisfying $\si$-{\rm GNS DBC}. To extend their approach to our case, we need to show a similar monotonicity result as \eqref{eq:monotone}  for $ \l A, [\rho]_{p,\ww}^{-1}  A\r$. It is a nontrivial task, since $\l A, [\rho]_{p,\ww}^{-1}  A\r$ is not $1$-homogeneous so that the contractivity of $\l A, [\rho]_{p,\ww}^{-1}  A\r$ under $\mc{P}_t^{\sss \dag}$ can not be implied from its joint convexity \cite{lesniewski1999monotone}. For the symmetric QMS, $\mc{P}_t = \mc{P}_t^{\sss \dag}$ is the unital quantum channel for each $t$. Note that \cite[Theorem 5.1]{zhang2021some} has shown that 
\begin{equation} \label{auxeq_p-nomo}
       \l \Phi(A), [\Phi(\rho)]_{p,0}^{-1} \Phi(A)\r \le  \l A, [ \rho]_{p,0}^{-1}  A\r\,, \q \forall \rho \in \dhh\,, A \in \bh\,,
\end{equation}
holds for any unital quantum channel $\Phi$. Thus, in this case, it is straightforward to conclude as in \cite[Theorem 10.9]{carlen2020non} that if the primitive symmetric QMS $\mc{P}_t$ satisfies 
the property \eqref{eq:intertwin},  then its Ricci curvature associated with $\mc{F}_{p,\si}$ has a lower bound $\kappa$. It is known \cite{carlen2017gradient,carlen2020non} that the infinite temperature Fermi Ornstein-Uhlenbeck semigroup is symmetric and satisfies \eqref{eq:intertwin} with $\kappa = 1$, which readily gives that it has Ricci curvature lower bound $1$ and then Beckner's inequality \eqref{ineq_becp} holds with \mb{$\alpha_p \ge 1/2$} by Corollary \ref{coro:poricci_beck}. But it seems not easy to extend the monotonicity result \eqref{auxeq_p-nomo} beyond the symmetric regime, namely, to show the monotonicity $\l A, [\rho]_{p,\ww}^{-1}  A\r$ under quantum channels with $\si$-{\rm GNS DBC}. We choose to investigate it in the future.

\section{Conclusions and discussion} \label{sec:discussion}
We have introduced the families of quantum Beckner's inequalities \eqref{ineq_becp} and \eqref{ineq_dbecq} \mb{on a finite-dimensional matrix algebra} that interpolate between Sobolev-type and Poincar\'{e} inequalities. The basic properties of Beckner's inequality, e.g., the monotonicity, the uniform
positivity, and the stability of the optimal constant, have been investigated in detail. We have also discussed their relations \mb{with the hypercontractivity} and other known quantum functional inequalities and applied Beckner's inequalities \eqref{ineq_becp} to estimating the mixing time and deriving moment estimates. Furthermore, we have provided a quantum optimal transport framework for analyzing Beckner's inequalities. In doing so, we have defined a new class of quantum transport distances $W_{2,p}$ such that the QMS with \mb{$\si$-{\rm GNS DBC}} is the gradient flow of the $p$-divergence $\mc{F}_{p,\si}$. The main properties of the metric space $(\dh,W_{2,p})$ have been analyzed. 
We have then introduced the associated entropic Ricci curvature and showed that it could yield a number of implications between \eqref{ineq_becp}, \eqref{ineq_pi2}, an HWI-type inequality, and a transport cost inequality. This provides an interesting starting point for an optimal transport-inspired
approach to study quantum Beckner's inequalities.

We briefly discuss below some further generalizations and applications of our results and methods. The details and refinements of these results might be worth being reported elsewhere.

\begin{enumerate}[\textbullet]
    \item As mentioned in the introduction, the tensorization property for quantum functional inequalities is much more subtle than classical ones. Some tensorization-type results have been obtained for quantum MLSI and LSI, e.g., \cite[Lemma 25]{kastoryano2013quantum}, \cite[Section 4]{beigi2020quantum} and \cite[Theorem 9]{temme2014hypercontractivity}; see also \cite{king2014hypercontractivity,montanaro2008quantum}. By comparison results in Section \ref{sec:positive_stab}, these results 
    can be easily adapted to provide dimension-independent lower bounds for the quantum Beckner constant $\alpha_p(\mc{L})$ in certain scenarios. 
    \mb{For instance, let $\Psi_t(X) = e^{-t} X + \frac{1}{2} (1 -  e^{- t}) \tr(X) \mi$ be the qubit depolarizing 
    channel. \cite[Theorem 1]{king2014hypercontractivity} shows that
    \begin{align*}
        \norm{\Psi_t^{\otimes n}(X)}_{q,\frac{\mi}{2^n}} \le \norm{X}_{p,\frac{\mi}{2^n}}\,,\q \text{for}\ q \ge p  > 1\,,\ t \ge \log \sqrt{\frac{q-1}{p-1}}\,.
    \end{align*}
    Then, by Lemma \ref{lem:sobo_hyper} and Proposition \ref{prop:connectpsob}, we have $\alpha_p \ge \frac{1}{p} \w{\beta}_p \ge \frac{1}{2p}$ for $p \ge 1$.} 
    \item Thanks to Lemma \ref{lem:two-sided}, we can estimate the strong data processing inequality constant for the quantum $p$-divergence $\mc{F}_{p,\si}$ in terms of the $\chi^2_{\kappa_{1/p}}$-contraction coefficient, in the sense of \cite[Theorem 4.1]{gao2021complete}. We can also 
    discuss the stability of the data processing inequality for the divergence $\mc{F}_{p,\si}$ similarly to \cite[Proposition 5.1]{junge2019stability} and 
    the approximate tensorization property of $\mc{F}_{p,\si}$ similarly to \cite[Theorem 5.1]{gao2021complete}. 
    \item It is straightforward to generalize the arguments in \cite{wirth2022dual} to derive the dual formulation for the distance $W_{2,p}$ in terms of noncommutative Hamilton-Jacobi-Bellman-type equations. Indeed, a formal calculation gives 
\begin{align*}
\frac{1}{2} W_{2,p}^2(\rho_0,\rho_1) = \sup\Big\{\tr(A(1)\rho_1 - A(0)\rho_0)\,;\ \tr(\dot{A}(t)\rho) + \frac{1}{2}\norm{\na A(t)}^2_{p,\rho} \le 0\,,\ \forall \rho \in \dh \Big\}\,.
\end{align*}
The above dual formula allows us to fit our distance $W_{2,p}$ into the general framework of noncommutative transportation metrics recently proposed in \cite{gao2021ricci}, and to further discuss the coarse Ricci curvature of quantum channels with respect to $W_{2,p}$. Moreover, 
    in view of \cite{wirth2021curvature}, it is also straightforward to consider the curvature-dimension conditions for quantum systems and investigate the finite-dimensional version of quantum Beckner's inequalities \eqref{ineq_becp}.  
\item Note that one main ingredient for our analysis is 
Alicki’s theorem in Lemma \ref{lem:struc_gene}, which actually holds for the QMS with $\si$-{\rm GNS DBC} on an 
arbitrary finite-dimensional unital $*$-subalgebra; see  
\cite[Corollary 5.4]{wirth2022christensen}. Hence, we are allowed to extend most results in this work to the QMS not necessarily acting on $\mc{B}(\mc{H})$ but only
on a $*$-subalgebra which includes the finite state Markov chain as a special case; see also \cite{carlen2020non}. In particular, it connects the classical Beckner's inequality with a class of Wasserstein distance, which enables us to investigate the Beckner constant for Markov chains in terms of the Ricci curvature lower bound as in \cite{fathi2016entropic}. 

\item For the numerical computation of functional inequality constants, there is a recent research line attempting to estimate the classical log-Sobolev constant by the sum-of-squares relaxation \cite{faust2021sum}. It is interesting to extend their method to quantum functional inequalities.
\end{enumerate}
 
We conclude with some interesting and important open questions. 
\begin{enumerate}[\textbullet] 
    \item \mb{There are two technical questions we are unable to solve. The first one is stated in Remark \ref{rem:positivediri}:  for a primitive Lindbladian $\mc{L}$ satisfying $[\si]_{p,0}$-DBC, 
    \begin{equation} \label{ques1}
        \mc{E}_{p,\mc{L}}(\gs^{-1}(\rho)) > 0\,,\q \text{for}\ \rho \in \dhh\ \text{with}\ \rho \neq \si\,.
    \end{equation}
    The case of $p = 1$ has been proved in \cite[Proposition\,6.2]{brooks2022characterisation}, from which we easily see that the key step for \eqref{ques1} is the convexity of $p$-Dirichlet form $\mc{E}_{p,\mc{L}}$ which seems open. Another one is mentioned at the end of Section \ref{sec:qot_beck}: the monotonicity $\l \Phi(A), [\Phi(\rho)]_{p,\ww}^{-1} \Phi(A)\r \le  \l A, [ \rho]_{p,\ww}^{-1}  A\r$  under quantum channels with $\si$-{\rm GNS DBC}, for which the characterization for such channels \cite[Lemma 13]{beigi2020quantum} might be helpful. Addressing these questions would directly make our results more complete.
    }
    \item \mb{One important feature of the classical Beckner's inequality is its ability to capture the tail behavior of the family of probability measures: $d \mu_\alpha(x) = c_\alpha e^{-(1+x^2)^{\alpha/2}} dx$ on $\R$ with $\alpha \in [1,2]$, which satisfies the Poincar\'{e} inequality for all $\alpha$ but satisfies LSI only for $\alpha = 2$ \cite{bakry2014analysis,barthe2001levels,latala2000between}. In detail, \cite{latala2000between} considers the following variant of \eqref{ineq_e}: for $a \in [0,1]$ and some $C_a > 0$, 
    \begin{align} \label{eq_lo}
         \sup_{q \in [1,2)} \frac{1}{(2-q)^a} \big( \mu[g^2] - \mu[g^q]^{2/q} \big)  \le C_a \mc{E}(g, g)\,,
    \end{align}
    where $\mu$ is a probability measure. It can be shown that the measure $\mu_\alpha$ satisfies the inequality \eqref{eq_lo} with parameter $a \in [0,1]$ if and only if $\alpha \ge \frac{2}{2-a}$; see \cite[Section 7.6]{bakry2014analysis}. It would be very much desirable to find a similar example in the quantum setting. 
    Moreover, while 
    we have obtained several generic results for the quantum Beckner constant, it would also be beneficial to establish more quantitative estimates (even numerically) for some concrete, physically meaningful models, for instance, the quantum spin system \cite{capel2020modified} and the quantum Markov semigroups constructed from the classical ones via the group transference \cite{gao2020fisher}. 
    
    In view of these, it seems necessary to develop the theory of quantum Beckner's inequalities on a general von Neumann algebra to access more advanced examples to enrich the applications; see \cite{brannan2021complete,brannan2022complete,wirth2021complete,wirth2018noncommutative} and references therein for recent progress on quantum MLSI in the operator algebra framework. }
   
\item We have only considered the quantum Beckner's inequality for the primitive QMS. It would be very useful to extend the theory to the non-primitive setting, which might be necessary for the general tensorization results and most of the applications. The study of non-primitive quantum functional inequalities is initiated by Bardet \cite{bardet2017estimating} and has become an active research topic in recent years. \mb{Here, we have included a short discussion on the non-primitive Beckner's inequality in Appendix \ref{app:nonprimitive} to 
stimulate the further investigation in the spirit of recent works \cite{gao2020fisher,brannan2021complete,brannan2022complete,gao2021complete}. }
\end{enumerate}

\vspace{3 mm}
\noindent\textbf{Funding}\ \ The work is supported in part by National Science Foundation via award CCF-1910571.

\vspace{3 mm}
\noindent\textbf{Data Availability}\ \  Data sharing not applicable to this article as no datasets were generated or analysed during
the current study.

\section*{Declarations}

\noindent\textbf{Conflict of interest} \ \ The authors declare no conflict of interest.

\begin{appendix}

\section{Some additional preliminaries} \label{app:pre}

\subsection{Quantum \texorpdfstring{$\chi^2$}--divergence} \label{app:quantum_diver} 
In this appendix, we briefly recall the quantum $\chi^2$-divergences introduced in \cite{temme2010chi}. For any $\si \in \dhh$ and $\kappa: (0,\infty) \to (0,\infty)$, we define the operator: 
\begin{align} \label{def:omega}
    \Omega_\si^\ka = R_\si^{-1} \kappa(\Delta_\si):\ \bh \to \bh\,.
\end{align}
 Recalling $J_\si^f$ in \eqref{def:operator_kernel}, clearly there holds
\begin{align} \label{aux_rela}
    (\Omega_\si^\kappa)^{-1} = J_\si^{1/\kappa}\,.
\end{align}
The quantum $\ck$-divergence for $\rho\in \dh$ and $\si \in \dhh$ is defined by
\begin{align} \label{def:quantumchi}
    \chi_{\kappa}^2(\rho,\sigma) = \l \rho - \si, \Omega_\si^\kappa(\rho-\si)\r\,.
\end{align}
To make the divergence $\chi_{\kappa}^2(\dd,\dd)$ have nice properties, we usually consider $\kappa$ in the following functional class:   
\begin{align*}
     \mc{K} = \{\kappa: (0,\infty) \to (0,\infty)\,;\ \kappa \ \text{is operator convex},\ x\kappa(x) = \kappa(x^{-1}),\ \kappa(1) = 1\}.
\end{align*}
For the purposes of this work, the following family of power difference means in $\mc{K}$ is of particular interest \cite{hiai1999means}:
\begin{align}\label{def:pdiff}
\kappa_\alpha = \frac{\alpha}{\alpha-1} \frac{x^{\alpha - 1} - 1}{x^{\alpha} - 1}\,, \q \alpha \in [-1,2]\,.
\end{align}
In fact, the kernel function of the operator $J_\si^{1/k}$ in \eqref{aux_rela}  is given by  
\begin{align*}
    M_\alpha = y \ka_\alpha^{-1}(x/y) = \frac{\alpha - 1}{\alpha}\frac{x^\alpha - y^\alpha}{x^{\alpha - 1} - y^{\alpha - 1}}\,,
\end{align*} 
which is the so-called A-L-G  interpolation mean since $M_\alpha$, for $\alpha = -1$, $\alpha = 1/2$, $\alpha  = 1$, and $\alpha = 2$, gives the harmonic mean,  the geometric mean, the logarithmic mean, and the arithmetic mean, respectively.

\subsection{Noncommutative calculus} \label{app:NCC}
In this appendix, following \cite{carlen2020non,bardet2017estimating}, we 
review some fundamentals about noncommutative calculus associated with the derivation $\p_j$. Let  $A,B \in \mc{B}_{sa}(\mc{H})$ admit the spectral decompositions:
\begin{align*}
    A = \sum_{i = 1}^{d}\lad_i A_i\,,\q B = \sum_{k=1}^{d} \mu_k B_k\,,
\end{align*} 
where $\lad_i$ and $\mu_k$ are eigenvalues of $A$ and $B$, respectively; $A_i$ and $B_i$ are the associated rank-one spectral projections. For a function $f \in C(I \t I)$ with $I$ being a compact interval containing the spectra of $A$ and $B$,  we define the Schur multiplier (double sum operator) by \cite{birman2003double,potapov2011operator,de2004differentiation}
\begin{align} \label{def:doutble_op_sum}
    f(A,B) = \sum_{i,k = 1}^{d} f(\lad_i,\mu_k) L_{A_i}R_{B_k}\,,
\end{align}
where $C(I \t I)$ is the Banach space of complex-valued continuous functions on $I \t I$. It was observed in \cite{bardet2017estimating} that given $A,B \in \mc{B}_{sa}(\mc{H})$, $f(A,B)$ is $*$-representation between $C(I \t I)$ and $\mc{B}(\mc{B}(\mc{H}))$. Indeed, we have the following lemma from \cite[Lemma 4.1]{bardet2017estimating} and \cite[Lemma 6.6]{carlen2020non}. 
\begin{lemma} \label{lem:double_inner}
Let $A,B \in \mc{B}_{sa}(\mc{H})$ and the compact interval $I$ contain the spectra of $A$ and $B$. It holds that 
\begin{enumerate}[1.]
    \item $f(A,B)g(A,B) = (fg)(A,B)$ for $f,g \in C(I \t I)$. 
    \item If $f \in C(I \t I)$ is non-negative, then $f(A,B)$ is a positive semidefinite operator on $B(\mc{H})$ with respect to the inner product $\l \dd, \dd \r$. It follows that if $f$ is strictly positive, the sequilinear form $\l \dd, f(A,B)(\dd)\r$ defines an inner product on $\mc{B}(\mc{H})$. 
\end{enumerate}
\end{lemma}



In this work, we mainly consider the case where $f$ is the divided difference of some differentiable function $\vp$ on $I$:
\begin{align} \label{def:divi_diff}
    \vp^{[1]}(\lad,\mu) = 
    \begin{cases}
        \frac{\vp(\lad)-\vp(\mu)}{\lad-\mu}\,, \q & \lad \neq \mu\,, \\ 
        \vp'(\lad)\,,\q  & \lad = \mu\,,
    \end{cases}
\end{align}
which is closely related to the chain rule for $\p_j$ (cf.\cite[Lemma 4.2]{bardet2017estimating} and \cite[Proposition 6.2]{carlen2020non}). 
\begin{lemma} \label{lem:chain_rule}
Under the same assumption as in Lemma \ref{lem:double_inner}, for any $f: I \to \C$, we have, for $V \in \bh$,
\begin{align} \label{eq:chain_rule}
    V f(B) - f(A) V = f^{[1]}(A,B)(V B - A V)\,.
\end{align}
\end{lemma}
Then, by Lemma \ref{lem:chain_rule}, for a differentiable curve $A(t): (a,b) \to \mc{B}(\mc{H})$ and function $f$, we have 
\begin{align}\label{eq:chain_curve}
& \frac{d}{d t} f(A(t)) = f^{[1]}(A(t),A(t))(A'(t))\,.    
\end{align}
We also need a multiple operator version of \eqref{eq:chain_curve}. We recall that for a differentiable function $\vp: \R^n \to \C$, the partial divided difference $\d_j \vp: \R^{n+1} \to \C $ with respect to the variable 
$x_j$ is defined by 
\begin{align} \label{def_1:divi_diff}
(\d_{j} \vp)(x_1,\cdots,x_{j-1},(\lad,\mu),x_{j+1},\cdots,x_n)=  \left(\vp(x_1,\cdots,x_{j-1},\dd,x_{j+1},\cdots,x_n)\right)^{[1]}(\lad,\mu)\,.
\end{align}
Let $A^{(k)}$, $k = 1,\ldots,n$, be self-adjoint operators with the spectral decompositions: $A_i = \sum_{i} \lad_{i}^{(k)}A_{i}^{(k)}$, where $\lad_i^{(k)}$ are eigenvalues and $A_{i}^{(k)}$ are the associated rank-one spectral projections. For a function $\vp: I \t \cdots \t I \to \C$ with the interval $I$ containing the spectra of $A^{(i)}$,
the multiple operator sum is defined as:
\begin{align} \label{def_2:multiple}
    \vp(A_1,\cdots,A_n) = \sum_{i_1,\cd,i_n = 1}^{d} \vp(\lad^{(1)}_{i_1},\cd,\lad^{(n)}_{i_n})A_{i_1}^{(1)}\otimes\cd \otimes A_{i_n}^{(n)}.
\end{align}
The following chain rule from \cite[Proposition 6.8]{carlen2020non} shall be useful in the expression of the geodesic equations for the generalized quantum transport distance.
\begin{lemma} \label{lem:high_order}
Let the curves $A_t, B_t: (a,b) \to \mc{B}_{sa}(\mc{H})$ be differentiable, and let $\vp: I \t I \to \C$ be differentiable with $I$ containing the spectra of $A_t$ and $B_t$ for all $t \in (a,b)$. Then there holds 
\begin{align*}
   \p_t \vp(A_t,B_t)(\dd) = (\d_1 \vp) ((A_t,A_t),B_t)[\p_t A_t, \dd] + (\d_2 \vp)(A_t,(B_t,B_t))[\dd, \p_t B_t].
\end{align*}
\end{lemma}

\section{Note on the detailed balance condition} \label{app:dbc}

Noting \eqref{auxeqq_rela_00}, it follows from Lemma \ref{lem:self_adjoint} that the QMS satisfying $\si$-{\rm GNS DBC} is also self-adjoint with respect to the inner product $ \l \dd , \dd \r_{[\si]_{p,0}}$. In this appendix, we modify the discussions in \cite[Appendix B]{carlen2020non} and \cite[Appendix B]{benoist2021deviation} to show that the $[\si]_{p,0}$-DBC, for $p \in (1,2)$, and KMS DBC are not comparable, and that there exists a primitive QMS satisfying $[\si]_{p,0}$-DBC but not $\si$-{\rm GNS DBC}. 

Let $\{\ket{0},\ket{1}\}$ be the standard basis of $\C^2$, and $\ket{v_1} = \frac{1}{\sqrt{2}}(\ket{0} + \ket{1})$ and $\ket{v_2} = \frac{1}{\sqrt{5}}(\ket{0} + 2\ket{1})$ be an another basis of $\C^2$. We define the quantum channel $\Phi(X) = K_1^* X K_1 + K_2^* X K_2$ with $K_1 = \ket{v_1}\bra{0}$ and $K_2 = \ket{v_2}\bra{1}$. It is easy to see that the associated unique invariant state of $\Phi^\dag$ is $$\si = \frac{1}{7}\mm 2 & 3 \\ 3 & 5 \nn\,,$$
and the spectrum of $\Phi^\dag$ is given by $\{1, \frac{3}{10}, 0\}$ with $0$ of multiplicity two. We denote by $\Phi^{\dagger}_{\mathrm{KMS}}$ the adjoint of $\Phi$ with respect to the KMS inner product and define $\Psi = \Phi^\dag_{\mathrm{KMS}} \Phi$, which is also a quantum channel with $\Psi^\dag(\si) = \si$ but satisfying the KMS DBC.  Note that for a linear map $\Phi$ on $\mc{B}(\mc{H})$ satisfying both $[\si]_{p,0}$-DBC and KMS DBC, there holds $g(\Delta_\si) \circ \Phi = \Phi \circ g(\Delta_\si)$, where $g(x) = \kappa_{1/p}(x) x^{1/2}$ and $g(\Delta_\si) = [\si]_{p,0}^{-1} \gs$. By definition \eqref{def:pdiff}, we have $g(1/x) = g(x)$. Then, by exactly the same argument as in \cite[Appendix B]{benoist2021deviation}, we can show that the operator $\Psi$ defined above does not commute with $g(\Delta_\si)$ and hence the Lindbladian $\Psi - \id$ satisfies the KMS DBC but not $[\si]_{p,0}$-DBC.

We next consider $\w{\Psi} := [\si]_{p,0}^{-1} \circ \Psi^\dag \circ \gs$, which is a quantum channel, since the operator $[\si]_{p,0}^{-1}$ is completely positive for $p \in (1,2)$ by \cite[Example 4.7]{hiai2013families}.  Then we can check from definition that $\w{\Psi}$ satisfies the $[\si]_{p,0}$-DBC. Again, since $\Psi$ defined above does not commute with $g(\Delta_\si)$, neither does $\w{\Psi}$. Hence, the Lindbladian $\w{\Psi} - \id$ satisfies $[\si]_{p,0}$-DBC but not KMS DBC. 

We conclude with an example of a primitive QMS with $[\si]_{p,0}$-DBC but not $\si$-{\rm GNS DBC}. The construction is modified from \cite{benoist2021deviation} as well. We define, for some $\eta \in (0,1/2)$,  
\begin{align*}
K_1 = \mm \sqrt{\eta} & 0 \\ 0 & \sqrt{1 - \eta} \nn\,,\q K_2 = \mm 0 & \sqrt{\eta} \\ \sqrt{1 - \eta} & 0 \nn\,,
\end{align*}
and the associated quantum channel $\Phi(X) = K_1^* X K_1 + K_2^* X K_2$, which has the unique invariant state:
\begin{equation*}
    \si = \mm \eta & 0 \\ 0 & 1 - \eta \nn\,.
\end{equation*}
It is direct to verify $[\si]_{p,0}^{-1} \circ \Phi^\dag = \Phi \circ [\si]_{p,0}^{-1}$, i.e., $\Phi$ satisfies the $[\si]_{p,0}$-DBC. However, it was shown in \cite{benoist2021deviation} that $\si$-{\rm GNS DBC} does not hold for $\Phi$. It follows that the generator $\Phi - \id$ is the desired example.   

\mb{
\section{Beckner's inequality for non-primitive QMS} \label{app:nonprimitive}

In this appendix, we introduce $p$-Beckner's inequality with $p \in (1,2]$ for the non-primitive QMS and show that it holds for any QMS satisfying {\rm GNS DBC}, which extends \cite[Theorem 3.3]{gao2021complete} for MLSI. 

Let $\mc{P}_t = e^{t \mc{L}}$ be a non-primitive QMS with a full-rank invariant state $\si \in \dhh$, which may be non-unique. We introduce the fixed-point algebra $\mc{F} := \left\{X \in \bh;\ \forall t \ge 0,\, \mc{P}_t(X) = X \right\}$, and denote the associated conditional expectation by $E_\mc{F}$. It is known from \cite{frigerio1982long,gao2021complete} that 
$E_\mc{F} \mc{P}_t = \mc{P}_t E_\mc{F} = E_\mc{F}$ and for $X \in \mc{B}(\mc{H})$,
\begin{align} \label{eq:conver}
\lim_{t \to \infty} \mc{P}_t (X -  E_{\mc{F}}(X)) = \lim_{t \to \infty} \mc{P}_t (X) -  E_{\mc{F}}(X) = 0\,.
\end{align}
Similarly, Beckner's inequality quantifies the convergence rate of \eqref{eq:conver} in terms of the $p$-divergence $  $\eqref{def:quanpdivi}. We consider QMS that satisfies $\si$-{\rm GNS DBC} as in Definition \ref{def:sidbc}, which is well-defined since the self-adjointness of $\mc{L}$ with respect to $\l \dd,\dd\r_{\si,1}$ is independent of the choice of invariant state $\si$; see \cite[Lemma 2.6]{gao2021complete}. 
We compute the entropy production of $\mc{F}_{p,\si}(\rho)$ as in \eqref{eq:ep_divi}: for $\rho_t = \mc{P}_t^\dag(\rho)$ with $\rho \in \dhh$,
\begin{align*}
    \frac{d}{dt}\Big|_{t = 0} \mc{F}_{p,E_\mc{F}^\dag(\rho)}(\rho_t) = -\frac{4}{p^2} \mc{E}_{p,\mc{L}}(X)\,,\q X = \Gamma_{E_\mc{F}^\dag(\rho)}^{-1}(\rho) \,,
\end{align*}
where $\mc{E}_{p,\mc{L}}$ is given in \eqref{def:p_diri} with $\si = E_\mc{F}^\dag(\rho)$.  Hence, we can define the non-primitive Beckner's inequality as \eqref{eq_beck}: for some $\alpha_p > 0$ and any $\rho \in \dhh$, 
\begin{align} \label{def:nonbeck}
    \alpha_p \mc{F}_{p,E_{\mc{F}}^\dag(\rho)}(\rho) \le p^{-2} \mc{E}_{p,\mc{L}}\big(\Gamma_{E_\mc{F}^\dag(\rho)}^{-1}(\rho)\big)\,,
\end{align}
which is equivalent to the exponential convergence: $ \mc{F}_{p,E_{\mc{F}}^\dag(\rho)} (\mc{P}_t^\dag(\rho)) \le e^{- 4 \alpha_p t} \mc{F}_{p,E_{\mc{F}}^\dag(\rho)}(\rho)$.

We introduce the subalgebra index for the fixed-point algebra $\mc{F}$ by
\begin{align} \label{def:index}
C(E_\mc{F}) = \inf\big\{c> 0 \,;\ \rho \le c E_{\mc{F}}^\dag(\rho)\,,\ \forall \rho \in \dh \big\}\,,
\end{align}
which is finite in the finite-dimensional setting \cite[Theorem 6.1]{pimsner1986entropy}.
For a primitive QMS with the unique invariant state $\si \in \dhh$, the index \eqref{def:index} reduces to 
\begin{align} \label{def:constcsi}
    C(\si):= & \inf\{c > 0\,;\ \rho \le c \si
 \ \, \text{for all}\ \rho \in \mc{D}(\mc{H})
\}\,,
\end{align}
which is closely related to the max-relative entropy $D_{\infty}$ in \eqref{def:maxentropy} and can be explicitly represented by \eqref{eq:suprhodp},
\begin{align}  \label{def:constcsi22}
    C(\si) = \sup_{\rho \in \mc{D}(\mc{H})} \exp\left(D_{\infty}(\rho \| \si) \right) = \si_{\min}^{-1}\,.
\end{align}
We also recall the spectral gap (Poincar\'{e} constant) $\lad(\mc{L})$ for a non-primitive Lindbladian $\mc{L}$: 
\begin{align*}
    \lad(\mc{L}) = \inf_{X \in \mc{B}(\mc{H})} \frac{-\l X, \mc{L}(X) \r_{\si,f}}{\norm{X - E_\mc{F}(X)}_{\si,f}^2}\,,
\end{align*}
for an invariant state $\si \in \dhh$ and function $f:(0,\infty) \to (0,\infty)$, where the inner product $\l\dd,\dd\r_{\si,f}$ is given in \eqref{eq:general_inner}. It was proved in \cite[Lemma 3.2]{gao2021complete} that $\lad(\mc{L})$ is independent of the choice of $\si$ and $f$. We are now prepared to give the following result.

\begin{theorem} \label{thm:beck_poincare}
Let $\mc{P}_t = e^{t \mc{L}}$ be a QMS satisfying $\si$-{\rm GNS DBC} for some $\si \in \dhh$. Then the Beckner's inequality  \eqref{def:nonbeck} holds for all $p \in (1,2]$ with constant $\alpha_p(\mc{L})$ satisfying the estimate:
\begin{align} \label{eq:beck_poin}
       \alpha_p(\mc{L}) \ge \frac{p}{4} C(E_\mc{F})^{p-2} \lad(\mc{L})\,.
\end{align}
\end{theorem}

\begin{proof} 
We consider the relative density $X = \Gamma_{E_\mc{F}^\dag(\rho)}^{-1}(\rho)$ for $\rho \in \dhh$ and then have 
\begin{align} \label{auxeqq_1} 
   p^2 \mc{F}_{p,E_{\mc{F}}^\dag(\rho)}(\rho) =  \frac{p}{p-1}\big(\norm{X}_{p,E_{\mc{F}}^\dag(\rho)}^p - 1\big) \le p \norm{X - \mi}_{E_{\mc{F}}^\dag(\rho),\vp_p}^2 \le - \lad(\mc{L})^{-1} p 
    \l X, \mc{L} X \r_{E_{\mc{F}}^\dag(\rho),\vp_p},
\end{align} 
by using $E_{\mc{F}}^\dag \gs = \gs E_\mc{F}$ for any invariant state $\si$ and the upper estimate in \eqref{eq:two_sided_est}. By definition \eqref{def:index} and formulas \eqref{eq:rep_epl} and \eqref{eq:rep_dbcl}, it follows from 
Lemma \ref{lem:mono_norm} that 
\begin{align} \label{auxeqq_2}
 - \left\l X, \mc{L} X \right\r_{\si,\vp_p}
   & =  \left\l \gs^{1/p} \big(\p_j X \big), f_p^{[1]}\left(e^{\ww_j/2p}\si^{1/p}, e^{-\ww_j/2p}\si^{1/p}\right) \gs^{1/p} \big(\p_j X\big) \right\r  \notag \\
   & \le C(E_{\mc{F}})^{2-p} \left\l \gs^{1/p}  \big(\p_j X\big), f_p^{[1]}\left(e^{\ww_j/2p} \gs^{1/p}(X), e^{-\ww_j/2p} \gs^{1/p}(X) \right) \gs^{1/p}\big(\p_j X\big) \right\r \notag \\
   & \le 4 p^{-2} C(E_{\mc{F}})^{2-p} \mc{E}_{p,\mc{L}}(X)\,,
\end{align}
where $\si = E_{\mc{F}}^\dag(\rho)$ and $X = \Gamma_{E_\mc{F}^\dag(\rho)}^{-1}(\rho)$. 
Therefore, by \eqref{auxeqq_1} and \eqref{auxeqq_2}, we obtain
\begin{equation*}
    p^2 \mc{F}_{p,E_{\mc{F}}^\dag(\rho)}(\rho) \le \frac{4}{p \lad(\mc{L})}
C(E_{\mc{F}})^{2-p} \mc{E}_{p,\mc{L}}\big(\Gamma_{E_\mc{F}^\dag(\rho)}^{-1}(\rho)\big)\,.
\end{equation*}
The proof is complete by \eqref{def:nonbeck}. 
\end{proof}

\begin{remark}
    If $\mc{P}_t$ in Theorem \ref{thm:beck_poincare} is primitive, then we have $\alpha_p \ge p\,\si_{\min}^{2-p} \lad/4$, which is asymptotically worse than the one in \eqref{est:lowerbeck2} for fixed $\si_{\min} \le 1/2$ and $p$ close to $1$ or fixed $p \in (1,2]$ and small enough $\si_{\min}$. 
\end{remark}

\begin{remark}
    When $p \to 1^+$, the lower bound \eqref{eq:beck_poin} reduces to the one in \cite[Theorem 3.3]{gao2021complete} for MLSI constant, which has been improved 
    very recently in \cite[Theorem 4.18]{gao2022complete} by Gao et al. with a logarithmic dependence on the complete version of 
    subalgebra index.  One may hence expect a similar improvement for $\alpha_p(\mc{L})$ as well, which we leave for future investigation. 
\end{remark}

}



\end{appendix}

\end{document}